\documentclass[10pt,leqno,twoside]{amsart}
\usepackage{tikz, subfigure, xcolor}
\usepackage{cite} 
\usepackage{dsfont} 
\usepackage{hyperref}
\usepackage{comment} 
\usepackage{inputenc}
\usepackage[normalem]{ulem}
\usepackage{doi}
\usepackage{mathtools}
\usepackage{mathrsfs}

\newtheorem{theorem}{Theorem}[section]
\newtheorem{lemma}[theorem]{Lemma}
\newtheorem{proposition}[theorem]{Proposition}

\theoremstyle{definition}
\newtheorem{definition}[theorem]{Definition}
\newtheorem{remark}[theorem]{Remark}

\newtheorem{example}[theorem]{Example}

\numberwithin{equation}{section}

\renewcommand{\div}{\mathrm{div}\,}			% div instead of %
\newcommand{\divh}{\mathrm{div}_{\!H}\,}	% horizontal divergence
\newcommand{\nablah}{\nabla_{\!H}\,} 		% horizontal gradient
\newcommand{\Deltah}{\Delta_{\!H}\,}  		% horizontal laplacian	
\newcommand{\Ah}{A_{\!H}}	  				% horizontal Stokes (removed \, due to ugly powers and domains)
\newcommand{\Apr}{A_{\rm{pr}}\,} 			% hydrostatic contribution to the pressure
\newcommand{\lso}{L^2_{\overline{\sigma}}} 	% horizontal solenoidal vector fields
\newcommand{\ls}{L^2_{\sigma}} 				% solenoidal vector fields
\newcommand{\Ps}{P_\sigma}					% solenoidal projection
\newcommand{\one}{{\mathds 1}}				% characteristic function
\newcommand{\hook}{\hookrightarrow}			% embedding
\newcommand{\harp}{\rightharpoonup}			% weak convergence

\providecommand{\norm}[1]{\left\| #1 \right\|} % Norm
\renewcommand{\d}{\operatorname{d}} 			%differential
\DeclareMathOperator{\dz}{\d\! z}				%dz
\DeclareMathOperator{\dxyz}{\d\!\rm{(x,y,z)}}				
\DeclareMathOperator{\dt}{\d\! t}				%dt
\DeclareMathOperator{\ds}{\d\! s}				%ds
\DeclareMathOperator{\dr}{\d\! r} 				%dr
\DeclareMathOperator{\dW}{\d\! W}				%dW
\newcommand{\E}{\mathbb{E}}

\newcommand{\R}{\mathbb{R}}

\newcommand{\N}{\mathbb{N}}
\newcommand{\PP}{\mathbb{P}}
\newcommand{\hPP}{\hat{\PP}}

\newcommand{\cU}{\mathcal{U}}
\newcommand{\cS}{\mathcal{S}}
\newcommand{\cX}{\mathcal{X}}
\newcommand{\cF}{\mathcal{F}}
\newcommand{\bF}{\mathbb{F}}
\newcommand{\Unk}{\hat{U}^{n_k}}
\newcommand{\Wnk}{\hat{W}^{n_k}}
\newcommand{\hU}{\hat{U}}
\newcommand{\hW}{\hat{W}}
\newcommand{\hOmega}{\tilde{\Omega}}
\newcommand{\Rc}{\mathcal{R}}
\newcommand{\Ac}{\mathcal{A}}

\title[Stochastic Primitive Equations with Horizontal Viscosity]{Stochastic 
Primitive Equations with Horizontal Viscosity and Diffusivity}

\subjclass[2020]{Primary: 35Q35;
Secondary: 35A01, 35K65, 35M10, 35Q86, 35R60, 60H15, 76D03, 86A05, 86A10.} 
\keywords{Primitive equations, horizontal viscosity, nonlinear stochastic PDE, multiplicative noise}

\author[Saal]{Martin Saal}
\address{Scuola Normale Superiore,
Pisa, Italy}
\email[Corresponding author]{msaal@mathematik.tu-darmstadt.de}
\curraddr{TU Darmstadt, Schlossgartenstr. 7, 64289 Darmstadt, Germany}

\author[Slav\'ik]{Jakub Slav\'ik}
\address{The Czech Academy of Sciences, Institute of Information Theory and Automation, Prague, Czech Republic}
\email{slavik@utia.cas.cz}

\thanks{The first author gratefully acknowledges the financial support of the Deutsche Forschungsgemeinschaft (DFG) through the research fellowship SA 3887/1-1.
}

\begin{document}

\begin{abstract}
We establish the existence and uniqueness of pathwise strong solutions to the stochastic 3D primitive equations with only horizontal viscosity and diffusivity driven by transport noise on a cylindrical domain $M=(-h,0) \times G$, $G\subset \R^2$ bounded and smooth, with the physical Dirichlet boundary conditions on the lateral part of the boundary. Compared to the deterministic case where the uniqueness of $z$-weak solutions holds in $L^2$, more regular initial data are necessary to establish uniqueness in the anisotropic space $H^1_z L^2_{xy}$ so that the existence of local pathwise solutions can be deduced from the Gy\"{o}ngy-Krylov theorem. Global existence is established using the logarithmic Sobolev embedding, the stochastic Gronwall lemma and an iterated stopping time argument.
\end{abstract}

\maketitle
\tableofcontents

%%%%%%%%%%%%%%%%%%%%%%%%%%%%%%%%%%%%%%%%%%%%%%%%%%%%%%%%%%%

\section{Introduction and main results}
The 3D primitive equations, one of the fundamental models for geophysical flows, describe oceanic and atmospheric dynamics. They are derived from the compressible 
Navier-Stokes equations assuming hydrostatic balance and the Boussinesq approximation.
The subject of this work is the initial value problem for the primitive 
equations with horizontal viscosity and diffusivity and the physical lateral Dirichlet boundary conditions. 

The study of problems with partial viscosity and diffusivity is motivated by the fact that in many geophysical models, the horizontal viscosity and diffusivity is considered dominant and the vertical one is neglected. Models with partial viscosity and diffusivity are also interesting from the analytical point 
of view since they combine 
features of both parabolic diffusion equations in horizontal directions (represented by the term 
$-\Deltah$) and hyperbolic transport equations in the vertical direction (represented by the nonlinear term 
$w \partial_z v$), see \eqref{eq:primeqhorvisc} below. Roughly speaking, one thus expects that 
regularity is preserved in the vertical direction while it is smoothed in horizontal directions. 
Following this intuition, we may identify classes of initial data for which the problem 
is locally or even globally well-posed.

\subsection{Primitive equations with horizontal viscosity and diffusivity}
Let $t > 0$ and let $G \subset \R^2$ be a bounded and smooth domain. We consider the primitive equations on a cylindrical domain $M$ of depth $h> 0$, defined by $M = G \times (-h, 0)$.
The boundary $\partial M$ decomposes into a lateral, upper and bottom parts as
\begin{align*}
\Gamma_l:= \partial G\times (-h,0), \quad \Gamma_u:= \overline{G} \times \{0\}, \quad  \Gamma_b:= \overline{G} \times \{-h\}.
\end{align*}
The primitive equations  
describe the velocity $u=(v,w)\colon M \rightarrow \R^3$, the temperature $T: M \to \R$ and the pressure $p\colon M \rightarrow\R$ 
of a fluid, where $v = (v_1, v_2)$ denotes the horizontal components and $w$ stands for the vertical component of velocity. The primitive equations with horizontal viscosity and diffusivity are
\begin{equation}
\begin{alignedat}{1} \label{eq:primeqhorvisc}
\partial_t v - \nu_v \Deltah v  + k\times v+ \frac{1}{\rho_0}\nablah p & \\
+ v \cdot \nablah v + w \partial_z v
& = f_v+\sigma_1(v,\nablah v,T,\nablah T) \dot W_1, \\
\partial_t T - \nu_T \Deltah T 
+ v \cdot \nablah T + w \partial_z T
& = f_T+\sigma_2(v,\nablah v,T,\nablah T) \dot W_2, \\
\partial_z p & =-\rho g, \\
  \divh v + \partial_z w & =0, \\
  \rho & =\rho_0(1-\beta_T(T-T_r)),
\end{alignedat}
\end{equation}
in $M \times (0,t)$ with
\begin{equation}\label{eq:primeqhorviscboundw}
w  =0 \quad \text {on }\Gamma_u\cup \Gamma_b \times (0,t).
\end{equation}
For the prognostic variables $v,T$, we have the initial conditions
\begin{equation} \label{eq:ic}
  v(t=0)=v_0, \qquad  T(t=0)=T_0,
\end{equation}
and the boundary conditions
\begin{align}\label{eq:bc}
v=0  \quad \hbox{and} \quad \partial_{n_G}T=0 \quad \hbox{on }\Gamma_l \times (0,t),
\end{align}
where $n_G$ is the outer normal to $G$ on $\Gamma_l$ (since $M$ is cylindrical 
$n_G$ does not depend on the vertical coordinate).
The first boundary condition in \eqref{eq:bc} is a lateral no-slip boundary condition for $v$, the second one is a 
Neumann-type condition for $T$. The condition \eqref{eq:primeqhorviscboundw} on $w$
is considered to be part of the system \eqref{eq:primeqhorvisc} since $w$ is a diagnostic variable, see \eqref{eq:w}.

In what follows, we will denote the variables of the horizontal domain by $(x, y) \in G$ and the vertical coordinate by $z \in (-h,0)$. We define 
$\nablah =(\partial_x, \partial_y)^T$, $\divh = \nablah^{\ast}$ and $\Deltah=\partial_x^2+\partial_y^2$ 
to be the horizontal gradient, divergence and Laplacian, respectively. Also let 
$v \cdot \nablah=v_1\partial_x+v_2\partial_y$ and let $k\times v=k_0(-v_2,v_1)$ be the Coriolis force. The terms 
$\sigma_i \dot W_i$ model the stochastic forces. The constants $\nu_v,\nu_T>0$ are the horizontal viscosity and horizontal diffusivity, $k_0\geq 0$ is the Coriolis parameter, $\rho_0,\beta_T,g>0$ denote the reference density, the expansion coefficient and the gravity, respectively.
Note that for the primitive equations the nonlinear term $w\partial_z v$ is of worse order 
compared to the nonlinearity of the Navier-Stokes equations since $w=w(v)$ given by \eqref{eq:w} 
also involves a first order derivative. However, the pressure $p$ in \eqref{eq:primeqhorvisc} is
 essentially two-dimensional up to a linear shift, see \eqref{eq:hydrostatic_presssure} below.

The general anisotropic primitive equations with full viscosity and diffusivity are obtained from  \eqref{eq:primeqhorvisc} if one replaces
the term $\nu_v \Deltah$ by $\nu_v\Deltah + \nu_z\partial_{zz}$ for horizontal viscosity $\nu_v > 0$ 
and vertical viscosity $ \nu_z > 0$ in the equation for $v$ and similarly in the equation for $T$.

\subsection{Previous results - deterministic}
The mathematical 
analysis of the initial value problem for primitive equations with full viscosity and diffusivity was started by Lions, 
Temam and Wang \cite{Lionsetal1992, Lionsetal1992_b, Lionsetal1993} and launched significant activity in the field.
In comparison to the 3D Navier-Stokes equations, 
the primitive equations are globally well-posed 
for initial data in $H^1(M)$ by a breakthrough result of Cao and Titi \cite{CaoTiti2007}, see also Kobelkov \cite{Ko07}. The more realistic case of Dirichlet boundary
conditions, also called physical boundary conditions in this context, in non-cylindrical domains were handled by Kukavica and Ziane in \cite{Ziane2007}. These results were refined to global well-posedness for initial data with 
$v_0,\partial_z v_0\in L^2$, see \cite{Ju2017}, or $v_0\in L^1(-h,0; L^\infty(G))$, see 
\cite{GigaGriesHusseinHieberKashiwabara2017NN}.

The primitive equations with only horizontal viscosity and diffusivity are of particular interest in the field of numerical weather prediction \cite{LauritzenJablonowskiTaylorNair}. On the one hand, the horizontal 
viscosity $\nu_v$ is much larger than the vertical viscosity $\nu_z$ in the atmosphere and the limiting case $\nu_z=0$ is considered a good approximation.
On the other hand, numerical (hyper-)viscosity acting only in the horizontal directions is often used in the computer simulations for the hydrostatic Euler 
equations.

Cao, Li and Titi \cite{CaoLiTiti2016, CaoLiTiti2017} were the first to study the primitive 
equations with only horizontal viscosity and diffusivity analytically. 
They tackled this problem in a periodic setting by considering a vanishing vertical viscosity limit, i.e. 
\begin{align*}
-\Deltah - \varepsilon
\partial_z^2 \quad \hbox{for}\quad  \varepsilon\to 0.
\end{align*}
Using this strategy, they obtained a remarkable global strong well-posedness result for the initial value 
problem with initial data of regularity near $H^1$, and local well-posedness for initial data in $H^1$.  
Recently, the first author applied a more direct approach considering the system without 
vanishing viscosity limit \cite{Saal} where local well-posedness results even for less partial viscosities has 
been proven and, in the case of only horizontal viscosity, unnecessary boundary conditions on the bottom and top have 
been avoided. 
The construction of weak solutions, namely the compactness argument, is difficult in this case due to the lack of dissipation in the vertical direction. In \cite{HusseinSaalWrona}, the existence of $z$-weak solutions, i.e.\ weak solutions with additional 
regularity in the vertical direction,
is shown for initial data with $v_0,\partial_z v_0\in L^2(M)$. Additionally, 
the global existence of solutions is shown for less regular data than in \cite{Saal} and
time-periodic solutions are constructed there. 
For more results and further references 
on the deterministic primitive equations with horizontal and full viscosity, we refer the reader to the surveys  \cite{LiTiti2016} and \cite{HieberHussein2020}.

Results on local well-posedness of the Navier-Stokes equations with only horizontal viscosity can be found in \cite[Chapter 6]{Danchin}. 

The primitive equations without any viscosity and the equation for the temperature are called the hydrostatic Euler equations. One important reason to consider this system is 
the understanding of the behaviour of numerical weather and climate models 
because the time scale associated with viscous dissipation is beyond the current
computational capability \cite[Chapter 2]{LauritzenJablonowskiTaylorNair}. The primitive equations with only horizontal viscosity are closer to this situation than the ones with full viscosities.    

For the hydrostatic Euler equations blow-up results were established by Wong 
\cite{Wong2015}, see also \cite{Caoetal2015}, and there are ill-posedness results in Sobolev spaces 
by Han-Kwan and Nguyen \cite{HanNguyen2016}. Local well-posedness was proven only for analytical 
data by Kukavica et al.\ \cite{Kukavica2011}. 

For more information on previous results and the geophysical applications of the primitive
equations, we refer to the works of Washington and
Parkinson \cite{WashingtonParkinson1986}, Pedlosky
\cite{Pedlosky1987}, Majda \cite{Majda2003} and Vallis
\cite{Vallis2006}.

\subsection*{Previous results - stochastic}
There are several reasons to study stochastic versions of the primitive equations. 
The introduction of additive and multiplicative noise (often called state-dependent noise in the geophysical literature)
into models for geophysical flows can  be used on the one hand	to account for numerical
and empirical uncertainties. On the other hand, let us emphasize the role of stochastic modelling in meteorology and climatology. Current
models in these fields consist of a deterministic dynamical
core based on equations of continuum mechanics and thermodynamics,
complemented by stochastic elements at several levels: random initial conditions,
reflecting partial knowledge of the initial state, and random inputs
distributed in space-time, related for instance to sub-grid stochastic
parametrizations. Predictions are probabilistic in the sense that they aim to produce a range
of scenarios with associated probabilities, usually by the method called
ensemble forecasting system. For more information on this 
field, see the review articles \cite{Franzke14, Palmer19} and the references 
therein.
Of particular interest in fluid dynamics and also for geophysical flows is a noise of transport type which appears naturally when stochastic models are derived from Hamiltonian principles as 
proposed in \cite{Holm2015} (see also \cite{BiFla20} for a brief description)  and yield 
a physically relevant randomization \cite{Bauer} with energy conservation. Recently, the importance of transport noise was discussed in the connection with unresolved small scales, see \cite{FlandoliPappalettera2020,FlandoliPappalettera2021} and the references therein.

Stochastic primitive equations with full viscosities were studied by 
several authors. In two space dimensions, i.e.\ neglecting one of the horizontal directions, the so-called 
weak-strong solutions for multiplicative white noise in time were constructed by a Galerkin approach 
in \cite{GlattholtzZiane2008} by Glatt-Holtz and Ziane, where the weak-strong solutions are weak in the PDE sense and strong in the stochastic sense. In particular, these 
solutions can be interpreted in a pathwise sense. For initial data $v_0$ with $v_0,\partial_z v_0\in 
L^p(\Omega,L^2(M))$ where $\Omega$ is the probability space and $M$ the spatial domain, the authors show the 
existence and uniqueness of such solutions for $p\geq 4$, so these are $z$-weak solutions. They make use of a special cancellation related to the assumed periodic boundary conditions. The long term behaviour of weak solutions of 
the stochastic two-dimensional primitive equations is studied in \cite{Medjo2010}. The existence of global pathwise strong solutions is shown in 
\cite{GlattholtzTemam2011} for initial data in $v_0\in L^2(\Omega,H^1(M))$ also by means of Galerkin approximation. The authors use continuous martingale theory and stopping time arguments to treat 
the primitive equations with physical boundary conditions without the above-mentioned cancellation by establishing stronger convergence of
the Galerkin approximations. Furthermore, a large deviation principle\cite{GaoSun2012} and a central limit 
theorem \cite{GuoZhangZhou2018} are known to hold.

Regarding the 3D problem, Debussche, Glatt-Holtz, Temam and Ziane established 
a global well-posedness result for pathwise strong solutions 
for multiplicative white noise in time in \cite{DebusscheGlattholtzTemamZiane2012} 
and the related 
work \cite{Debussche2011} by a different method than in the 2D case.
They first show the existence of 
martingale solutions and pathwise uniqueness, which then leads to the 
existence of pathwise solutions by a Yamada-Watanabe type argument. The assumptions on the noise for global existence in \cite{DebusscheGlattholtzTemamZiane2012} require certain smoothing properties that fail for transport noise. The second author with Brze\'{z}niak \cite{Brzezniak2020} established global existence for noise allowing transport by the vertical average of the horizontal velocity $\overline{v}$.
In the case of additive noise, the existence 
of a random pull-back attractor is known \cite{GuoHuang2009}. 
Logarithmic moment bounds in $H^2(M)$ are obtained 
in \cite{GlattholtzKukavicaVicolZiane2014} and used to prove the existence of ergodic invariant measures supported 
in $H^1(M)$. A construction of weak-martingale solutions, i.e.\ martingale 
solutions whose regularity in space and time is the one of a weak solution, by an implicit Euler scheme is given in 
\cite{GlattholtzTemamWang2017}. Large deviation principles \cite{DongZhaiZhang2017} and moderate deviation principles \cite{Slavik2020} are 
known to hold for small multiplicative noise and short times \cite{DongZhang2018}. The existence 
of a Markov selection is proven in \cite{DongZhang2017} for additive noise.

All these above results are shown for the primitive equations with vertical (i.e.\ full) viscosity and diffusivity. To the best of our knowledge, there are no results on the stochastic system with only horizontal viscosity and diffusivity or stochastic hydrostatic Euler equations.

\subsection{Main results}

The notation used below is established in Section \ref{sec:pre}. The following theorems will be proved in Sections \ref{sect:maximalSolutions_proof} and \ref{sect:globalSolutions_proof}, respectively.

\begin{theorem}[Maximal existence]
\label{thm:maximalExistence}
Let $(\Omega, \cF, \bF, \PP)$ be a stochastic basis with filtration $\bF = \lbrace \cF_t \rbrace_{t \geq 0}$ satisfying the usual conditions. Let $\cU$ be a separable Hilbert space and let $\sigma$ satisfy \eqref{eq:sigmaGrowthL2}-\eqref{eq:sigma2boundary}. Let $(f_v, f_T)^T \in L^4(\Omega; L^2(0, t;  H^2_{N,z} L^2_{xy}\times H^2_{D,z} L^2_{xy}))$. Then for all initial data $U_0=(v_0,T_0) \in L^2(\Omega;  (L^2_zH^1_{D,xy}\times  L^2_zH^1_{xy}) \cap (H^2_{N,z} L^2_{xy}\times H^2_{D,z} L^2_{xy}))$ there exists a unique maximal pathwise solution $(U, \tau)$ of \eqref{eq:primeqhorvisc}-\eqref{eq:bc} in the sense of Definition \ref{def:pathwise_solution}. Moreover, the solution satisfies
\begin{equation}
	\label{eq:max_thm_additional_reg}
	\E \left[ \sup_{s \in [0, t \wedge \tau_N]} \| U \|_{H^2_z L^2_{xy}}^2 + \int_{0}^{t \wedge \tau_N} \| U \|_{H^2_z H^1_{xy}}^2 \ds \right] < \infty
\end{equation}
for all $N \in \N$ and $t>0$ and $U(s) \in H^2_{N,z} L^2_{xy}\times H^2_{D,z} L^2_{xy}$ for all $0<s<\tau$ $\PP$-almost surely.
\end{theorem}

The proof of Theorem \ref{thm:maximalExistence} follows the method from \cite{Debussche2011}. First, the global existence of martingale solutions is established for a modified problem with a cut-off. Compared to \cite{Debussche2011}, where the cut-off acts on the $H^1$-norm of the solution, we need to use a weaker cut-off acting on $L^\infty_z L^4_{xy}$-norm of the solution due to the lack of vertical smoothing. After a standard argument using the theorems by Prokhorov and Skorokhod (relying on compactness in the anisotropic spaces), the local existence of strong (in the stochastic sense) solutions is established by the Gy\"{o}ngy-Krylov theorem. However, to establish strong uniqueness in the required spaces, we need to assume initial data more regular in the vertical direction. The additional regularity of the initial data then leads to the additional regularity of the solution \eqref{eq:max_thm_additional_reg}.
The improved regularity \eqref{eq:max_thm_additional_reg} requires additional boundary conditions for $v$ and $T$ in the vertical direction which are not necessary in the deterministic case studied in \cite{HusseinSaalWrona}. There, the uniqueness of $z$-weak solutions was established in the space $L^2$. Similarly as in \cite{CaoLiTiti2016}, we chose the homogeneous Neumann boundary condition for the velocity field and the homogeneous Dirichlet boundary conditions for the temperature. These boundary conditions are preserved by the primitive equations, see Remark \ref{rem:bc} for more details.

\begin{theorem}[Global existence]
\label{thm:globalExistence}
Let $(\Omega, \cF, \bF, \PP)$, $\cU$, $U_0$ and $\sigma_1$ be as in Theorem \ref{thm:maximalExistence}. Additionally, let $\sigma$ satisfy \eqref{eq:noiseglobal}-\eqref{eq:noise_eta} and let
\[
	U_0=(v_0,T_0) \in L^{16/3}(\Omega; H^1) \cap L^{8/3}(\Omega; L^{132}), \quad \partial_z v_0 \in L^6(\Omega; L^6),
\]
and
\[
(f_v, f_T)^T \in L^4(\Omega; L^2_{\mathrm{loc}}(0, \infty;  H^2_{N,z} L^2_{xy}\cap L^2_z H^1_{D,xy})) \cap L^{8/3}(\Omega; L^2_{\mathrm{loc}}(0, \infty; L^{132})).
\]
Then there exists a unique global pathwise solution $(U, \xi, (\tau_n))$ of \eqref{eq:primeqhorvisc}-\eqref{eq:bc} in the sense of Definition \ref{def:pathwise_solution}. Moreover, the solution satisfies \eqref{eq:max_thm_additional_reg} for all $N \in \N$ and $t>0$.
\end{theorem}

Theorem \ref{thm:globalExistence} is established combining the (deterministic) estimates, the logarithmic Sobolev inequality and logarithmic Gronwall lemma from \cite{CaoLiTiti2017} and an iterated stopping time argument from \cite{DebusscheGlattholtzTemamZiane2012,Brzezniak2020}. In \cite{CaoLiTiti2017}, one of the key steps is an estimate on the asymptotic behaviour of $\| v \|_{L^q}/q^{1/2}$ w.r.t.\ $q \to \infty$ which later leads to a bound on $v$ in $L^\infty$ and $\partial_z v$ in $L^2$. However, it seems impossible to get a similar asymptotic bound in the stochastic setting with noise acting on $\nablah v$. The asymptotic estimate is thus substituted by bounds in $L^{132}$. Using an argument similar to the logarithmic Gronwall lemma, this is sufficient to establish $L^2$-integrability in time of the $L^\infty$-norm which then leads to the desired bound on $\partial_z v$.

Compared to previous results for the primitive equations with full viscosity and diffusivity \cite{DebusscheGlattholtzTemamZiane2012,Brzezniak2020}, our approach allows for full transport noise.
The key difference to \cite{DebusscheGlattholtzTemamZiane2012} are the pressure estimates in $L^q(M)$ necessary to deduce $L^q(M)$-estimates for the velocity field.
In \cite{DebusscheGlattholtzTemamZiane2012}, the authors consider the
corresponding Stokes problem with the noise term and then prove estimates for
the difference of the solution of the full non-linear problem and the solution of the
Stokes problem. The difference solves a random PDE and standard
analytic tools can be used to estimate the pressure term. The disadvantage of this
approach is that it requires the solution of the Stokes problem to be rather smooth
and transport noise cannot be included for this reason. We follow the approach of \cite{Brzezniak2020} where the conditions on the noise allow to obtain a random PDE for the pressure by using a hydrostatic Leray-Helmholtz projection. Using the linear structure of the transport part of the noise, we can go beyond the results from \cite{Brzezniak2020} and consider transport noise acting not only on the vertical average of the velocity $\overline v$ but also on the remainder $\widetilde v = v - \overline{v}$.

\subsection{Organization of the paper}

In Section \ref{sec:pre}, we define the function spaces used throughout the rest of the paper and reformulate the primitive equations as an abstract functional problem. In particular, the discussion on the assumptions on the noise term $\sigma$ and the definition of solution can be found in Sections \ref{sect:noise_assumptions} and \ref{sect:solutions}, respectively. In Section \ref{sect:maximal_solutions_main}, we prove the maximal existence theorem above. After defining the Galerkin approximations of a modified problem in Section \ref{sect:galerkin} and establishing bounds on the approximations in Section \ref{sect:gal_estimates}, we prove the tightness of the corresponding measures and the global existence of martingale solutions of the modified problem in Section \ref{sect:compactness}. Uniqueness is established in Section \ref{sect:uniqueness} for more regular initial data due to the use of the Gy\"{o}ngy-Krylov theorem. Finally, Section \ref{sec:globalexistence} contains the proof of the global existence of strong solutions using estimates on the barotropic and baroclinic modes of the velocity, logarithmic Sobolev embedding and logarithmic Gronwall inequality.

\section{Preliminaries}  \label{sec:pre}
\subsection{Function spaces and notations}\label{subsec:functionspaces} 

By $L^2(M)$, we denote the standard real Lebesgue space with scalar product
\begin{align*}
\left< f,g \right>_{M}:=\int_{M}f(x,y,z)g(x,y,z)\d(x,y,z),
\end{align*}
with $L^2(G)$ and $\left< f,g \right>_{G}$ defined analogously. We denote the induced norms by $\norm{f}_{L^2(M)}$ and $\norm{f}_{L^2(G)}$, respectively. If there is no ambiguity, we will
not specify the domain in the notation and write e.g.\ $<f, g>$ instead of $<f, g>_M$.

For $k\in\N$, the space $H^k(M)$ consists of $f\in L^2(M)$ such that 
$\nabla^\alpha f \in L^2(M)$ for all multi-indices $|\alpha|\leq k$ endowed with the norm
\begin{align*}
\norm{f}_{H^k(M)}=\sum_{|\alpha|\leq k} \norm{\nabla^\alpha f}_{L^2(M)},
\end{align*}
where
$\nabla^\alpha=\partial_x^{\alpha_1} 
\partial_y^{\alpha_2}\partial_z^{\alpha_3}$ 
for $\alpha\in\N_0^3$. The space $H^k(G)$ is defined analogously and we set $H_0^1(G):=\{f\in H^1(G) \mid f\vert_{\partial G} =0 \}$. Again, we will write
$\norm{f}_{H^k}$ if there is no 
ambiguity.

For non-integer $s\geq 0$ the Bessel potential spaces $H^s$
are defined by complex interpolation, see \cite{Triebel} for details.
Analogous definitions 
hold for spaces $H^{s,p}$ for $p\in [1,\infty]$ and for Banach space-valued function spaces such as the Bochner space $H^1(0,t; L^p(M))$ 
and the anisotropic space $H^{s,p}(-h,0; H^{r,q}(G))$ for which we will often use the notation $H^{s,p}_zH^{r,q}_{xy}$. We identify $H^s=H^{s,2}$.
Furthermore, for $s \geq 1$, we write
\begin{align}
	\label{eq:spacebddirichlet}
	L^p_zH^s_{D,xy} &:= L^2(-h,0;H^s(G)\cap H^1_0(G)),\\
	\label{eq:spacebdneumannhoriz}
	L^2_{z}H^2_{N,xy} &:= \{f\in L^2_zH^2_{xy} \mid \partial_{n_G}T=0 \text{ on } \Gamma_l\}
\end{align}
and
\begin{align}
	\label{eq:spacebddirichletvert}
	H^2_{D,z}L^q_{xy} & := \{f\in H^2(-h,0;L^q(G)):  f=0 \text{ on } \Gamma_b\cup \Gamma_u \},\\
	\label{eq:spacebdneumannvert}
	H^2_{N,z}L^q_{xy} & := \{f\in H^2(-h,0;L^q(G)): \partial_z f=0  \text{ on } \Gamma_b\cup \Gamma_u\}.
\end{align}
for the spaces encoding the boundary condition on the lateral and the vertical boundary.
The boundary condition $\partial_{n_G}T=0$ on 
$\Gamma_l$ has to be understood as a condition for the trace of $\nablah T$ on $\Gamma_l$ for 
almost every $z\in (-h,0)$.

For the sake of completeness, let us recall the definition of fractional Sobolev spaces $W^{\alpha, p}(0, t; X)$ 
%and the Nikolskii spaces $N^{\alpha, p}(0, t; X)$ 
for some Banach space $X$ and $\alpha\in (0,1]$.  
%we define
%\begin{align*}
%	W^{\alpha, p}(0, t; X) &= \left\lbrace u \in L^p(0, t; X) \mid \norm{U}_{W^{\alpha, p}(0, t; X)} < \infty \right\rbrace,\\
%	\norm{U}_{W^{\alpha, p}(0, t; X)}^p &= \int_0^t \norm{U}_X^p \ds + \int_0^t \int_0^t \frac{\| u(s) - u(r) \|_X^p}{|s - r|^{1 + \alpha p}} \ds \d\!r.
%\end{align*}
For $\alpha=1$, let
\begin{align*}
	W^{1, p}(0, t; X) &= \left\lbrace u \in L^p(0, t; X) \mid \tfrac{\d}{\dt} u \in L^p(0, t; X) \right\rbrace,\\
	\norm{U}_{W^{1, p}(0, t; X)}^p &= \int_0^t \norm{U}_X^p + \left\| \tfrac{\d}{\dt} u \right\|_X^p \ds.
\end{align*}
For $\alpha \in (0, 1)$ and $p \in (1, \infty)$, the fractional Sobolev spaces $W^{\alpha, p}(0, t; X)$ are defined by real interpolation.
%Finally for $\alpha \in (0, 1)$ and $p \in [1, \infty]$ we define the Nikolskii spaces by
%\begin{align*}
%	N^{\alpha, p}(0, t; X) &= \left\lbrace u \in L^p(0, t; X) \mid \norm{U}_{N^{\alpha, p}(0, t; X)} < \infty \right\rbrace,\\
%	\norm{U}_{N^{\alpha, p}(0, t; X)} &= \sup_{\lambda > 0} \lambda^{-\alpha} \| u(\cdot + \lambda) - u \|_{L^p(0, t-\lambda; X)}.
%\end{align*}
By \cite{Simon1990}, we have $H^{\alpha, p}(0, t; X) \subset W^{\alpha, p}(0, t; X)$ %$\subset N^{\alpha, p}(0, t; X)$ 
for $\alpha \in (0, 1)$ and $p \in (1, \infty)$. We can now state the following compactness result. For proofs see \cite[Theorem 5]{Simon1987} and \cite[Theorem 2.1]{FlandoliGatarek}, respectively.

\begin{lemma} 
\label{lem:BochnerCompactness}
a) \emph{(the Aubin-Lions-Simon lemma)} Let $X_2 \subset X \subset X_1$ be Banach spaces such that the embedding $X_2 \hook \hook X$ is compact and the embedding $X \hook X_1$ is continuous. Let $p \in [1, \infty]$ and $\alpha \in (0, 1]$. Then the following embedding is compact:
\[
	L^p(0, t; X_2) \cap W^{\alpha, p}(0, t; X_1) \hook \hook L^p(0, t; X).
\]

b) Let $X_2 \subset X$ be Banach spaces such that $X_2$ is reflexive and the embedding $X_2 \hook \hook X$ is compact. Let $\alpha \in (0, 1)$ and $p \in (1, \infty)$ be such that $\alpha p > 1$. Then the following embedding is compact:
\[
	W^{\alpha, p}(0, t; X_2) \hook \hook C([0, t], X).	
\]
\end{lemma}

Let us now summarize some of the properties of the anisotropic spaces on a cylindric domain.

\begin{proposition}
	The anisotropic spaces have the following properties:
	\begin{enumerate}
		\item Embeddings in the spaces $H^{s,p}_zH^{r,q}_{xy}$ can be performed separately:
		\begin{align*}
			H^{s,p}_zH^{r,q}_{xy} &\hook H^{s',p'}_zH^{r,q}_{xy} & &\text{if} & H^{s,p}((-h, 0)) &\hook H^{s',p'}((-h, 0)),\\
			H^{s,p}_zH^{r,q}_{xy} &\hook H^{s,p}_zH^{r',q'}_{xy} & &\text{if} & H^{r,q}(G) &\hook H^{r',q'}(G).
		\end{align*}
		\item The following anisotropic H\"{o}lder inequality holds: Let $p,p_1,p_2,q,q_1,q_2 \in [1, \infty]$ with $\frac{1}{p_1}+\frac{1}{p_2}=\frac1p, 
		\frac{1}{q_1}+\frac{1}{q_2}=\frac1q$ and $f\in L^{p_1}_zL^{q_1}_{xy}, g\in L^{p_2}_zL^{q_2}_{xy}$.
		 Then $fg\in L^p_zL^q_{xy}$ and
		\begin{align}
		\label{eq:anisotropichoelder} 
		\norm{fg}_{L^p_zL^q_{xy}}\leq \norm{f}_{L^{p_1}_zL^{q_1}_{xy}} \norm{g}_{L^{p_2}_zL^{q_2}_{xy}}.
		\end{align}
		\item Let $v \in H^1_z L^p_{xy}$. Then $v \in L^\infty_z L^p_{xy}$ and
		\begin{equation}
			\label{eq:gginterpolation}
			\norm{v}_{L^\infty_z L^p_{xy}} \leq c \norm{v}_{L^2_z L^p_{xy}}^{1/2} \norm{v}_{H^1_z L^p_{xy}}^{1/2}.
		\end{equation}
		\item Let $v \in H^1$. Then $v \in L^\infty_z L^4_{xy}$ and
		\begin{equation}
			\label{eq:cltinterpolation}
			\norm{v}_{L^\infty_z L^4_{xy}} \leq c \norm{v}_{H^1_z L^2_{xy}}^{1/2} \norm{v}_{L^2_z H^1_{xy}}^{1/2}.
		\end{equation}
		\item Let $\alpha,\beta > 1/2$. The following embeddings are compact:
		\[
			L^2_z H^2_{xy} \cap H^1_z H^1_{xy} \hook \hook L^2_z H^1_{xy}, \qquad H^{\beta}_z H^{\alpha}_{xy} \hook \hook L^\infty_z L^4_{xy}.
		\]
	\end{enumerate}
\end{proposition}

\begin{proof}
Checking the embedding properties of the anisotropic spaces and the aniso\-tropic H\"{o}lder inequality is straightforward. The interpolation inequality \eqref{eq:gginterpolation} has been established in \cite[Lemma 3.3(a)]{GuillenGonzalez2001} for $p=2$. The more general case follows in the same manner. The interpolation inequality \eqref{eq:cltinterpolation} has been proven in \cite[Lemma 2.3]{CaoLiTiti2017}.

It remains to establish the compact embeddings. The first follows directly from the Aubin-Lions lemma, see Lemma \ref{lem:BochnerCompactness} a). The second embedding is a consequence of the Aubin-Lions lemma, the continuous embedding $H^{\beta}_z H^{\alpha}_{xy} \hook L^{\infty}_z H^\alpha_{xy}\cap H^{\gamma,\infty}_z L^4_{xy}$ holding for some $\gamma > 0$ sufficiently small and the compact embedding $H^\alpha(G) \hook \hook L^4(G)$.
\end{proof}

\subsection{Hydrostatic-solenoidal vector fields}\label{subsec:lso}
We reformulate the primitive equations \eqref{eq:primeqhorvisc} and \eqref{eq:primeqhorviscboundw}
as a system containing only two-dimensional surface pressure $p_s$ and  
prognostic variables $v$ and $T$. The divergence free 
condition $\partial_z w+\divh v=0$ and the 
boundary condition $w=0$ on  $\Gamma_b\cup \Gamma_u$ are equivalent to
\begin{gather}
	\label{eq:w}
	w(v)(t,x,y,z) = w(t,x,y,z) =-\divh \int_{-h}^{z}v(t,x,y,z') \d\! z',\\
	\label{eq:meanDivFree}
	\divh \int_{-h}^0 v(t,x,y,z')\d\!z'=0,
\end{gather}
for $v$ sufficiently smooth, e.g.\ with $\divh v \in L^1(M)$. The vertical velocity $w$ is thus a function of the horizontal velocity $v$. If we denote the vertical average and its complement by $\overline{v}$ and $\tilde{v}$, respectively, i.e.\
\begin{align}\label{eq:defvbar} %\label{eq:defvtilde}
\overline v(t,x,y):=\frac{1}{h}\int_{-h}^0 v(t,x,y,z')  \d\! z' 
\quad \hbox{and} \quad \tilde v:=v-\overline v,
\end{align}
then \eqref{eq:meanDivFree} implies that the vertical average $\overline{v}$ is divergence free, i.e.\ $\divh \overline{v}=0$.
To ease the notation from the typographical point of view, we sometimes write $\Ac v$ and $\Rc v$ instead of $\overline{v}$ and $\widetilde{v}$, respectively. 
Hence, one identifies a suitable hydrostatic-solenoidal space as
\begin{align*}
\lso(M):= \overline{\{v\in C_c^{\infty}(M)^2 \mid \divh \overline{v}=0  \}}^{||\cdot||_{L^2}},
\end{align*}
where $C_c^{\infty}$ stands for smooth compactly supported functions. Note that this space admits the 
decomposition
\begin{align}\label{eq:lso_split}
\lso(M) = \ls(G) \oplus \{v\in L^2(M)^2 \mid \overline{v}=0\},
\end{align}
where $\ls(G)= \overline{\{\varphi \in C_c^{\infty}(G)^2\colon \divh \varphi=0  \}}^{||\cdot||_{L^2}}$ is the space of solenoidal vector fields over $G$. The hydrostatic Helmholtz projection thereon is therefore defined by
\begin{align}
\Ps\colon L^2(M)^2 \rightarrow \lso(M), \quad \Ps v = \tilde{v} + P_G \overline{v},
\end{align}
where $P_G$ the classical Leray-Helmholtz projection on $L^2(G)$. 
More precisely, since due to the product structure $L^2(M)=\overline{L^2(G) \otimes L^2(-h,0)}$, one obtains 
by applying $\Ps$ that
\begin{align}\label{eq:lso_split2}
\lso(M) = \overline{\ls(G) \otimes \operatorname{span}\{1\}} \oplus \overline{L^2(G)^2 \otimes L_0^2(-h,0)},
\end{align}
where $L^2_0(-h,0)=\{v\in L^2(-h,0) \mid \int_{-h}^{0} v(z) \dz=0\}$ and $1\in L^2(-h,0)$ is a constant function. 

We now may reformulate the system (\ref{eq:primeqhorvisc},\ref{eq:primeqhorviscboundw}) as 
(see \cite[Section 2.1]{PetcuTemamZiane2009} for more details)
\begin{equation}
\begin{alignedat}{1} \label{eq:primeqreformulated}
\partial_t v - \nu_v \Deltah v + k\times v\\
+ \frac{1}{\rho_0}\nablah p_s -\beta_Tg\int_{z}^{0} \nablah T(x,y,z')\d\! z'\\
+ v \cdot \nablah v -w(v) \partial_z v
& = f_v+\sigma_1(v,\nablah v,T,\nablah T) \dot W_1,\\
\partial_t T - \nu_T \Deltah T 
+ v \cdot \nablah T + w(v) \partial_z T
& = f_T+\sigma_2(v,\nablah v,T,\nablah T) \dot W_2
\end{alignedat}
\end{equation}
in $M  \times (0,t)$ and
\begin{equation}\label{eq:primeqreformulateddivfree}
v(s)\in  \lso(M) \quad \text{for }  s \in (0,t)
\end{equation}
complemented by the initial conditions \eqref{eq:ic} and boundary conditions \eqref{eq:bc}. The above system is closed 
and we can reconstruct the full pressure $p$ from the surface 
pressure $p_s$ by
\begin{equation}
\label{eq:hydrostatic_presssure}
p(x,y,z)=p_s(x,y)+g \int_{z}^0 \rho(x,y,z') \d\! z',
\end{equation}
where $\rho=\rho_0(1-\beta_T(T-T_r))$ as in \eqref{eq:primeqhorvisc}. 
We emphasize that $p_s$ is independent of the vertical variable $z$.

\begin{remark}\label{rem:bc}
To show the existence of strong solutions, we assume $T_0=0$ on $\Gamma_b \cup 
\Gamma_u$. The homogeneous Dirichlet boundary condition can be relaxed to $T_0=c_b$ on $\Gamma_b$ and 
$T_0=c_u$ on $\Gamma_u$ for constants $c_b,c_u$. Applying the linear transformation
$\hat T=T-c_u \frac{z+h}{h}+c_b\frac{z}{h}$, one immediately observes that $\hat T_0=0$ 
on $\Gamma_b \cup \Gamma_u$ and that $(v,\hat T)$ satisfies 
\eqref{eq:primeqreformulated} with an additional term $w(v)\frac{c_u-c_b}{h}$ and 
$\sigma_2(v,\nablah v,\hat T,\nablah \hat T)$ replaced by $\sigma_2(v,\nablah v,
\hat T+c_u \frac{z+h}{h}-c_b\frac{z}{h},\nablah \hat T)$. The additional 
deterministic term can be handled in the same way as $-\beta_Tg\int_{z}^{0} \nablah T(x,y,z')\d\! z'$ 
in the equation for $v$ and it vanishes on $\Gamma_b\cup\Gamma_u$. Since it does not cause any additional difficulties, we omit it.
\end{remark}

\subsection{Functional formulation} \label{sec:functional} Let us now formulate the original 
equations \eqref{eq:primeqhorvisc}-\eqref{eq:bc} in an abstract functional 
form. We also provide estimates on the nonlinear term demonstrating the 
importance of the anisotropic spaces defined in the previous section.

We call the operator $-P_\sigma \Deltah$ the \emph{hydrostatic Stokes operator}. Let
\[
D(\Ah) = \left( L^2_zH^2_{D,xy} \cap \lso \right) \times L^2_zH^2_{N,xy}.
\]
For $U \in D(\Ah)$, we define the operator $\Ah$ by
\[
	\Ah\, U = \begin{pmatrix}
		-\Ps \Deltah v\\
		-\Deltah T
	\end{pmatrix}.
\]

The representation of $\Ps$ given in the previous subsection and 
\eqref{eq:lso_split2} allow us to introduce an orthonormal system 
of eigenfunctions of $\Ah$ 
using the known orthonormal systems associated with the 2D Stokes operator the 3D horizontal 
Laplacian. To this end, let $(\varphi_m)_m \subset C^{\infty}(\overline{G})^2
\cap H^1_0(G)^2$ and \ $(\hat{\varphi}_m)_m 
\subset C^{\infty}(\overline{G})^2\cap H^1_0(G)^2\cap \ls(G)$ be orthonormal bases of eigenfunctions of the Dirichlet Laplacian in $L^2(G)^2$ and the Stokes operator 
in $\ls (G)$, respectively, with corresponding increasing 
sequences of eigenvalues 
$(\mu_m)_m$ and  $(\hat{\mu}_m)_m$.
Moreover, recall that $\cos\left(\frac{k\pi}{h}(z+h)\right)$ for $k\in \N_0$ form a basis of eigenfunctions of the Neumann Laplacian on $L^2(-h,0)$ and, by the 
first representation theorem, a basis on $H^1(-h,0)$ and of $H^2_N(-h,0):=\{v\in H^2(-h,0)\mid \partial_z v(-h)=\partial_z v(0)=0\}$ as well,.

For $m\in \N$ and $k\in\N_0$, we define the functions $\Phi_{m,k}\in 
C^{\infty}(\overline{M})$ by
\begin{align}\label{eq:Phi}
\Phi_{m,k}(x,y,z):=\begin{cases}
  	\frac{2}{h}\varphi_m(x,y)
 	\cos\left(\frac{k\pi}{h}(z+h)\right)
 	& \hbox{for} \quad k>0,\\
  	\frac{1}{h}\hat{\varphi}_m(x,y) & \hbox{for} \quad k=0.
\end{cases}
\end{align}
Then $\operatorname{span} \{\Phi_{m,k} | m\in\N,k\in\N_0\}$ is dense 
in $\lso$, in particular $\divh\overline{\Phi}_{m,k}=0$, 
because for $k>0$ we have  $\overline{\Phi}_{m,k}=0$ and $\divh\overline{\Phi}_{m0}=
\divh \hat{\varphi}_m=0$. 

For the temperature, we use the functions $\sin\left(\frac{k\pi}{h}(z+h)\right)$ for $k\in \N$,
they form a basis of eigenfunctions of the Dirichlet Laplacian on $L^2(-h,0)$ and 
a basis on $H^1_0(-h,0)$ and of $H^2_D(-h,0):=H^2(-h,0)\cap H^1_0(-h,0)$. Let $(\psi_m)_m \subset C^{\infty}(\overline{G})^2$ be an
orthonormal basis of eigenfunctions of 
the Neumann Laplacian in $L^2(G)^2$ with an increasing 
sequence of eigenvalues $(\lambda_m)_m$ and define for $m,k\in \N$ 
the functions $\Psi_{m,k}\in C^{\infty}(\overline{M})$ by
\begin{align} \label{eq:Psi}
 \Psi_{m,k}(x,y,z):=
  	\frac{2}{h}\psi_m(x,y)
 	\sin\left(\frac{k\pi}{h}(z+h)\right).
\end{align}
Then $\operatorname{span} \{\Psi_{m,k} | m,k\in\N\}$ is dense 
in $L^2$. Hence, $\operatorname{span} \{\Phi_{m,k} | m\in\N,k\in\N_0\} 
\times \operatorname{span} \{\Psi_{m,k} | m,k\in\N\}$ is dense 
in $\lso\times L^2(M)$ and from the construction we observe that it is also
dense in $D(\Ah)$. Additionally,
\begin{equation}
\label{eq:eigenfunctions}
-\Deltah \Psi_{m,k}=\lambda_m \Psi_{m,k} \quad \text{and} \quad -\Ps \Deltah \Phi_{m,k}=\begin{cases} \mu_m \Phi_{m,k} & k>0,\\ \tilde{\mu}_m \Phi_{m,k}, & 
k = 0,\end{cases}
\end{equation}
therefore the sequences $(\Psi_{m, k})_{m, k}$ and $(\Phi_{m, k})_{m, k}$ are indeed eigenfunctions of $-\Deltah$ and $-P_\sigma \Deltah$, respectively.
Note, that the eigenspaces of all the eigenvalues have infinite dimension.

\begin{proposition}\label{prop:poincare}
The operator $\Ah$ is a self-adjoint unbounded operator.
In particular, $D(A_{\!H}) \subseteq L^2_{z} H^2_{xy}$, the domains of fractional powers $D(A_{\!H}^\alpha)$ are dense in 
$\lso \times L^2$ for $\alpha \geq 0$. Moreover, let
\[
	\tilde P_n: \lso \times L^2 \to \operatorname{span}\{\Phi_{m,k} \mid m\leq n, k\in\N_0 \}\times \operatorname{span}
\{\Psi_{m,k} \mid m\leq n, k\in\N\}
\]
where $\Phi_{m,k}$ 
and $\Psi_{m,k}$ are defined in \eqref{eq:Phi} and \eqref{eq:Psi}, be the orthogonal projection onto its range, $Q_n := I - \tilde P_n$ and $\| \cdot \|_{\alpha} := \| \Ah \cdot \|_{\ls \times L^2}$. Then the following Poincar\'e 
inequalities hold
\begin{align}
\label{eq:poincareAbstract}
\|\tilde P_n U \|_{\alpha_2} \leq \overline{\lambda}_n^{\alpha_2 - \alpha_1} \|\tilde P_n U \|_{\alpha_1},
\quad \| Q_n U \|_{\alpha_1} 
\leq \overline{\lambda}_n^{\alpha_1 - \alpha_2} \| Q_n U \|_{\alpha_1},
\end{align}
for $n\in\N$, $0 \leq \alpha_1 < \alpha_2$ and $\overline{\lambda}_n:=\max\{\mu_n, \hat{\mu}_n, \lambda_{n}\}$.
\end{proposition}

\begin{proof}
Density of the domains of fractional powers is well-known, see e.g.\ \cite[Theorem 37.2]{SellYou}. 
It is easy to see that the operator $\Ah$ is positive and symmetric, hence there 
exists a self-adjoint extension. Having already found an orthonormal system of 
eigenfunctions for which the corresponding eigenvalues don't have an accumulation 
point, we conclude that $\Ah$ itself is this extension. The inequalities \eqref{eq:poincareAbstract} follow from the estimate
\begin{align*}
 \|\tilde P_nU\|_{\alpha_2}^2
& =\int_{-h}^0\|\tilde P_n U(z)\|_{H^{\alpha_2}_{xy}}^2 \dz \\
&  \leq \int_{-h}^0\max\{\mu_n, \tilde{\mu}_n, \lambda_{n}\}^{2(\alpha_2 - \alpha_1)}\|\tilde P_n U(z)\|_{H^{\alpha_2-\alpha_1}_{xy}}^2 \dz \\
& = \max\{\mu_n, \tilde{\mu}_n, \lambda_{n}\}^{2(\alpha_2 - \alpha_1)}\|\tilde P_n U\|_{\alpha_2-\alpha_1}^2,
\end{align*}
where we used the inequality 
\[
\|\tilde P_n U(z)\|_{H^{\alpha_2}_{xy}}
 \leq \max\{\mu_n, \tilde{\mu}_n, \lambda_{n}\}^{2(\alpha_2 - \alpha_1)}\|\tilde P_n U(z)\|_{H^{\alpha_2-\alpha_1}_{xy}},
\]
which can be derived similarly as in \cite[Lemma 2.1]{GlattHoltz2009} by considering the operator $\Ah$ in $L^2(G)$.
\end{proof}

Let $B: H^1 \times H^1 \to H^{-2}$ be the bilinear operator
\[
B(U, U^\flat) = \begin{pmatrix}
\Ps \left[ (v \cdot \nablah )v^\flat+ w(v) \partial_z v^\flat \right]\\
(v \cdot \nablah) T^\flat + w(v) \partial_z T^\flat
\end{pmatrix}, \qquad U=(v,U), U^\flat=(v^\flat,T^\flat) \in H^1.
\]
The operator $B$ is continuous by e.g.\ \cite[Lemma 2.1]{PetcuTemamZiane2009}. In fact, by \eqref{eq:nonlinearNorm} in Proposition \ref{prop:nonlinearityEstimates},  the operator is well defined if $ U \in L^2_z H^2_{xy} \cap H^1_z H^1_{xy}$ are 
and we have the standard cancellation property 
\begin{align*}
	\left<B(U, U^\sharp), |U^\sharp|^q U^\sharp \right> & =0,
\end{align*}
for $q\geq0$, $U = (v, T)$ and $U^\sharp = (v^\sharp, T^\sharp)$ sufficiently regular.
Indeed, denoting $b(v^\sharp,v):=v^\sharp \cdot \nablah v+ w(v^\sharp) \partial_z v$ and
$b(v^\sharp,T):=v^\sharp \cdot \nablah T+ w(v^\sharp) \partial_z T$ with a slight 
abuse of notation, we have
\begin{align*}
	\left<b(v,v^\sharp), |v^\sharp|^q v^\sharp \right> & =\int_{M} (v \cdot \nablah v^\sharp) |v^\sharp|^q v^\sharp \dxyz\\
&\qquad-\int_{M} \left( \int_{-h}^{z}\divh v(x,y,z')\dz' \right) \partial_z v^\sharp  \;|v^\sharp|^q v^\sharp \dxyz\\
	&= - \frac{1}{q+2}\int_{M} (\divh v) |v^\sharp|^{q+2} - (\divh v)  |v^\sharp|^{q+2}  \dxyz\\
	&=0,
\end{align*}
where one uses the divergence-free vertical average condition \eqref{eq:meanDivFree}. The identity $\left<b(v,T^\sharp),|T^\sharp|^q T^\sharp \right> =0$ can be obtained similarly.

\begin{proposition}
\label{prop:nonlinearityEstimates}
We have the following estimates on the nonlinear term:
\begin{enumerate}
	\item For $U \in H^1$, $U^\flat \in H^1_zH^1_{xy}$ and $U^\sharp 
	\in L^2_z H^1_{xy}$, we have
	\begin{align}
		\label{eq:nonlinearFirst}
		\left| \left< B(U, U^\flat), U^\sharp \right> \right| 
		\leq c  \norm{U}_{H^1}\|  U^\flat \|_{H^1_zH^1_{xy}} \| U^\sharp \|_{L^2_z H^1_{xy}}.
	\end{align}
	\item For 
	$U \in L^2_z H^1_{xy}\cap L^\infty_z L^4_{xy}$, 
	$U^\flat \in L^2_z H^1_{xy} \cap L^\infty_z L^4_{xy}$, 
	$U^\sharp \in H^1_z L^4_{xy}$, we have
	\begin{multline}
		\label{eq:nonlinearSecond}
		\left| \left< B(U, U^\flat), U^\sharp \right> \right| \\
		\leq c (\norm{U}_{L^\infty_z L^4_{xy}}\| \nablah U^\flat \|_{L^2}+ \| \nablah U \|_{L^2} \| U^\flat \|_{L^\infty_z L^4_{xy}}) \| U^\sharp \|_{H^1_z L^4_{xy}}.
	\end{multline}
	\item For $U \in H^1_z H^1_{xy}$ with $v\in L^2_zH^1_{D,xy}$, we have
	\begin{multline}
		\label{eq:nonlineardz}
		\left| \left< \partial_z B(U, U), \partial_z U \right> \right|\\
		\leq c \norm{U}_{L^\infty_z L^4_{xy}} \left( \| \nablah \partial_z U \|_{L^2} 
		\|\partial_z U \|_{L^2} +\| \nablah \partial_z U \|_{L^2}^{3/2} \| \partial_z U \|_{L^2}^{1/2} \right).	
	\end{multline}
	\item For $U \in L^\infty_z L^4_{xy} \cap L^2_z H^2_{xy}$, $U^\flat \in L^2_z H^2_{xy} \cap H^1_z H^1_{xy}$, we have
	\begin{multline}
		\label{eq:nonlinearNorm}
		\| B(U, U^\flat) \|_{L^2}^2 \leq c \norm{U}_{L^\infty_z L^4_{xy}}^2 \| \nablah U^\flat \|_{L^2}
		\left( \| \nablah U^\flat \|_{L^2} + \| \Deltah U^\flat \|_{L^2} \right)\\
		+ c \| \nablah U \|_{L^2} \| \Deltah U \|_{L^2} \| \partial_z U^\flat \|_{L^2} 
		\left( \| \partial_z U^\flat \|_{L^2} + \| \nablah \partial_z U^\flat \|_{L^2} \right).
	\end{multline}
	\item For $U \in H^1_z H^1_{xy}$ with $(v,T) \in L^2_zH^2_{D,xy}\times L^2_zH^2_{N,xy}$, we have
	\begin{multline}
		\label{eq:nonlinearDelta}
		\left| \left<  B(U, U), \Deltah U\right> \right|
		%\leq c \norm{U}_{L^\infty_z L^4_{xy}}
		%\| \Deltah U \|_{L^2} \| \nablah U \|_{L^2}^{1/2} \left( \| \nablah U \|_{L^2} + \| \Deltah U \|_{L^2} \right)^{1/2}\\
		%+\| \nablah U \|_{L^2} \| \Deltah v \|_{L^2} \| \Deltah U \|_{L^2} 
		%\| \partial_z U \|_{L^2}^{1/2} \left( \| \partial_z U \|_{L^2} + \|\partial_z \nablah U \| \right)^{1/2}\\
		\leq c \| U \|_{L^\infty_zL^4_{xy}} \| U \|_{L^2_zH^1_{xy}}^{1/2} \| U \|_{L^2_zH^2_{xy}}^{3/2}\\
		+c \| U \|_{L^\infty_zL^4_{xy}}\| U \|_{L^2_zH^1_{xy}}^{1/2} \| U \|_{L^2_zH^2_{xy}}^{1/2} \| U \|_{H^1_zH^1_{xy}}.
	\end{multline}
\end{enumerate}
\end{proposition}

\begin{proof}
The estimates are established using the anisotropic H\"older inequality \eqref{eq:anisotropichoelder}. 
For simplicity 
of the presentation, we prove the estimates only for the temperature.

To prove \eqref{eq:nonlinearFirst}, we employ \eqref{eq:cltinterpolation} to obtain
\begin{align*}
	\left| \left< v \cdot \nablah T^\flat, T^\sharp \right> \right| 
	& \leq \norm{v}_{L^\infty_z L^4_{xy}} \| \nablah T^\flat \|_{L^2} \| T^\sharp \|_{L^2_z L^4_{xy}}\\
	& \leq \norm{v}_{H^1} \| T^\flat \|_{H^1} \| T^\sharp \|_{L^2_z H^1_{xy}}.
\end{align*}
Similarly, using $\| w(v) \|_{L^\infty_z L^2_{xy}}\leq \| \nablah v \|_{L^2}$ which follows from the definition of $w$ \eqref{eq:w}, we get
\begin{align*}
	\left| \left< w(v) \partial_z T^\flat, T^\sharp \right> \right| \leq c \| \nablah v \|_{L^2} \| \partial_z T^\flat \|_{L^2_z L^4_{xy}} \| T^\sharp \|_{L^2_z L^4_{xy}}.
\end{align*}

Continuing to \eqref{eq:nonlinearSecond}, we treat the first term as above by
\begin{align*}
	\left| \left< v \cdot \nablah T^\flat, T^\sharp \right> \right| 
	\leq \norm{v}_{L^\infty_z L^4_{xy}} \| \nablah T^\flat \|_{L^2} \| T^\sharp \|_{L^2_z L^4_{xy}}.
\end{align*}
Recalling that $w = 0$ on $\Gamma_u \cup \Gamma_b$, we use integration by parts and the definition of $w$ \eqref{eq:w} to get
\begin{align*}
	&\bigg| \left< w(v) \partial_z T^\flat, T^\sharp \right> \bigg| 
	= \left| \left< \divh v T^\flat, T^\sharp \right> + \left< w(v) T^\flat, \partial_z T^\sharp \right> \right|\\
	&\qquad \leq \| \divh v \|_{L^2} \| T^\flat \|_{L^\infty_z L^4_{xy}} \| T^\sharp \|_{L^2_z L^4_{xy}} + \| w(v) \|_{L^2} \| T^\flat \|_{L^\infty_z L^4_{xy}} \| \partial_z T^\sharp \|_{L^2_z L^4_{xy}}\\
	&\qquad \leq \| \nablah v \|_{L^2} \| T^\flat \|_{L^\infty_z L^4_{xy}} \| T^\sharp \|_{H^1_z L^4_{xy}}.
\end{align*}

For \eqref{eq:nonlineardz}, we use the cancellation property $\left<B(U, \partial_z U), \partial_z U\right> = 0$ and integration by parts with $v = 0$ on $\Gamma_l$ to obtain
\begin{align*}
	\left< \partial_z b(v, T), \partial_z T \right> 
	&= \left< \partial_z v \nablah T, \partial_z T \right> - \left< \div_H v \partial_z T, \partial_z T \right>\\
	&= -\left< \divh\partial_z v  T, \partial_z T \right>-\left< \partial_z v T, \nablah \partial_z T \right> - \left< \div_H v \partial_z T, \partial_z T \right>.
\end{align*}
Therefore, by the H\"{o}lder inequality and the Ladyzhenskaya inequality, we get
\begin{align*}
	&\left| \left< \partial_z b(v, T), \partial_z T \right> \right|
	\leq c \norm{U}_{L^\infty_z L^4_{xy}} \| \partial_z U \|_{L^2_z L^4_{xy}}  \| \nablah \partial_z U \|_{L^2}\\
	&\qquad \leq c \norm{U}_{L^\infty_z L^4_{xy}} \left( \| \nablah \partial_z U \|_{L^2} \| \partial_z U \|_{L^2} 
	+\| \nablah \partial_z U \|_{L^2}^{3/2} \| \partial_z U \|_{L^2}^{1/2} \right).
\end{align*}

To establish \eqref{eq:nonlinearNorm}, we use the 2D Gagliardo-Nirenberg inequality to deduce
\begin{align*}
	\| v \cdot \nablah T^\flat \|_{L^2}^2 
	&\leq  \norm{v}_{L^\infty_z L^4_{xy}}^2  \| \nablah T^\flat \|_{L^2_z L^4_{xy}}^2\\
	&\leq c \norm{v}_{L^\infty_z L^4_{xy}}^2 \| \nablah T^\flat \|_{L^2} 
	\left( \| \nablah T^\flat \|_{L^2} + \| \Deltah T^\flat \|_{L^2} \right).
\end{align*}
Similarly, from the definition of $w$ \eqref{eq:w} and the interpolation inequality \eqref{eq:cltinterpolation}, we get
\begin{align*}
	\| w(v) \partial_z T^\flat \|_{L^2}^2 
	&\leq  \| w(v) \|_{L^\infty_z L^4_{xy}}^2 
	  \| \partial_z T^\flat \|_{L^2_z L^4_{xy}}^2\\
	&\leq c \| \nablah v \|_{L^2} \| \Deltah v \|_{L^2} \| \partial_z T^\flat \|_{L^2} 
	\left( \| \partial_z T^\flat \|_{L^2} + \| \nablah \partial_z T^\flat \|_{L^2} \right).
\end{align*}

Finally, regarding \eqref{eq:nonlinearDelta}, we integrate by parts to get
%\[
%	\left< (-\Deltah)^{1/2} B(v, v), (-\Deltah)^{1/2} v\right> = \left< B(v, v), \Deltah v \right>.
%\end{align*}
%The bound \eqref{eq:nonlinearDelta} then follows from \eqref{eq:nonlinearNorm}.
%  
\begin{align*}
\left< b(v, T), \Deltah T \right> 
& =	-\left< \nablah b(v, T), \nablah T \right>\\
& = -\left<  b(v, \nablah T), \nablah T \right>\\
&\hphantom{ = }\quad -\left< \nablah v_1 \partial_x T + \nablah v_2 \partial_y T +\nablah w(v) \partial_z T, \nablah T \right>\\
&=- \left< \nablah v_1 \partial_x T, \nablah T\right>-\left< \nablah v_2 \partial_y T , \nablah T\right>-\left<  \nablah w(v) \partial_z T,  \nablah T \right>\\
&= \left<  v_1 \nablah \partial_x T , \nablah T\right>+ \left<  v_1  \partial_x T , \Deltah T\right>\\
&\hphantom{ = }\quad +\left<  v_2 \nablah \partial_y T , \nablah T\right>+ \left<  v_2  \partial_y T , \Deltah T\right>\\
&\hphantom{ = }\quad -\left< \nablah \divh v T,  \nablah T \right>+\left< w(v) T,  \nablah \partial_z T \right>,
\end{align*}
where we used the boundary conditions for $v$, $T$ and $w$. Hence, we have
\begin{align*}
|\left< b(v, T), \Deltah T \right>| 
&\leq c \| T \|_{L^2_zH^2_{xy}} \| v \|_{L^\infty_zL^4_{xy}} \| \nablah T \|_{L^2_zL^4_{xy}}\\
&\quad +c\| w(v) \|_{L^2_zL^4_{xy}} \| T \|_{L^\infty_zL^4_{xy}} \| \nablah \partial_z T \|_{L^2}\\
&\leq c \| U \|_{L^\infty_zL^4_{xy}} \| U \|_{L^2_zH^1_{xy}}^{1/2} \| U \|_{L^2_zH^2_{xy}}^{3/2}\\
&\quad +c \| U \|_{L^\infty_zL^4_{xy}}\| U \|_{L^2_zH^1_{xy}}^{1/2} \| U \|_{L^2_zH^2_{xy}}^{1/2} \| U \|_{H^1_zH^1_{xy}}.
\end{align*}
\end{proof}

Next, we turn to a higher order estimate in the vertical direction. It is used later to establish the existence of solutions regular enough to prove uniqueness. 

\begin{proposition}
\label{prop:nonlinearityEstimateshigherorder}
For $U,\partial_z U \in H^1_z H^1_{xy}$ with  $(v,T) \in L^2_zH^1_{D,xy} \times L^2_zH^2_{N,xy}$, we have
	\begin{align}\label{eq:nonlineardzz}
		\left| \left< \partial_{zz} B(U, U), \partial_{zz} U\right> \right|
		  \leq & c \norm{U}_{H^1_zL^4_{xy}}\norm{\partial_{zz} U}_{L^2}^{1/2}
							\norm{ U}_{H^2_zH^1_{xy}}^{3/2}.
	\end{align}
\end{proposition}
\begin{proof}
Similarly as above, we only show the estimates for the temperature to keep the presentation concise. We prove the claim by the anisotropic H\"{o}lder inequality \eqref{eq:anisotropichoelder}. Using the cancellation $\left<  B(v, \partial_{zz} T), \partial_{zz} T\right> =0$, we get
\begin{align*}
\left| \left< \partial_{zz} b(v, T), \partial_{zz} T\right> \right|\leq \left| \left< b(\partial_{zz}v, T), \partial_{zz} T\right> \right|+2\left| \left<b(\partial_z v,\partial_z T), \partial_{zz} T\right> \right|.
\end{align*}
Integration by parts and the Ladyzhenskaya inequality yield
		\begin{align*}
		\left| \left< \partial_{zz} v\cdot \nablah T,\partial_{zz} T \right> \right| 
		&= \left| \left< \partial_{zz} \divh v T,\partial_{zz} T \right>+ \left< \partial_{zz} v T,\nablah\partial_{zz} T \right> \right|\\
		& \leq c \norm{\partial_{zz} U}_{L^2_zL^4_{xy}}\norm{U}_{L^\infty_zL^4_{xy}}\norm{\nablah \partial_{zz} U}_{L^2}\\
		&  \leq c \norm{U}_{L^\infty_zL^4_{xy}}\norm{ U}_{H^2_zL^2_{xy}}^{1/2}\norm{U}_{H^2_zH^1_{xy}}^{3/2}
		\end{align*}		
		and
		\begin{align*}
		\left|\left< \divh v \partial_{zz} T,\partial_{zz} T\right>\right| 
		 \leq   \norm{U}_{L^\infty_zL^4_{xy}}\norm{\partial_{zz} U}_{L^2}^{1/2}\norm{\partial_{zz} U}_{L^2_zH^1_{xy}}^{3/2}.
		\end{align*}	
		From \eqref{eq:anisotropichoelder} and the interpolation inequality  \eqref{eq:gginterpolation} we deduce
		\begin{align*}	
		& \left| \left<\divh  \partial_z v \partial_z T,\partial_{zz} T\right> \right|\\
		& \qquad\leq c 	\norm{\divh  \partial_z v}_{L^{\infty}_zL^2_{xy}} 
						\norm{\partial_z T}_{L^2_zL^4_{xy}}
						\norm{\partial_{zz} T}_{L^2_zL^4_{xy}}\\
		& \qquad\leq   c 	\norm{\divh  \partial_z U}_{L^{2}}^{1/2}
							\norm{\divh  \partial_z U}_{H^1_zL^2_{xy}}^{1/2} 
							\norm{U}_{H^1_zL^4_{xy}}
							\norm{\partial_{zz} U}_{L^2}^{1/2}
							\norm{\partial_{zz} U}_{L^2_zH^1_{xy}}^{1/2}\\
		& \qquad\leq   c 	\norm{U}_{H^1_zL^4_{xy}}
							\norm{ U}_{H^2_zL^2_{xy}}^{1/2}
							\norm{ U}_{H^2_zH^1_{xy}}^{3/2}
		\end{align*}		
		and
		\begin{align*}	
		 \left| \left<   \partial_z v\cdot \partial_z\nablah T,\partial_{zz} T\right> \right|
		 \leq c 	\norm{U}_{H^1_zL^4_{xy}}
							\norm{ U}_{H^2_zL^2_{xy}}^{1/2}
							\norm{ U}_{H^2_zH^1_{xy}}^{3/2}.
		\end{align*}	
\end{proof}

We denote the hydrostatic contribution of the pressure by
\[
	\Apr U = \begin{pmatrix}
		-\Ps \beta_T g \int_z^0 \nablah T(x, y, z') \d\!z'\\
		0
	\end{pmatrix}, \qquad U \in L^2_z H^1_{xy},
\]
the Coriolis forcing by
\[
	EU = \begin{pmatrix}
		\Ps k \times v\\
		0
	\end{pmatrix}, \qquad U \in L^2,
\]
and the regular forcing by
\[
	F_U = \begin{pmatrix}
		\Ps f_v\\
		f_T
	\end{pmatrix} \in H^1 \quad \text{a.e.\ in} \ [0, t].
\]
To ease up the notation, we define
\[
	F(U) \equiv \begin{pmatrix} F_v(U) \\ F_T(U) \end{pmatrix} = \Apr U + EU - F_U, \qquad U \in L^2_z H^1_{xy}.
\]
Clearly, $F(U)$ satisfies a sublinear growth condition and is Lipschitz continuous, that is
\begin{align*}
	\| F(U) \|_{L^2} &\leq c\left( \| F_U \|_{L^2} + \norm{U}_{L^2_z H^1_{xy}} \right), & U &\in L^2_z H^1_{xy},\\
	\| F(U) - F(U^\sharp) \|_{L^2} &\leq c \| U - U^\sharp \|_{L^2_z H^1_{xy}}, & U, U^\sharp &\in L^2_z H^1_{xy}.
\end{align*}

Let $\cU$ be a separable Hilbert space. For another Hilbert space $X$ let $L_2(\cU, X)$ be the space of Hilbert-Schmidt operators $G: \cU \to X$. We define the noise term $\sigma: L^2_z H^1_{xy} \to L_2(\cU, L^2)$ by
\[
	\sigma(U) = \begin{pmatrix}
		\Ps \sigma_1(v, \nablah v, T, \nablah T)\\
		\sigma_2(v, \nablah v, T, \nablah T)
	\end{pmatrix}, \qquad U \in L^2_z H^1_{xy}.
\]
Assumptions on the noise term $\sigma$ are discussed in Section \ref{sect:noise_assumptions}.

We may now reformulate the equation \eqref{eq:primeqreformulated} as
\begin{equation}
	\label{eq:primeqFunct}
	\d\!U +[\Ah U + B(U) + F(U)] \dt = \sigma(U) \dW, \qquad U(t=0)=U_0.
\end{equation}

To establish bounds in better spaces required for global existence, we need to use the particular structure of the primitive equations, in particular the possibility of the decomposition into the barotropic and baroclinic modes.

Recalling the notation in \eqref{eq:defvbar} and below, we follow \cite{CaoTiti2007} and split the momentum equation in \eqref{eq:primeqreformulated} into equations for the barotropic mode $\overline{v}$
\begin{align}
	\label{eq:bar}
	\partial_t \overline{v} - \nu_v \Deltah \overline{v} + \frac{1}{\rho_0}\nablah p_s + \overline{v} \cdot \nablah \overline{v} &=-N(\widetilde{v})+ \Ac F_v(U)+\Ac \sigma_1(U) \dot W_1,\\
	\label{eq:v_bar_div}
	\divh \overline{v} &= 0,
\end{align}
where
\begin{align*}
	\Ac F_v(U) &= -k\times \overline{v}+\frac{\beta_Tg}{h}\int_{-h}^{0}\int_{z}^{0} \nablah T(x,y,z')\dz' \d\!z+\overline{f_v},\\
	N(\widetilde{v}) &= \frac{1}{h} \int_{-h}^{0} (\widetilde{v}\cdot\nabla_H\widetilde{v} + (\divh \widetilde{v})\,\widetilde{v}) \dz,
\end{align*}
and the baroclinic mode $\widetilde{v}$
\begin{equation} \label{eq:tilde}
	\partial_t \widetilde{v} - \nu_v \Deltah\widetilde{v}		+\widetilde{v}\cdot\nablah\widetilde{v} = - \bar v\cdot\nabla_H\widetilde{v}
	- \tilde v\cdot\nabla_H\overline{v}+N(\widetilde{v}) + \Rc F_v(U)+\Rc\sigma_1 \dot W_1,
\end{equation}
where
\begin{multline*}
	\Rc F_v(U) = F_v(U) - \Ac F_v(U) = -k \times \widetilde{v} + \widetilde{f}_v\\
	+\beta_T g \left[ \int_{z}^{0} \nablah T(x,y,z')\dz' - \frac{1}{h} \int_{-h}^{0}\int_{z}^{0} \nablah T(x,y,z')\dz'  \dz \right].
\end{multline*}

\subsection{Assumptions on noise}
\label{sect:noise_assumptions}

We assume that $\sigma$ satisfies the growth conditions
\begin{align}
	\label{eq:sigmaGrowthL2}
	\|\sigma(U)\|^2_{L_2(\cU,L^2)} &\leq c(1+\norm{U}_{L^{2}}^{2})+ \eta^2 \|(-\Deltah)^{1/2} U\|_{L^{2}}^{2},\\
	\label{eq:sigmaGrowthH1L2}
	\|\partial_z \sigma(U)\|^2_{L_2(\cU,L^2)} &\leq c(1+\| U\|_{H^1_zL^{2}_{xy}}^{2}) + \eta^2 \|(-\Deltah)^{1/2}\partial_z   U\|_{L^{2}}^{2},\\
	\label{eq:sigmaGrowthL2H1}
	\|(-\Deltah)^{1/2} \sigma(U)\|^2_{L_2(\cU,L^2)}& \leq c(1 + \| U\|_{L^{2}_zH^1_{xy}}^{2}) + \eta^2 \|\Deltah  U\|_{L^{2}}^{2},\\
	\label{eq:sigmaGrowthH2L2}
	\|\partial_{zz}  \sigma(U)\|^2_{L_2(\cU,L^2)}& \leq c(1 + \| U\|_{H^{2}_zL^2_{xy}}^{2}) + \eta^2 \|\partial_{zz} (-\Deltah)^{1/2}  U^n\|_{L^{2}}^{2},
\end{align}
for $U \in L^2_z H^1_{xy}$, $H^1_z H^1_{xy}$, $L^2_z H^2_{xy}$ and $H^2_z H^{1}_{xy}$, respectively, with $\eta > 0$, and is Lipschitz continuous 
\begin{align}
	\label{eq:sigmaLipL2}
	\|\sigma(U) - \sigma(U^\sharp)\|_{L_2(\cU,L^2)} &\leq \gamma \|U - U^\sharp\|_{L^2_z H^1_{xy}}, & U, U^\sharp &\in L^2_z H^1_{xy},\\
	\label{eq:sigmaLipH1L2}
	\|\partial_z \left[ \sigma(U) - \sigma(U^\sharp) \right] \|_{L_2(\cU,L^2)} &\leq  \gamma\|U - U^\sharp\|_{H^1_z H^1_{xy}}, & U, U^\sharp &\in H^1_z H^1_{xy},\\
	\label{eq:sigmaLipL2H1}
	\|(-\Deltah)^{1/2} \left[ \sigma(U) - \sigma(U^\sharp) \right] \|_{L_2(\cU,L^2)}& \leq \gamma \| U - U^\sharp \|_{L^2_z H^2_{xy}}, & U, U^\sharp &\in L^2_z H^2_{xy},	\\
	\label{eq:sigmaLipH2L2}
	\|\partial_{zz} \left[ \sigma(U) - \sigma(U^\sharp) \right] \|_{L_2(\cU,L^2)}& \leq \gamma \| U - U^\sharp \|_{H^2_z H^1_{xy}}, & U, U^\sharp &\in H^2_z H^1_{xy}
\end{align}
for some $\gamma > 0$. An example of a noise term $\sigma$ satisfying the above can be constructed similarly as in \cite[Section 2.5]{Brzezniak2020}. Furthermore, let on $\Gamma_b\cup \Gamma_u$,
\begin{align}
	\label{eq:sigma1boundary}
	\|\partial_z \sigma_1(U)\|^2_{L_2(\cU,L^2(G))} &\leq c(\|\partial_z v \|_{L^{2}(G)}^{2}+\| T\|_{L^{2}(G)}^{2}) + \eta^2 \|(-\Deltah)^{1/2}\partial_z   v\|_{L^{2}(G)}^{2},\\
	\label{eq:sigma2boundary}
	\|\sigma_2(U)\|^2_{L_2(\cU,L^2(G))} &\leq c(\|\partial_z v \|_{L^{2}(G)}^{2}+\| T\|_{L^{2}(G)}^{2}) + \eta^2 \|(-\Deltah)^{1/2}T\|_{L^{2}}^{2}.
\end{align}

Under the above conditions, we show the existence of a maximal solution from Theorem \ref{thm:maximalExistence} in Section \ref{sect:maximal_solutions_main}. However, to obtain global solutions from Theorem \ref{thm:globalExistence}, we need stronger assumptions using the split into \eqref{eq:bar} and \eqref{eq:tilde}. Thus, in Section \ref{sec:globalexistence}, we will consider noise of the form
\begin{equation}
	\label{eq:noiseglobal}
	\begin{split}
		\sigma_1(v)e_k&=\Psi_k \cdot \nablah \widetilde{v}+\Phi_k \cdot 
		\nablah \overline{v}+h_k(v),\\
		\sigma_2(U)e_k&=\Psi_k^T \cdot \nablah T+g_k(T,v),
	\end{split}
\end{equation}
where $(e_k)_k$ is a basis of the underlying Hilbert space $\cU$, the functions $\Psi_k: G \to \R^2$, $\Phi_k: (-h,0) \to \R^2$, $\Psi_k^T: M \to \R^2$ satisfy
\[
\Psi_k \in W^{1, \infty}(G), \quad \Phi_k \in W^{2, \infty}\left(-h, 0\right), \quad \Psi_k^T \in W^{2, \infty}_z W^{1, \infty}_{xy},
\]
 $\partial_z \Phi_k=0$ on $\Gamma_b\cup \Gamma_u$ 
and $h_k(v), g_k(v, T): M\to \R^2$ are such that
\begin{gather}
	\label{eq:h_k_div_free}
	\divh \Ac h_k(v) = 0 \ \text{for} \ \divh \overline{v} = 0,\\
	\label{eq:h_k_growth}
	\sum_{k=1}^\infty \| D h_k(v)\|_{L^q}^2 \leq c\left(1 + \| v\|_{L^q}^2 + \| D v\|_{L^q}^2 \right), \quad D \in \lbrace 1, \nablah, \partial_z \rbrace,\\
	\label{eq:g_k_growth}
	\sum_{k=1}^\infty \|g_k(v, T)\|_{L^q}^2 \leq c\left(1+\|T\|_{L^q}^2+\|\nablah \overline{v}\|_{L^q}^2+\|v\|_{L^q}^2\right),\\
	\label{eq:noise_eta}
	\sum_{k=1}^\infty \| \Psi_k^{T} \|_{L^\infty}^2 \leq \eta^2, \quad  \sum_{k=1}^\infty \| \Ac \Phi_k \|_{L^\infty}^2 \leq \eta^2, \quad  \sum_{k=1}^\infty \| \Rc \Psi_k \|_{L^\infty}^2 \leq \eta^2.
\end{gather}
Moreover, let on $\Gamma_b\cup \Gamma_u$ 
\begin{align*}
\|\partial_z h_k(v)\|_{L^2(G)}  \leq   c \|\partial_z v \|_{L^{2}(G)}, \quad
\|g_k(v)\|_{L^2(G)}  \leq   c(\|\partial_z v \|_{L^{2}(G)}+\| T\|_{L^{2}(G)})
\end{align*}
The above assumptions imply
\begin{align}
	\label{eq:sigma1_dz}
	\partial_z \sigma_1(v)e_k &= \Psi_k\cdot\nablah \partial_z v+ \partial_z \phi_k \cdot \nablah \overline{v} + \partial_z h_k(v),\\
	\label{eq:sigma1_a}
	\Ac \sigma_1(v)e_k &= (\Ac \Phi_k) \cdot \nablah \overline{v} + \Ac  h_k(v),\\
	\label{eq:sigma1_r}
	\Rc \sigma_1(v) e_k &= \Psi_k \cdot \nablah \widetilde{v} + (\Rc \Phi_k) \cdot \nablah \overline{v} + \Rc h_k(v).
\end{align}
Moreover, since $(\Ac \Phi_k)$ is constant, \eqref{eq:h_k_div_free} yields $\divh \Ac \sigma_1(v)=0$, in other words
\begin{equation}
	\label{eq:sigma_leray}
	(1 - P_G) \Ac \sigma_1(v) = 0 \ \text{in} \ L_2(\cU, L^2),
\end{equation}
where $P_G$ is the standard Leray-Helmholtz projection on the 2D domain $G$. By \eqref{eq:h_k_growth}, \eqref{eq:noise_eta}, \eqref{eq:sigma1_a} and the boundary conditions for $v$, we get
\begin{equation}
	\label{eq:sigma_bar_growth}
	\|(-\Deltah)^{1/2} \Ac \sigma_1(U)\|^2_{L_2(\cU,L^2)}
	\leq \eta^2 \| \Deltah \overline{v}\|_{L^2}^2 + c\left(1+\| (-\Deltah)^{1/2} v\|_{L^2}^{2}\right).
\end{equation}

\begin{example}
	The additional assumptions \eqref{eq:h_k_div_free}-\eqref{eq:noise_eta} are satisfied for
	\begin{align*}
		h_k(v) &= \zeta_k \overline{v} + \nu_k \widetilde{v} + \chi_k,\\
		g_k(v) &= \gamma_k T+ \Theta_k\cdot \nablah \overline{v} +\hat \zeta_k \overline{v} + \hat\nu_k \widetilde{v} + \hat\chi_k
	\end{align*}
	for sufficiently regular functions $\zeta_k$, $\nu_k$, $\chi_k$, $\gamma_k$, $\Theta_k$, $\hat\zeta_k$, $\hat \nu_k$ and $\hat\chi_k$ such that $\partial_z \zeta_k$, $\partial_z\nu_k$, $\partial_z \chi_k$, $\Theta_k$, $\hat \zeta_k$, $\hat\nu_k $ and $\hat\chi_k$ vanish on $\Gamma_b\cup \Gamma_u$.
\end{example}

\subsection{Stochastic preliminaries}

Let $(\Omega, \cF, \bF, \PP)$ be a stochastic basis with filtration $\bF = \lbrace \cF_t \rbrace_{t \geq 0}$ satisfying the usual conditions, in particular we assume the filtration $\bF$ to be right-continuous. Let $\cU$ be a separable Hilbert space and let $W$ be an $\bF$-cylindrical Wiener proceess with reproducing kernel Hilbert space $\cU$. It is well-known that if $\cU_0$ is a Hilbert space such that the embedding $\cU \hook \cU_0$ is Hilbert-Schmidt, the trajectories of $W$ are continuous in time in $\cU_0$.

Let $X$ be a Hilbert space. For a predictable process $\Phi: (0, t) \times \Omega \to L_2(\cU, X)$ satisfying $\| \Phi \|_{L^2(0, t; L_2(\cU, X)} < \infty$ $\PP$-a.s.\ the stochastic integral $\int_0^\cdot G \dW$ is well defined, see e.g.\ \cite[Section 4]{DaPratoZabczyk}. We will often use the following two variants of the Burkholder-Davis-Gundy inequality. Let $\Phi \in L^p(\Omega; L^2(0, t; L_2(\cU, X)))$ be predictable for some $p \geq 1$. Then
\begin{equation}
	\label{eq:bdg}
	\E \sup_{s \in [0, t]} \left\| \int_0^s \Phi \dW \right\|_X^{p} \leq c_{BDG} \E \left( \int_0^t \| \Phi \|_{L_2(\cU, X)}^2 \, ds \right)^{p/2}.
\end{equation}
For proof see e.g.\ \cite[Theorem 3.28, p.\ 166]{KaratzasShreve}. We note that the constant $c_{BDG}$ depends on $p$, even though we will tacitly omit the dependence  since it does not significantly affect the results in this paper. A fractional variant of the Burkholder-Davis-Gundy inequality has been established in \cite[Lemma 2.1]{FlandoliGatarek}. Let $p \geq 2$, $\Phi \in L^p(\Omega; L^p(0, t; L_2(\cU, X)))$ be predictable and let $\alpha \in [0, 1/2)$. Then
\begin{equation}
	\label{eq:bdgFrac}
	\E \left\| \int_0^t \Phi \dW \right\|_{W^{\alpha, p}(0, t; X)}^p \leq c_{BDG} \E \int_0^t \| \Phi \|_{L_2(\cU, X)}^p \ds.
\end{equation}

\subsection{Notion of solution}
\label{sect:solutions}
We adapt the definitions from \cite{Debussche2011}. All the solutions considered here are strong in the PDE sense. The definitions can be changed in a straightforward way to cover modified variants of the equation \eqref{eq:primeqFunct} which we will study in Section \ref{sect:maximal_solutions_main}.

\begin{definition}
Let $\mu_0$ be a probability measure on $H^1$ such that
\[
	\int_{H^1} \| U \|_{H^1}^2 \d\! U < \infty.
\]
\begin{enumerate}
	\item A triple $(\cS, U, \tau)$ is a \emph{local martingale solution} of \eqref{eq:primeqFunct} if $\cS = (\Omega, \cF, \bF, \PP)$ be a stochastic basis, $\tau$ is a strictly positive $\bF$-stopping time and $U = U(\cdot \wedge \tau): \Omega \times [0, \infty) \to H^1$ is an $\bF$-adapted stochastic process satisfying
	\begin{equation}
		\label{eq:mart_sol_regularity}
		\begin{aligned}
		U(\cdot \wedge \tau) &\in L^2\left( \Omega; C([0, \infty), H^1) \right),\\
		\one_{[0, \tau]} U &\in L^2\left( \Omega; L^2_{\rm{loc}}(0, \infty; H^1_z H^1_{xy} \cap D(\Ah)) \right),
		\end{aligned}
	\end{equation}
	the law of $U(0)$ is $\mu_0$ and the process $U$ satisfies
	\begin{equation}
		\label{eq:mart_sol_equation}
		U(s \wedge \tau) + \int_0^s \Ah U + B(U, U) + F(U) \dr = U(0) + \int_0^s \sigma(U) \dW
	\end{equation}
	for all $s \geq 0$ in $H$.
	\item The martingale solution $(\cS, U, \tau)$ is \emph{global} if $\tau = \infty$ $\PP$-almost surely.
\end{enumerate}
\end{definition}

\begin{definition}
\label{def:pathwise_solution}
Let $U_0 \in L^2(\Omega; H^1)$ be an $\cF_0$-measurable random variable and let $\cS = (\Omega, \cF, \bF, \PP)$ be a stochastic basis.
\begin{enumerate}
	\item A pair $(U, \tau)$ is a \emph{local pathwise solution} $\tau$ is a strictly positive $\bF$-stopping time and $U = U(\cdot \wedge \tau): \Omega \times [0, \infty) \to H^1$ is an $\bF$-adapted stochastic process such that \eqref{eq:mart_sol_regularity} and \eqref{eq:mart_sol_equation} hold w.r.t.\ the stochastic basis $\cS$.
	\item Let $(\tau_n)$ be an increasing sequence of $\bF$-stopping times converging $\PP$-a.s.\ to an $\bF$-stopping time $\xi$. The triple $(U, \xi, (\tau_n))$ is called a \emph{maximal pathwise solution} if $(U, \tau_n)$ is a local pathwise solution for all $n \in \N$ and
	\begin{equation}
		\label{eq:blowup}
		\sup_{s \in [0, \xi)} \| U \|_{H^1}^2 + \int_0^\xi \| \Ah U \|_{L^2}^2 +\| U \|_{H^1_z H^1_{xy}}^2 = \infty
	\end{equation}
	for a.a.\ $\omega \in \lbrace \xi < \infty \rbrace$.
	\item The maximal pathwise solution $(U, \xi, (\tau_n))$ is called \emph{global} if $\xi = \infty$ $\PP$-almost surely.
\end{enumerate} 
\end{definition}

\section{Existence of maximal solutions}
\label{sect:maximal_solutions_main}

To establish local existence of strong solutions of the system \eqref{eq:primeqreformulated}, we study a modified equation with a cut-off to make all the nonlinear transport term globally Lipschitz in suitable spaces. Local existence will then follow by a localization argument.

A similar approximation has been done in \cite{Debussche2011} for primitive equations with full diffusion. The estimates here are more involved due to the absence of vertical diffusion and a weaker cut-off function that does not act on $\partial_z U$.

Let $\theta \in C^{\infty}(\R)$ be a fixed function such that
\begin{equation*}
	\label{eq:theta}
	\one_{[-1/2,1/2]} \leq \theta \leq \one_{[-1,1]},
\end{equation*}
i.e.\ $\theta(r)=1$ for $r \in [-1/2,1/2]$ and $\theta(r)=0$ 
for $|r|\geq 1$. Let $\theta_\lambda(\cdot) = \theta(\cdot/\lambda)$ for $\lambda > 0$. For $\rho > 0$ fixed, we define
\begin{equation}
	\label{eq:cutoff}
	\theta(U(t)) =  \theta_\rho\left( \| U(t) \|_{L^\infty_z L^4_{xy}}  \right).
	% \theta_\kappa\left( \| U(s) \|_{L^2_z H^1_{xy}}  \right).
\end{equation}
We are looking for a solution $U=(v,T)$ to the modified system
\begin{equation}
	\label{eq:primeqcutoff}
	\d\! U + [\Ah U + \theta(U) B(U, U) + F(U)] \dt = \sigma(U) \dW, \quad U(0) = U_0.
\end{equation}

As we have already described in the introduction, the original system \eqref{eq:primeqreformulated}
doesn't provide any direct control over vertical derivatives $\partial_z v,
\partial_z T$. 
%Therefore, there is no hope for compactness and one cannot pass to 
%the limit of finite dimensional approximations in the nonlinear term in the usual way without 
%proving higher order. 
To work around this issue, we use the basis defined in \eqref{eq:Phi} and \eqref{eq:Psi} and consider
the spaces
\begin{align*}
	\; H:=\lso \times L^2 \text{ and } 	V:=L^2_zH^1_{D,xy}\cap \lso\times L^2_zH^1_{xy}
\end{align*}
with the corresponding inner products $\left< \cdot ,\cdot \right>_{H}$ 
and $\left< \cdot ,\cdot \right>_{V}$.By $V_1',V_2'$ we 
denote the dual spaces 
\begin{align*}
V':=V_1'\times V_2'=L^2_z H^{-1}_{xy}\times L^2_z H^{-1}_{xy},
\end{align*}
where, with slight abuse of notation, $H^{-1}_{xy}$ denotes the duals of $H^1_0(G)$ and $H^1(G)$, respectively.
We denote the dual pairing in $V\times V'$ by 
$\left< \cdot ,\cdot\right>_{H}$ to keep the notation simple.

\subsection{Galerkin scheme}
\label{sect:galerkin}
To define a suitable basis of $H_{\overline{\sigma}}$ for a Galerkin scheme, 
one can take advantage of the direct sum \eqref{eq:lso_split2}. In Section \ref{sec:functional} we have defined for $m,l\in \N, k\in\N_0$ the functions $\Phi_{m,k}$ and $\Psi_{m,k}$ in \eqref{eq:Phi} and \eqref{eq:Psi} to be  eigenfunctions of the hydrostatic Stokes operator and the horizontal Laplacian, see \eqref{eq:eigenfunctions}. 
Recall that these functions are dense in $H_{\overline{\sigma}}$ and $H_{2}$, respectively.
We set 
\begin{align*}
H_{\overline{\sigma},n}:=\operatorname{span}\{\Phi_{m,k} | m,k\leq n \}, \, H_{2,n}:=\operatorname{span}\{\Psi_{m,l} | m,l\leq n \}, \, H_{n}:=H_{\overline{\sigma},n}\times H_{2,n}
\end{align*} 
and define $P_n: H\to H_{n}$ to be the orthogonal projection onto $H_n$, see Proposition \ref{prop:poincare}. We remark that this step is only possible because 
we consider a cylindrical domain. In the case of a more realistic topography one has 
to transform the domain into a cylindrical one which leads to 
additional lower order terms in both equations. The next step 
is also a consequence of having a cylindrical domain.

The projector $P_n$ has the following important property:
For $k,m\in \N$ and for any function 
$g\in C^{\infty}(\overline{M})$, it holds
\begin{align*}
\left<g,\Phi_{m,k}\right>=- \frac{h^2}{k^2\pi^2}\left<g,-\partial_{zz}\Phi_{m,k}\right> =\frac{h^2}{k^2\pi^2}\left<\partial_{z}g,\partial_{z}\Phi_{m,k}\right>.
\end{align*}
Similarly, we have $\left<g,\Phi_{m,k}\right> =\frac{h^4}{k^4\pi^4}\left<\partial_{zz}g,\partial_{zz}\Phi_{m,k}\right>$ if additionally $\partial_z g=0$ on $\Gamma_b\cup\Gamma_u$ and we obtain the same equalities for $\Psi_{m,k}$ and $g\in C^{\infty}(\overline{M})$ with $g=0$ on $\Gamma_b\cup\Gamma_u$. Hence, when projecting the system 
\eqref{eq:primeqreformulated} onto $H_n$, the projected equations do not only 
hold for $U$ but also for $\partial_zU$ and $\partial_{zz}U$ because of \eqref{eq:sigma1boundary}, \eqref{eq:sigma2boundary} and $w=0$ on $\Gamma_b\cup\Gamma_u$. 

We now project the primitive equation onto the finite-dimensional space $H_{n}$ and 
we look for a
%weak (in the PDE sense) 
solution $U^n=(v^n,T^n): [0,T] \to V_n$ of the system of stochastic differential equations
\begin{equation}
	\label{eq:primeqgalerkin}
	\d\! U^n + [\Ah U^n + \theta(U^n) B^n(U^n, U^n) + F^n(U^n)] \dt = \sigma^n(U^n) \dW
\end{equation}
with the initial condition
\begin{equation} \label{eq:icgalerkin}
	U^n(t=0) = U_0^n := (P_n v_0, P_n T_0),
\end{equation}
where
\[
	B^n(\cdot, \cdot) = P_n B(\cdot, \cdot), \qquad F^n(\cdot) = P_n F(\cdot) 
	\quad \text{and} \quad \sigma^n(\cdot)(\cdot) = P_n[\sigma(\cdot)(\cdot)].
\]

Existence of solutions $U^n$ of the finite-dimensional system \eqref{eq:primeqgalerkin} 
follows by a standard fixed point argument if $U_0\in H$ since the nonlinear term $\theta(U^n)B^n(U^n, U^n)$ 
is globally Lipschitz and 
satisfies a sublinear growth condition thanks the cut-off function $\theta$ and the 
form of the eigenvectors $\lbrace \Phi_{m, k}, \Psi_{m, l} \mid m,l \in \N, k \in \N_0 \rbrace$.
Note that such a solution with enough regularity to have the trace of $\nablah T$ defined on
$\Gamma_l$ satisfies the Neumann boundary condition for $T$.

Note that $v_n(t) \in H_{\overline{\sigma},n}$ encodes 
a divergence free 
condition and determines the pressures $p_s^{n}$ and that the projected equations hold 
also for the vertical derivatives. Thus, we have for
$\partial_z U^n$ the system
\begin{multline}
	\label{eq:primeqgalerkinz}
	\d\! \partial_z U^n + [\Ah  \partial_z U^n + \theta(U^n) \partial_z  B^n(U^n, U^n) + \partial_z F(U^n)] \dt\\
	= \partial_z F^n_U \dt + \partial_z \sigma^n(U^n) \dW,
\end{multline}
with the initial condition $ \partial_z U^n(t=0) = \partial_z U^n_0$
and for $\partial_{zz} U^n$ the system
\begin{multline}
	\label{eq:primeqgalerkinzz}
	\d\! \partial_{zz} U^n + [\Ah  \partial_{zz} U^n + \theta(U^n) \partial_{zz}  B^n(U^n, U^n) + \partial_{zz} F(U^n)] \dt\\
	= \partial_{zz} F^n_U \dt + \partial_{zz} \sigma^n(U^n) \dW,
\end{multline}
with the initial condition $ \partial_{zz} U^n(t=0) = \partial_{zz}  U^n_0$.

\subsection{Estimates}
\label{sect:gal_estimates}
In this section, we establish the main estimates needed to pass to the limit in the Galerkin approximations.

\begin{lemma}\label{lemma:galerkinestimate}
Let $t>0$, $q\geq 2$ and let $U_0=(v_0,T_0) \in L^q(\Omega; H^1)$ be an $\cF_0$-measurable random variable. Let $\sigma$ satisfy \eqref{eq:sigmaGrowthL2}-\eqref{eq:sigmaGrowthL2H1}, \eqref{eq:sigmaLipL2}-\eqref{eq:sigmaLipL2H1} and let
\[
	F_U \in L^q\left( \Omega; L^q(0, t; H^1) \right).
\]
Assume $\nu > \eta^2\left(\frac{q-1}2 + qc^2_B \right)$. Then the following holds:
\begin{enumerate}
	\item The sequence $U^n$ is bounded in
	\begin{gather*}
		L^q\left( \Omega; L^\infty(0, t; H^1) \right), \quad L^q\left( \Omega; L^2(0, t; H^1_z H^1_{xy} \cap L^2_z H^2_{xy}) \right).
	\end{gather*}
		\item Let  $\alpha\in [0,1/2)$. Then $\int_0^\cdot \sigma^n(U^n) \dW$ is bounded in
	\[
		L^q\left( \Omega; W^{\alpha, q}(0, t; L^2) \right).	
	\]
	\item Let $q \geq 4$ and $p\geq q/2$. Then $U^n - \int_0^\cdot \sigma^n(U^n) \dW$ is bounded in 
	\[
		L^p\left( \Omega; W^{1, 2}(0, t; L^2)\right) % \cap N^{1/2, p}(0, t; L^2)\right)
	\]
\end{enumerate}
\end{lemma}

\begin{proof}
We divide the proof into eight steps. Due to the length of the estimates we do not try make the estimates as sharp as possible and rather aim for readability. In particular, some of the constants below do depend on $\eta$ even though it is not always necessary.

\textbf{Step 1}: $U^n$ is bounded in $L^q(\Omega; L^\infty(0, t; L^2))$.\\
Applying the finite-dimensional It\^{o} formula to \eqref{eq:primeqgalerkin} and estimating the trace term, we have
\begin{align}
\begin{alignedat}{1} \label{eq:galerkin_bound_L2}
	\d\! &\|U^n\|_{L^2}^q + q \nu \|U^n\|_{L^2}^{q-2} \|(-\Deltah)^{1/2} U^n\|_{L^2}^{2} \dt\\
&\leq -q \|U^n\|_{L^2}^{q-2} \left< U^n, F^n(U^n) \right> \dt 
+\frac{q(q-1)}{2}  \|U^n\|_{L^2}^{q-2} \|\sigma^n (U^n)\|^2_{L_2(U,L^2)} \dt\\
&\quad  +q \|U^n\|_{L^2}^{q-2} \left< U^n,\sigma^n(U^n) \d\! W\right>\\
&= \sum_{i = 1}^2 I^n_i \dt + I^n_3 \dW 
\end{alignedat}
\end{align}
where we already used the cancellation property $\left<B(U^n,U^n),U^n\right>=0$. 
The linear part is straightforward. For any $\varepsilon>0$, we have 
\begin{align*}
	\int_0^t \left| I_1^n \right| \ds \leq & q\varepsilon \int_0^t\|U^n\|_{L^2}^{q-2}\|(-\Deltah)^{1/2} U^n\|_{L^2}^{2} \ds + c_\varepsilon\|U^n\|_{L^2}^{q}\\
	&+\frac14  \sup_{s\in[0,t]}\|U^n\|_{L^2}^{q}+ \left(\int_0^t \|F_U^n\|_{L^2}^{2}\ds \right)^{q/2}.
\end{align*}
By the growth estimate \eqref{eq:sigmaGrowthL2} of $\sigma$ in $L_2(\cU, L^2)$, we obtain 
\begin{align*}
I_2^n &\leq \frac{q(q-1)}{2} \|U^n\|_{L^2}^{q-2} \|\sigma^n (U^n)\|^2_{L_2(\cU,L^2)}\\
% & \leq \frac{q}{2} \|U^n\|_{L^2}^{q-2} (c+c\|U^n\|_{L^{2}}^{2}+ \eta^2 \|(-\Deltah)^{1/2} U^n\|_{L^{2}}^{2})\\
 & \leq c\left( 1 + \|U^n\|_{L^2}^{q} \right) + \frac{q(q-1)}{2} \eta^2\|U^n\|_{L^2}^{q-2} \|(-\Deltah)^{1/2} U^n\|_{L^{2}}^{2}.
\end{align*}
The stochastic integral is estimated by the Burkholder-Davis-Gundy inequality \eqref{eq:bdg} by
\begin{align*}
	\E \sup_{s\in[0,t]} &\left| \int_0^s q \|U^n\|_{L^2}^{q-2} \left< U^n,\sigma^n(U^n) \d\! W\right> \right| \\
	&\leq qc_{BDG} \E \left( \int_0^t \|U^n\|_{L^2}^{2q-2} \|\sigma^n (U^n)\|^2_{L_2(U,L^2)} \ds \right)^{1/2}\\
& \leq qc_{BDG}\E \left(\int_0^t c\left( 1 + \|U^n\|_{L^2}^{2q} \right) +  \eta^2\|U^n\|_{L^2}^{2q-2}\|(-\Deltah)^{1/2} U^n\|_{L^{2}}^{2} \ds\right)^{1/2}\\
	& \leq \frac14 \E \sup_{s\in[0,t]}\|U^n\|_{L^2}^{q} 
			+q^2c_{BDG}^2\eta^2 \E \int_0^t \|U^n\|_{L^2}^{q-2}\|(-\Deltah)^{1/2} U^n\|_{L^{2}}^{2}\ds\\
	&\hphantom{\leq}+c\E\int_0^t 1 + \|U^n\|_{L^2}^{q} \ds.
\end{align*}
Collecting the estimates above yields
\begin{multline*}
 \frac14 \E \sup_{s\in[0,t]}\|U^n\|_{L^2}^{q}
+ c(q, \nu, \varepsilon, \eta) \E\int_0^t \|U^n\|_{L^2}^{q-2} \|(-\Deltah)^{1/2} U^n\|_{L^2}^{2} \ds\\
\leq c_\varepsilon \E\left[\|U^n(0)\|_{L^2}^{q} + 1+\int_0^t  \|U^n\|_{L^2}^{q} \ds+ \left(\int_0^t \|F_U^n\|_{L^2}^{2}\ds \right)^{q/2}\right],
\end{multline*}
where $c(q, \nu, \varepsilon, \eta) = q[\nu -\varepsilon- \eta^2(\tfrac{q-1}{2} + qc^2_B)]$. Choosing $\varepsilon$ sufficiently small, we employ Gronwall's lemma to get
\begin{multline}\label{eq:boundL2}
 \E \sup_{s\in[0,t]}\|U^n\|_{L^2}^{q}
 +\E\int_0^t \|U^n\|_{L^2}^{q-2} \|(-\Deltah)^{1/2}U^n\|_{L^2}^{2} \ds\\
\leq c\E\left[\|U^n(0)\|_{L^2}^{q} + 1+ \left(\int_0^t \|F_U^n\|_{L^2}^{2}\ds \right)^{q/2} \right].
\end{multline}

\textbf{Step 2}: $U^n$ is bounded in $L^q(\Omega; L^2(0, t; L^2_z H^1_{xy}))$.\\
Assuming $\eta$ sufficiently small so that the gradient terms in the estimates correction terms can be moved to the left-hand side, we use \eqref{eq:galerkin_bound_L2}, the estimates from Step 1 with $q = 2$ and the growth estimate \eqref{eq:sigmaGrowthL2} on $\sigma$ in $L_2(\cU, L^2)$ to get
\begin{multline*}
	\left( \int_0^t \| (-\Deltah)^{1/2} U^n \|_{L^2}^2 \ds \right)^{q/2} \leq c \| U^n(0) \|_{L^2}^q+\sup_{s\in[0,t]}\|U^n\|_{L^2}^{q}+c\\
%	  + \varepsilon \left( \int_0^t \| (-\Deltah)^{1/2} U^n \|_{L^2}^2 \ds \right)^{q/2}
+ c \left(\int_0^t \|F_U^n\|_{L^2}^{2}\ds \right)^{q/2} + c \left( \int_0^t \left< U^n, \sigma^n(U^n) \dW \right> \right)^{q/2}.
\end{multline*}
%for any $\varepsilon > 0$.
By the Burkholder-Davis-Gundy inequality \eqref{eq:bdg} we have
\begin{align*}
	\E \sup_{s \in [0, t]} &\left| \int_0^s \left< U^n,\sigma^n(U^n) \d\! W\right> \right|^{q/2} \leq c_{BDG} \E\left( \int_0^t \| U^n \|_{L^2}^2 \| \sigma^n(U^n) \|_{L_2(\cU, L^2)}^2 \ds \right)^{q/4}\\
	&\leq c \E\left( \int_0^t c\left(1 + \| U^n \|_{L^2}^4 \right) + \eta^2 \| U^n \|_{L^2}^2  \| (-\Deltah)^{1/2} U^n \|_{L^2}^2 \ds \right)^{q/4}\\
	&\leq \varepsilon \E \left( \int_0^t \| (-\Deltah)^{1/2} U^n \|_{L^2}^2 \ds \right)^{q/2} + c_\varepsilon \E\left[ 1 + \sup_{s \in [0, t]} \| U^n \|^q \right].
\end{align*}
Choosing $\varepsilon$ sufficiently small, we use the bound \eqref{eq:boundL2} from Step 1 to get
\begin{equation}
\label{eq:boundLpL2H1}
\begin{split}
	\E &\left( \int_0^t \| (-\Deltah)^{1/2} U^n \|_{L^2}^2 \ds \right)^{q/2}\\
	&\leq c\E \left[ \| U^n(0) \|_{L^2}^q + \sup_{s \in [0, t]} \| U^n \|_{L^2}^q +1+\left(\int_0^t \|F_U^n\|_{L^2}^{2}\ds \right)^{q/2} \right]\\
	&\leq c\E \left[ \| U^n(0) \|_{L^2}^q+1+\left(\int_0^t \|F_U^n\|_{L^2}^{2}\ds \right)^{q/2} \right].
\end{split}
\end{equation}

\textbf{Step 3}: $U^n$ is bounded in $L^q(\Omega; L^\infty(0, t; H^1_zL^2_x))$.\\
To obtain the estimates for vertical derivative, we use the properties 
of the projector $P$, in particular \eqref{eq:primeqgalerkinz}. 
Employing the It\^{o} formula and estimates for the trace term, we get
\begin{equation*}
%\label{eq:galerkinuz}
\begin{split}
\d\! &\|\partial_z U^n\|_{L^2}^q +q \nu \|\partial_z U^n\|_{L^2}^{q-2} \|(-\Deltah)^{1/2} \partial_z U^n\|_{L^2}^{2} \dt\\
&\leq -q \|\partial_z U^n\|_{L^2}^{q-2} \left< \partial_z U^n, \partial_z F^n(U^n) + \theta(U^n) P_n \partial_z B( U^n,U^n)\right> \dt\\
&\quad +\frac{q(q-1)}{2}  \|\partial_z U^n\|_{L^2}^{q-2} \|\partial_z \sigma^n (U^n)\|^2_{L_2(U,L^2)} \dt\\
&\quad +q \|\partial_z U^n\|_{L^2}^{q-2} \left< \partial_z U^n,\partial_z \sigma^n(U^n) \dW \right>\\
&= \sum_{i=1}^3 I^n_i \dt + I^n_4 \dW.
\end{split}
\end{equation*}
Similarly as in Step 1, we estimate the lower order linear term as
\begin{multline*}
	\int_0^t \left| I_1^n \right| \ds \leq \frac14 \sup_{s \in [0, t]} \| \partial_z U^n\|_{L^2}^q + c \int_0^t \| \partial_z U^n\|_{L^2}^q \ds\\
	+ c \left( \int_0^t \| \nablah U^n \|_{L^2}^2 \ds \right)^{q/2} + c \left( \int_0^t \| \partial_z F_U \|_{L^2}^2 \ds \right)^{q/2},
\end{multline*}
the correction terms as
\[
	I^n_3 \leq c\left( 1 + \| \partial_z U^n \|_{L^2}^q \right) + \frac{q(q-1)}{2} \eta^2 \| U^n\|_{H^1_zL^2_{xy}}^{q-2} \|(-\Deltah)^{1/2} \partial_z U^n\|_{L^2}^{2},
\]
and the stochastic term by the Burkholder-Davis-Gundy inequality \eqref{eq:bdg}
\begin{multline*}
	\E \sup_{s \in [0, t]} \left| \int_0^s I^n_4 \dW \right| \leq \frac14 \E \sup_{s\in[0,t]}\| \partial_z U^n\|_{L^2}^{q} + c\E\int_0^t 1 + \| U^n\|_{H^1_zL^2{xy}}^{q} \ds
\\
	+q^2c_{BDG}^2\eta^2 \E \int_0^t \|\partial_{z}U^n\|_{L^2}^{q-2}\|(-\Deltah)^{1/2} \partial_z U^n\|_{L^{2}}^{2}\ds.
\end{multline*}
The nonlinear term is dealt with by the estimate \eqref{eq:nonlineardz}. Recalling the form of the cut-off $\theta$ \eqref{eq:cutoff}, for $\varepsilon>0$, we have 
\begin{align*}
	\left| I_2^n \right| &\leq cq \theta(U^n)\norm{U}_{L^\infty_z L^4_{xy}} \| \nablah \partial_z U \|_{L^2}^{3/2} \| \partial_z U \|_{L^2}^{q-3/2} \\
	&\leq cq\rho \| \nablah \partial_z U \|_{L^2}^{3/2} \| \partial_z U \|_{L^2}^{q-3/2} \\
	&\leq q\varepsilon \| \nablah \partial_z U^n \|_{L^2}^{2} \| \partial_z U^n \|_{L^2}^{q-2} +c_\varepsilon q \rho^{4}\| \partial_z U^n \|_{L^2}^{q}.
\end{align*}

Collecting the above estimates, we obtain
\begin{multline*}
	\E\left[ \frac12 \sup_{s\in[0,t]}\|\partial_z U^n\|_{L^2}^{q} + c(q, \nu, \varepsilon, \eta) \int_0^t \|\partial_z U^n\|_{L^2}^{q-2} \|(-\Deltah)^{1/2} \partial_z U^n\|_{L^2}^{2} \ds\right]\\
	\leq c \E\left[\|\partial_z U^n(0)\|_{L^2}^{q} +1+ c \left( \int_0^t \| \partial_z F_U \|_{L^2}^2 \ds \right)^{q/2}+ \left( \int_0^t \| \nablah U^n \|_{L^2}^2 \ds \right)^{q/2} \right]\\
	+ c \E \int_0^t \| U^n\|_{L^2}^{q} + \|\partial_z U^n\|_{L^2}^{q}\ds,
\end{multline*}
where again $c(q, \nu,\varepsilon, \eta) = q[\nu - \varepsilon -\eta^2(q c_{BDG}^2 + \frac{q-1}{2})]$.
With $\varepsilon$ sufficiently small, Gronwall's lemma and the bound \eqref{eq:boundLpL2H1} yield
\begin{multline}
	\label{eq:boundH1L2}
	\E\left[ \sup_{s\in[0,t]}\|\partial_z U^n\|_{L^2}^{q} + \int_0^t \|\partial_z U^n\|_{L^2}^{q-2} (\|(-\Deltah)^{1/2} \partial_z U^n\|_{L^2}^{2}) \ds\right]\\
	\leq c \E \left[  \|U^n(0)\|_{H^1_zL^2_{xy}}^{q} + 1+ \left( \int_0^t \| F_U \|_{H^1_zL^2_{xy}}^2 \ds \right)^{q/2} \right].
\end{multline}

\textbf{Step 4}: $U^n$ is bounded in $L^q(\Omega; L^2(0, t; H^1_z H^1_{xy}))$.\\
The desired estimate 
\begin{multline}
	\label{eq:boundLpH1H1}
	\E \left( \int_0^t \| \nablah \partial_z U^n \|_{L^2}^2 \ds \right)^{q/2}\\
	\leq c \E \left[  \|U^n(0)\|_{H^1_zL^2_{xy}}^{q} + 1+ \left( \int_0^t \| F_U \|_{H^1_zL^2_{xy}}^2 \ds \right)^{q/2} \right]
\end{multline}
can be obtained from \eqref{eq:boundLpL2H1} and \eqref{eq:boundH1L2} with $q=2$ similarly as in Step 2.

\textbf{Step 5}: $U^n$ is bounded in $L^q(\Omega; L^\infty(0, t; L^2_z H^1_{xy}))$.\\
Since our basis consists of eigenvectors of the linear operator $\Ah$, we may use the It\^{o} formula on \eqref{eq:primeqgalerkin} to obtain 
\begin{equation*}
%\label{eq:galerkinAU}
\begin{split}
	\d &\|(-\Deltah)^{1/2} U^n\|_{L^2}^q + q \nu \|(-\Deltah)^{1/2} U^n\|_{L^2}^{q-2}\|\Deltah U^n\|_{L^2}^{2} \dt\\
	&\leq -q \|(-\Deltah)^{1/2} U^n\|_{L^2}^{q-2} \left< (-\Deltah)^{1/2} U^n, (-\Deltah)^{1/2} F^n(U^n) \right> \dt\\
	&\hphantom{= \ } -q \|(-\Deltah)^{1/2} U^n\|_{L^2}^{q-2} \left< (-\Deltah)^{1/2} U^n,\theta(U^n)P_n(-\Deltah)^{1/2} B( U^n,U^n) \right> \dt\\
	&\hphantom{= \ } +\frac{q(q-1)}{2}\|(-\Deltah)^{1/2} U^n\|_{L^2}^{q-2} \|(-\Deltah)^{1/2} \sigma^n (U^n)\|^2_{L_2(\cU,L^2)} \dt\\
	&\hphantom{= \ } +q \|(-\Deltah)^{1/2} U^n\|_{L^2}^{q-2} \left< (-\Deltah)^{1/2} U^n,(-\Deltah)^{1/2} \sigma^n(U^n) \dW \right>\\
	&= \sum_{i = 1}^3 I_i^n \dt +  I_4^n \dW.
\end{split}
\end{equation*}
Similarly as above, for $\varepsilon>0$, we estimate the lower order term and the correction term by
\begin{align*}
	\int_0^t I_1^n \ds &\leq \frac{\varepsilon}{4}\int_0^t \| \Deltah U^n \|_{L^2}^2 \| \nablah U^n \|_{L^2}^{q-2}\ds+ c_\varepsilon \int_0^t  \| \nablah U^n \|_{L^2}^{q}\ds\\
	& \phantom{\leq} + \frac14 \sup_{s\in[0,t]}  \| \nablah U^n \|_{L^2}^{q} +c  \left( \int_0^t \| F_U \|_{L^2}^2 \ds \right)^{q/2},\\
	I_3^n &\leq  \frac{q(q-1)}{2} \eta^2 \| \Deltah U^n \|_{L^2}^2 \| \nablah U^n \|_{L^2}^{q-2} + c \left( 1 + \|  U^n \|_{L^2_zH^1_{xy}}^{q}\right).
\end{align*}
From the Burkholder-Davis-Gundy inequality \eqref{eq:bdg}, we deduce
\begin{multline*}
	\E \sup_{s \in [0, t]} \left| \int_0^s I_4^n \dW \right| \leq \frac14 \E \sup_{s\in[0,t]}\| \nablah U^n\|_{L^2}^{q} + c\E\int_0^t 1 + \| U^n\|_{L^2_zH^1_{xy}}^{q} \ds
\\
	+q^2c_{BDG}^2\eta^2 \E \int_0^t \|\nablah U^n\|_{L^2}^{q-2}\|\Deltah U^n\|_{L^{2}}^{2}\ds.
\end{multline*}
To deal with the nonlinear term $I_2^n$, we use \eqref{eq:nonlinearDelta}. We estimate the first term on the right-hand side of \eqref{eq:nonlinearDelta} multiplied by $\| \nablah U^n \|_{L^2}^{q-2}$ using the form of the cut-off $\theta$ \eqref{eq:cutoff} by
\begin{multline*}
	\theta(U^n) \| U \|_{L^\infty_zL^4_{xy}} \| U \|_{L^2_zH^1_{xy}}^{1/2} \| U \|_{L^2_zH^2_{xy}}^{3/2}\| \nablah U^n \|_{L^2}^{q-2}\\
	\leq \frac{\varepsilon}{4} \| \Deltah U^n \|_{L^2}^{2} \| \nablah U^n \|_{L^2}^{q-2} + c_\varepsilon  \rho^{4}\left(1 + \| U^n \|_{L^2}^q + \| \nablah U^n \|_{L^2}^q \right).
\end{multline*}
For the second term on the right-hand side of \eqref{eq:nonlinearDelta}, we have
\begin{align*}
	\theta(U^n)& \| U \|_{L^\infty_zL^4_{xy}}\| U \|_{L^2_zH^1_{xy}}^{1/2} \| U \|_{L^2_zH^2_{xy}}^{1/2} \| U \|_{H^1_zH^1_{xy}}\| \nablah U^n \|_{L^2}^{q-2}\\
	&\leq q\frac{\varepsilon}{4} \| \Deltah U^n \|_{L^2}^{2} \| \nablah U^n \|_{L^2}^{q-2}\\
	&\hphantom{\leq } + c_\varepsilon \rho^{4/3} \left(1 + \| U^n \|_{H^1_zL^2_{xy}}^q + \| \nablah U^n \|_{L^2}^q+ \| \nablah \partial_z U^n \|_{L^2}^2 \| \nablah U^n \|_{L^2}^{q-2}\right).
\end{align*}

Collecting the above and using the estimate
\begin{multline*}
\int_0^t \| \nablah \partial_z U^n \|_{L^2}^2 \| \nablah U^n \|_{L^2}^{q-2}\ds\\
\leq \frac14 \sup_{s\in[0,t]} \| \nablah U^n\|_{L^2}^{q}+c\left(\int_0^t \| \nablah \partial_z U^n \|_{L^2}^2 \ds\right)^{q/2},
\end{multline*}
we deduce
\begin{multline*}
	\E\left[ \frac14 \sup_{s\in[0,t]} \| \nablah U^n\|_{L^2}^{q} + c(p, \nu, \varepsilon, \eta) \int_0^t \|\nablah U^n\|_{L^2}^{q-2} \|\Deltah U^n\|_{L^2}^{2} \ds\right]\\
	\leq c  \E \left[ \|\nablah U^n(0)\|_{L^2}^{q} +1+ \left( \int_0^t \| F_U \|_{L^2}^2 \ds \right)^{q/2}+\left(\int_0^t \| \nablah \partial_z U^n \|_{L^2}^2 \ds\right)^{q/2}\right]\\
	+ c \E \int_0^t \| U^n \|_{H^1_zL^2_{xy}}^q + \| U^n\|_{L^2_zH^1_{xy}}^{q}  \ds,
\end{multline*}
where $c(q,\nu,\varepsilon, \eta) = q[ \nu -\varepsilon- \eta^2( \frac{q-1}2 + qc^2_B)]$. Choosing $\varepsilon$ sufficiently small, we may employ Gronwall's lemma and \eqref{eq:boundH1L2} to get
\begin{multline}
	\label{eq:boundL2H1}
	\E\left[ \sup_{s\in[0,t]} \| \nablah U^n\|_{L^2}^{q} + \int_0^t \|\nablah U^n\|_{L^2}^{q-2} \|\Deltah U^n\|_{L^2}^{2} \ds \right]\\
	\leq c \E \left[ \|U^n(0)\|_{H^1}^{q} +1+  \left( \int_0^t \| F_U \|_{H^1_zL^2_{xy}}^2 \ds \right)^{q/2}\right].
\end{multline}

\textbf{Step 6}: $U^n$ is bounded in $L^q(\Omega; L^2(0, t ;L^2_z H^2_{xy}))$.\\
The estimate
\begin{multline}
	\label{eq:boundLpL2H2}
	\E \left( \int_0^t \| \Deltah U^n \|_{L^2}^2 \ds \right)^{q/2}\\
	\leq c \E \left[  \| U^n(0) \|_{H^1}^q +1+ \left( \int_0^t \| F_U \|_{H^1_zL^2_{xy}}^2 \ds \right)^{q/2} \right]
\end{multline}
follows from the bounds  \eqref{eq:boundH1L2} and \eqref{eq:boundL2H1} with $q=2$ similarly as in Step 2.

\textbf{Step 7}: $\int_0^\cdot \sigma^n(U^n) \dW$ is bounded in $L^q(\Omega; W^{\alpha, q}(0,t;L^2))$ for $\alpha \in [0, 1/2)$.\\
Using the fractional version of the Burkholder-Davis-Gundy inequality \eqref{eq:bdgFrac} and the bound \eqref{eq:boundL2H1}, we get
\begin{multline}
	\label{eq:boundSigmaWap}
	\E \left\| \int_0^\cdot \sigma^n(U^n) \dW \right\|_{W^{\alpha, q}(0, t; L^2)}^q  \leq c \E \int_0^t \| \sigma^n(U^n) \|_{L_2(\cU, L^2)}^q \ds\\
	%&  \leq c \E \int_0^t 1+\|U^n \|_{L^2_zH^1_{xy}}^q \ds\\
	 \leq c \E \left[ \|U^n(0)\|_{H^1}^{q} +1+ \left( \int_0^t \| F_U \|_{H^1_zL^2_{xy}}^2 \ds \right)^{q/2} \right].
\end{multline}

\textbf{Step 8}: $U^n(\cdot) - \int_0^\cdot \sigma^n(U^n) \dW$ is bounded in $L^p(\Omega; W^{1, 2}(0,t; L^2))$.\\
The boundedness follows from the definition of the norm and the estimate \eqref{eq:nonlinearNorm} and the bounds from previous steps thanks to the assumption $q\geq 4$.
\end{proof}

\begin{remark}
By interpolation inequality \eqref{eq:cltinterpolation} we immediately observe that $U^n$ is bounded in the space $L^q(\Omega;L^\infty(0,t;L^\infty_zL^4_{xy}))$.
%The compactness argument relies on the compact embedding $H^1_zH^1_{xy} \hook \hook L^\infty_zL^4_{xy}$.
%\begin{align*}
%H^1_zH^1_{xy} \cap L^2_zH^2_{xy} \hookrightarow H^{1-\alpha}_zH^{1+\alpha}_{xy}
%\end{align*}
\end{remark}

\subsection{Convergence of finite-dimensional approximations}
\label{sect:compactness}

The convergence of Galerkin approximations is established in a similar manner as in \cite{Debussche2011}. We will thus only briefly summarize the argument and provide details only in the parts where the absence of vertical dissipation plays an important role.

We first establish the existence of martingale solutions.

\begin{proposition}
\label{prop:martingale_existence}
Let $U_0\in L^q(\Omega; H^1)$ with $v_0\in L^q(\Omega; L^2_zH^1_{D,xy}\cap \lso)$ for some $q \geq 4$. Then there exists a global martingale solution of the modified equation \eqref{eq:primeqcutoff}.
\end{proposition}

\begin{proof}
Let $\cU_0$ be a Hilbert space such that the embedding $\cU \hook \cU_0$ is Hilbert-Schmidt. Let
\begin{align*}
	\cX_U &= L^2(0, t; L^2_z H^1_{xy} \cap L^\infty_z L^4_{xy}) \cap C([0, t], H^{-1}(M)),\\
	\cX_W &= C([0, t], \cU_0), \qquad \cX = \cX_U \times \cX_W.
\end{align*}
Let $\mu_U^n$, $\mu_W^n$ and $\mu^n$ be the probability measures on $\cX_U$, $\cX_W$ and $\cX$, respectively, defined by
\[
	\mu_U^n(\cdot) = \PP(U^n \in \cdot), \qquad \mu_W^n(\cdot) = \PP( W \in \cdot), \qquad \mu = \mu_U^n \times \mu_W^n.
\]
Recall that the embeddings
\begin{gather*}
	L^2(0, t; L^2_z H^2_{xy} \cap H^1_z H^1_{xy}) \cap W^{1/4, 2}(0, t; L^2) \hook \hook L^2(0, t; L^2_z H^1_{xy}\cap L^\infty_z L^4_{xy}),
\end{gather*}
and
\begin{gather*}
	W^{\alpha, p}(0, t; L^2) \hook \hook C([0, t]; H^{-1})
\end{gather*}
are compact by Lemma \ref{lem:BochnerCompactness} a) and b) if $\alpha p>1$, respectively. We may follow the argument of \cite[Lemma 4.1]{Debussche2011} and use the Prokhorov theorem to establish that the sequence of measures $\mu^n$ on $\cX$ is tight and therefore weakly compact.

By the Skorokhod theorem, there exists a probability space $(\hOmega, \hat{\cF}, \hPP)$, an increasing sequence $(n_k)_k$ and $\cX$-valued random variables $(\Unk, \Wnk)$ and $(\hU, \hW)$ such that $(\Unk, \Wnk) \to (\hU, \hW)$ $\hPP$-a.s.\ in the space $\cX$. The processes $\Wnk$ are cylindrical Wiener processes w.r.t.\ the filtration $\hat{\bF}^{n_k} = \lbrace \hat{\cF}_t^{n_k} \rbrace_{t \geq 0}$, where $\hat{\cF}_t^{n_k}$ is the completion of $\sigma(\lbrace \Unk(s), \Wnk(s) \mid 0 \leq s \leq t \rbrace)$. Moreover, by the Bensoussan argument from \cite[Section 4.3.4]{Bensoussan1995}, the couple $(\Unk, \Wnk)$ solves the equation
\[
	\d\! \Unk + [\Ah \Unk + \theta(\Unk)B^{n_k}(\Unk, \Unk) + F^{n_k}(\Unk)] \dt = \sigma^{n_k}(\Unk) \d\!\Wnk
\]
with the initial condition
\[
	\Unk(0) = U^{n_k}_0.
\]

Similarly as in \cite[Section 7.1]{Debussche2011}, we can establish
\begin{equation}
	\label{eq:UnkBound}
	\begin{gathered}
	\Unk \in L^2\left( \hOmega; L^2(0, t; L^2_z H^2_{xy} \cap H^1_z H^1_{xy}) \cap L^{\infty}(0, t; H^1) \right),\\
	\Unk \harp \hat{U} \ \text{in} \ L^2\left( \hOmega; L^2(0, t; L^2_z H^2_{xy} \cap H^1_z H^1_{xy}) \right), 
	\end{gathered}
\end{equation}
and by the Vitali convergence theorem and the convergence $\hPP$-a.s.\ in $\cX$
\begin{equation}
	\label{eq:UnkConvergence}
	\Unk \to \hU \ \text{in} \ L^2\left(\hOmega; L^2(0, t; L^2_{z} H^1_{xy}\cap L^\infty_z L^4_{xy}) \right).
\end{equation}
In particular, by further thinning the sequence we may assume that
\begin{equation}
	\label{eq:UnkConvergenceAS}
	\| \Unk - \hU \|_{L^2_z H^1_{xy}}, 
	 \| \Unk - \hU \|_{L^\infty_z L^4_{xy}} \to 0 \quad \text{a.s.\ in} \ [0, t] \times \hOmega.
\end{equation}

The limiting process of \cite[Section 7.2]{Debussche2011} can be divided into four steps: First, one needs to establish convergence of the deterministic terms in the equation a.s.\ in $[0, t] \times \hOmega$. This is followed by showing that the deterministic terms converge in $L^p([0, t] \times \hOmega)$ is shown for $1 \leq p < 2$. Next, convergence of the stochastic term $\int_0^\cdot \sigma^n(\Unk) \d\!\Wnk$ in $L^2(\hOmega; L^2(0, t; L_2(\cU, L^2)))$ is established. Finally, one combines the previous steps and shows that one can take the limit of the variational formulation of the equation provided the set of test functions is sufficiently smooth and dense in $\ls \times L^2$.

The above argument from \cite{Debussche2011} carries over to this case almost completely with the following two exceptions - the a.s.\ convergence of the deterministic nonlinear term and convergence of gradient-dependent stochastic term. The latter has been established in \cite[Sectiom 3.4]{Brzezniak2020} and therefore it remains to show only the a.s.\ convergence.

Let $s \in [0, t]$ be fixed and let $U^\sharp \in D(\Ah \cap H^1_z L^4_{xy})$. Then
\begin{align*}
	&\left| \int_0^s \left< \theta(\Unk) B^{n_k}(\Unk, \Unk) - \theta(\hU)B(\hU, \hU), U^\sharp \right> \d\!r \right|\\
	&\quad \leq \left| \int_0^s \theta(\Unk) \left< B(\Unk, \Unk) - B(\hU, \hU), P_n U^\sharp \right> \d\!r \right|\\
	&\ \hphantom{\leq \ } + \left| \int_0^s \theta(\Unk)\left<B(\hU, \hU), Q_n U^\sharp \right> \d\!r \right| + \left| \int_0^s \left( \theta(\Unk) - \theta(\hU) \right) \left<B(\hU, \hU), U^\sharp \right> \d\!r \right|\\
	&\quad = I_1^n + I_2^n + I_3^n.
\end{align*}
Since $B$ is bilinear, we have
\[
	I_1^n \leq \int_0^s \left| \left< B(\Unk - \hU, \hU), P_n U^\sharp \right> \right| + \left| \left< B(\hU, \Unk - \hU),P_n U^\sharp \right> \right| \d\!r.
\]
Hence, by \eqref{eq:nonlinearSecond}, we deduce
\begin{align*}
	I_{1}^n &\leq c \int_0^s \left( \| \Unk - \hU \|_{L^\infty_z L^4_{xy}} + \| \nablah (\Unk - \hU) \|_{L^2} \right) \| \Unk \|_{H^1} \| U^\sharp \|_{H^1_z L^4{xy}} \d\!r\\
	&\leq c \left( \| \Unk - \hU \|_{L^2(0, s; L^\infty_z L^4_{xy})} + \| \Unk - \hU \|_{L^2(0, s; L^2_z H^1_{xy})} \right)\\
	&\hphantom{\leq \ } \cdot  \| \Unk \|_{L^2(0, s; H^1)}  \| U^\sharp \|_{H^1_z L^4_{xy}} \to 0
\end{align*}
We treat the term $I_2^n$ using the Poincar\'e inequality \eqref{eq:poincareAbstract} and the bound \eqref{eq:nonlinearFirst} as
\begin{align*}
	I_2^n &\leq c \| Q_n U^\sharp \|_{L^2_z H^1_{xy}} \int_0^s \| \hU \|_{H^1} \| \hU \|_{H^1_zH^1_{xy}} \ds\\
	&\leq c \| Q_n U^\sharp \|_{D(\Ah^{1/2})} \leq \frac{c}{\overline{\lambda}_n^{1/2}} \| U^\sharp \|_{D(\Ah)} \to 0,
\end{align*}
where $\overline{\lambda}_n$ and $Q_n$ are as in Proposition \ref{prop:poincare}. Regarding $I_3^n$, \eqref{eq:nonlinearSecond} and \eqref{eq:gginterpolation} yield
\[
	\E \int_0^t I_3^n(s) \ds \leq c \| U^\sharp \|_{H^1_z L^4_{xy}} \E \int_0^t \| \hU \|_{H^1}^2 \ds < \infty.
\]
By the Dominated convergence theorem with the convergence \eqref{eq:UnkConvergenceAS} and Lipschitz continuity of $\theta$, the estimate above implies $I_3^n \to 0$ a.s.\ in $[0, t] \times \hOmega$.
\end{proof}

\subsection{Uniqueness}
\label{sect:uniqueness}

To establish strong uniqueness of martingale solutions, we need additional regularity in the vertical direction which can be obtained provided the initial data is additionally in $H^2_{N,z}L^2_{xy} \times H^2_{D,z}L^2_{xy}$, which implies by the particular structure of the primitive equations that also the solution takes values in  $H^2_{N,z}L^2_{xy} \times H^2_{D,z}L^2_{xy}$.

The requirement of higher regularity introduces a minor technical drawback as we need a stronger cut-off to show the desired estimate. Let $\theta \in C^{\infty}(\R)$ be as in \eqref{eq:theta}. For $\mu>0$ fixed, we define
\begin{equation}
	\label{eq:cutoffmoreregular}
	\tilde{\theta}(U(s)) = \theta_\mu\left( \| U(s) \|_{H^1_zL^4_{xy}} \right).
\end{equation}
We are looking for a solution $U=(v,T)$ of the modified system
\begin{equation}
	\label{eq:primeqcutoffmoreregular}
	\d\! U + [\Ah U + \tilde{\theta}(U) B(U, U) + F(U)] \dt = \sigma(U) \dW, \quad U(0) = U_0.
\end{equation}

Approximated solutions $U^n$ for this system are constructed as in Section \ref{sect:galerkin}. Moreover, the estimates from Lemma \ref{lemma:galerkinestimate}
also hold for these approximations since the cut-off is stronger than the one considered in Lemma \ref{lemma:galerkinestimate}.

\begin{lemma}\label{lemma:galerkinestimatemoreregular}
Let the assumptions of Lemma \ref{lemma:galerkinestimate} hold. Additionally, let $\partial_{zz} U_0 \in L^q(\Omega; L^2)$, $\partial_{zz}  F_U \in L^q\left( \Omega; L^q(0, t; L^2) \right)$ and let $\sigma$ satisfy \eqref{eq:sigmaGrowthL2}-\eqref{eq:sigmaLipH2L2} and consider the sequence of approximating solutions $(U^n)_n$ of \eqref{eq:primeqcutoffmoreregular}. Then $\partial_{zz} U^n$ is bounded in
	\begin{equation}
		\label{eq:U_additional_regularity_in_z}
		L^q\left( \Omega; C([0, t], L^2) \right) \cap L^q\left( \Omega; L^2(0, t; L^2_z H^1_{xy}) \right).
	\end{equation}
\end{lemma}

\begin{remark}
Assuming the solution $U$ has the regularity from Definition \ref{def:pathwise_solution} and \eqref{eq:U_additional_regularity_in_z}, in particular the continuity in $H^2_z L^2_{xy}$ and $L^2_z H^1_{xy}$, one may use a variant of the mixed derivative theorem \cite[Proposition 3.2]{Meyries2012} to establish $U \in C([0, t], H^1_z H^{1/2}_{xy}) \subseteq C([0, t], H^1_z L^4_{xy})$. The argument of the cut-off function $\tilde{\theta}$ is therefore bounded on $[0, t]$.
\end{remark}

\begin{proof}
We only show $\partial_{zz} U^n\in L^q\left( \Omega; L^\infty(0, t; L^2) \right)$, the other conclusion follows as above.

We proceed similarly as in Step 3 of Lemma \ref{lemma:galerkinestimate}. To this end we apply It\^{o}'s formula and get because of \eqref{eq:primeqgalerkinzz}
\begin{equation*}
%\label{eq:galerkinuzz}
\begin{split}
\d\! &\|\partial_{zz}U^n\|_{L^2}^q +q \nu \|\partial_{zz} U^n\|_{L^2}^{q-2} \|(-\Deltah)^{1/2}\partial_{zz}  U^n\|_{L^2}^{2} \dt\\
&\leq -q \|\partial_{zz} U^n\|_{L^2}^{q-2} \left< \partial_{zz} U^n,\partial_{zz} F^n(U^n) + \tilde{\theta}(U^n) P_n \partial_{zz} B( U^n,U^n)\right> \dt\\
&\quad +\frac{q(q-1)}{2}  \|\partial_{zz} U^n\|_{L^2}^{q-2} \|\partial_{zz}\sigma^n (U^n)\|^2_{L_2(U,L^2)} \dt\\
&\quad +q \|\partial_{zz} U^n\|_{L^2}^{q-2} \left<\partial_{zz} U^n,\partial_{zz} \sigma^n(U^n) \dW \right>\\
&= \sum_{i=1}^3 I^n_i \dt + I^n_4 \dW.
\end{split}
\end{equation*}
We estimate the lower order linear term as
\begin{multline*}
	\int_{0}^{t} \left| I_1^n \right| \ds \leq \frac14 \sup_{s \in [0,t]} \| \partial_{zz} U^n\|_{L^2}^q + c \int_{0}^{t} \| \partial_{zz} U^n\|_{L^2}^q\ds\\
	+ c \left( \int_{0}^{t} \| \partial_{z} \nablah U^n \|_{L^2}^2 \ds \right)^{q/2}+ c \left( \int_{0}^{t} \|  \partial_{zz}F^n_U\|_{L^2}^2 \ds \right)^{q/2}
\end{multline*}
the correction terms as
\[
	I^n_3 \leq c\left( 1 + \| U^n \|_{H^2_zL^2_{xy}}^q \right) + \frac{q(q-1)}{2} \eta^2 \|\partial_{zz} U^n\|_{L^2}^{q-2} \|(-\Deltah)^{1/2}\partial_{zz} U^n\|_{L^2}^{2},
\]
and the stochastic term by the Burkholder-Davis-Gundy inequality \eqref{eq:bdg}
\begin{multline*}
	\E \sup_{s \in [0,t]} \left| \int_0^s I^n_4 \dW \right| \leq \frac14 \E \sup_{s\in[0,t]}\|\partial_{zz} U^n\|_{L^2}^{q} + c\E\int_{0}^{t} 1 + \|U^n\|_{H^2_zL^2_{xy}}^{q} \ds
\\
	+q^2c_{BDG}^2\eta^2 \E \int_{0}^{t} \|\partial_{zz} U^n\|_{L^2}^{q-2}\|(-\Deltah)^{1/2} \partial_{zz} U^n\|_{L^{2}}^{2}\ds.
\end{multline*}
The nonlinear term is dealt with by the estimate \eqref{eq:nonlineardzz}. For $\varepsilon>0$, we have 
\begin{align*}
	\left| I_2^n \right|
	&\leq c \tilde{\theta}(U^n)\|\partial_{zz}U^n\|_{L^2}^{q-3/2}
	\norm{U}_{H^1_zL^4_{xy}}\norm{U^n}_{H^2_zH^1_{xy}}^{3/2}\\
	&\leq \varepsilon \|\partial_{zz} U^n\|_{L^2}^{q-2}\norm{(-\Deltah)^{1/2} \partial_{zz} U^n}_{L^2}^2
	+c_\varepsilon\mu^{4}\|\partial_{zz} U^n\|_{L^2}^{q}.
\end{align*}
The estimates above lead to
\begin{multline*}
	\E\left[ \frac12 \sup_{s\in[0,t]}\|\partial_{zz} U^n\|_{L^2}^{q} + c(q, \nu, \varepsilon, \eta) \int_{0}^{t} \|\partial_{zz} U^n\|_{L^2}^{q-2} \|(-\Deltah)^{1/2} \partial_{zz} U^n\|_{L^2}^{2} \ds\right]\\
	\leq c \E\left[\|\partial_{zz} U^n(0)\|_{L^2}^{q} +1+\left( \int_{0}^{t} \|  \partial_{zz}F^n_U\|_{L^2}^2 \ds \right)^{q/2} +\left( \int_{0}^{r} \|  U^n \|_{H^1_zH^1_{xy}}^2 \ds \right)^{q/2} \right]\\
	+ c_\varepsilon  \E \int_{0}^{t}\|U^n\|_{H^2_zL^2_{xy}}^{q}+\mu^{4}\|\partial_{zz} U^n\|_{L^2}^{q}\ds
\end{multline*}
where $c(q, \nu, \varepsilon, \eta) = q[\nu_v - \varepsilon -\eta^2(q c_{BDG}^2 + \frac{q-1}{2})]$.
For $\varepsilon$ small enough, from Gronwall's inequality and the bound \eqref{eq:boundLpH1H1}, it follows
\begin{multline}\label{eq:estimategalerkinzz}
	\E\left[ \sup_{s\in[0,t]}\|\partial_{zz} U^n\|_{L^2}^{q} + \int_0^t \|\partial_{zz} U^n\|_{L^2}^{q-2} (\|(-\Deltah)^{1/2} \partial_{zz} U^n\|_{L^2}^{2}) \ds\right]\\
	\leq c \E \left[ \|\partial_{zz} U^n(0)\|_{L^2}^{q}+\|U^n(0)\|_{H^1}^{q} +1+\left( \int_{0}^{t} \| F^n_U\|_{H^2_zL^2_{xy}}^2 \ds \right)^{q/2} \right].
\end{multline}
\end{proof}

\begin{remark}
The existence of martingale solutions of the modified system \eqref{eq:primeqcutoffmoreregular} follows similarly as in Proposition \ref{prop:martingale_existence} using the compact embedding $H^2_z H^1_{xy} \hook \hook H^1_zL^4_{xy}$.
\end{remark}

We may now state the result on strong uniqueness of martingale solutions of the modified problem \eqref{eq:primeqcutoffmoreregular}.

\begin{proposition}
\label{prop:uniqueness}
Let $\partial_{zz}  F_U \in L^q\left( \Omega; L^q(0, t; L^2) \right)$. Let $(\cS, W, U_1)$ and $(\cS, W, U_2)$ be two global martingale solutions of the modified problem \eqref{eq:primeqcutoffmoreregular} over the same stochastic basis $\cS$ with the same cylindrical Wiener process $W$. Let $\Omega_0 = \lbrace U_1(0) = U_2(0) \rbrace$. Let also $\partial_{zz} U_1(0), \partial_{zz} U_2(0) \in L^2(\Omega; L^2)$. Then $U_1$ and $U_2$ are indistinguishable on $\Omega_0$ in the sense
\[
	\PP\left( \lbrace \mathds{1}_{\Omega_0} \left( U_1(t) - U_2(t) \right) = 0 \ \text{for all} \ t \geq 0 \rbrace \right) = 1.
\]
\end{proposition}

\begin{proof}
The following proposition is established similarly as \cite[Proposition 5.1]{Debussche2011} using stopping times
\begin{align*}
	\tau_n = \lbrace t \geq 0 \mid  \int_0^t 
	&\| U_1 \|_{H^2_z L^2_{xy}}^{2/3} \|  U_1\|_{L^2_zH^1_{xy}}^{2/3} \| U_1 \|_{H^2_z H^1_{xy}}^{2/3} \| U_1\|_{L^2_zH^2_{xy}}^{2/3}\\
	&+\|U_1 \|_{L^2}^{2/3}\|U_1 \|_{L^2_z H^2_{xy}}^{2/3} \| U_1 \|_{H^2_zH^1_{xy}}^{4/3} + \| U_1 \|_{H^1_z L^2_{xy}}^2 \| U_1 \|_{H^1_z H^1_{xy}}^2\\
	&+ \|U_2\|_{L^2_zH^2_{xy}}^{4/3}+\|U_2 \|_{H^2_zH^1_{xy}}^2 + \|  U_2 \|_{H^2_zL^2_{xy}}^{2}  \| U_2 \|_{H^2_zH^1_{xy}}^{2}
  \ds \geq n \rbrace,
\end{align*}
and the stochastic Gronwall lemma, see Lemma \ref{lemma:stoch_Gronwall}. From the the estimates \eqref{eq:nonlineardz}, \eqref{eq:nonlinearNorm} and \eqref{eq:estimategalerkinzz}, we observe $\tau_n\to\infty$ almost surely. To prove an estimate for the the difference of two solutions in the $H^1_z L^2_{xy}$-norm we need an estimate on $U_{zz}$ to control the stray terms since we cannot employ the cancellation property for differences of solutions and we don't have any vertical dissipation at hand to absorb the terms.
Defining $R=(v_R,T_R)=U_1-U_2$ and $\hat{R} = \one_{\Omega_0} R$, we get
\begin{multline*}
\d\! R + [\Ah R + \tilde{\theta}(U_1) B(U_1, U_1) + F(U_1)- \tilde{\theta}(U_2) B(U_2, U_2) -F(U_2)] \dt\\
= (\sigma(U_1)-\sigma(U_2)) \dW
\end{multline*}
For fixed $n \in \N$, we apply the It\^{o} formula to $\|R\|_{H^1_zL^2_{xy}}^2$ and estimate the trace term, which yields
\begin{align*}
&\d\!\|R\|_{H^1_zL^2_{xy}}^2 +2\nu \|(-\Deltah)^{1/2} R\|_{H^1_zL^2_{xy}}^{2} \dt\\
&\quad \leq  2\left< R, F(U_1)-F(U_2) \right>+2\left< \partial_z R, \partial_z (F(U_1)-F(U_2)) \right> \dt\\
&\qquad + 2\left<  R,\tilde{\theta}(U_1) B(U_1, U_1)-\tilde{\theta}(U_2) B(U_2, U_2) \right> \dt\\
&\qquad + 2\left< \partial_z R,\partial_z(\tilde{\theta}(U_1) B(U_1, U_1)-\tilde{\theta}(U_2) B(U_2, U_2)) \right> \dt\\
&\qquad + \|\sigma (U_1)-\sigma (U_2)\|^2_{L_2(U,H^1_zL^2_{xy})} \dt\\
&\qquad +2 \left< R,\sigma(U_1)-\sigma(U_2) \d\! W\right>+2 \left<\partial_z R,\partial_z(\sigma(U_1)-\sigma(U_2)) \d\! W\right>.
\end{align*}
For arbitrary $\Upsilon>0$ fixed and arbitrary stopping times $0\leq \tau_a\leq \tau_b\leq \tau_n \wedge \Upsilon$, we have
\begin{align*}
	\E \int_{\tau_a}^{\tau_b} &\one_{\Omega_0} \left|\left< R, F(U_1)-F(U_2) \right>+\left< \partial_z R, \partial_z (F(U_1)-F(U_2)) \right>  \right|\dt \\
	&\leq c \E \int_{\tau_a}^{\tau_b}\| \hat{R}\|_{H^1_zL^2_{xy}} \| \hat{R}\|_{H^1_zH^1_{xy}} \dt\\
	&\leq \frac{\varepsilon}{6} \E \int_{\tau_a}^{\tau_b}\| (-\Deltah)^{1/2} \hat{R}\|_{H^1_zL^2_{xy}}^2 \dt+c_\varepsilon\E \int_{\tau_a}^{\tau_b}\| \hat{R} \|_{H^1_zL^2_{xy}}^2\dt
	\end{align*}
	for all $\varepsilon>0$ and  
\begin{align*}
	\E \int_{\tau_a}^{\tau_b} &\one_{\Omega_0} \|\sigma (U_1)-\sigma (U_2)\|^2_{L_2(U,H^1_zL^2_{xy})} \dt \\
	&\leq c \E \int_{\tau_a}^{\tau_b} \|\hat{R}\|_{H^1_zL^2_{xy}}^2 \dt + \eta^2 \E \int_{\tau_a}^{\tau_b} \one_{\Omega_0} \|(-\Deltah)^{\frac12}\hat{R}\|_{H^1_zL^2_{xy}}^2 \dt
%	&=c \E \int_{\tau_a}^{\tau_b}  \|\hat R\|_{H^1_zL^2_{xy}}  \dt+ \frac{\varepsilon}{6} \E \int_{\tau_a}^{\tau_b} \|(-\Deltah)^{\frac12}\hat R\|_{H^1_zL^2_{xy}} \dt.
\end{align*}	
By the Burkholder-Davis-Gundy inequality \eqref{eq:bdg} and estimates similar to the ones in Lemma \ref{lemma:galerkinestimate}, we obtain
\begin{align*}
	&\E \sup_{t\in [\tau_a,\tau_b]} \left| \int_{\tau_a}^{t} \left<\hat{R},\sigma(U_1)-\sigma(U_2) \d\! W\right>+ \int_{\tau_a}^{t} \left<\partial_z \hat{R},\partial_z(\sigma(U_1)-\sigma(U_2)) \d\! W\right>  \right| \\
	&\quad \leq c\E  \left( \int_{\tau_a}^{\tau_b} \|\hat{R}\|_{H^1_zL^2_{xy}}^2 \|\sigma(U_1)-\sigma(U_2)\|_{H^1_zL^2_{xy}}^2 \dt\right)^{1/2} \\
	&\quad \leq \frac{1}{4}\sup_{t\in [\tau_a,\tau_b]} \|\hat{R}\|_{H^1_zL^2_{xy}}^2
	+c \gamma^2\E  \int_{\tau_a}^{\tau_b} \|(-\Deltah)^{\frac12}\overline R\|_{H^1_zL^2_{xy}}^2 \dt
	+c\E \int_{\tau_a}^{\tau_b} \|\hat{R}\|_{H^1_zL^2_{xy}}^2 \dt. 
\end{align*}
For the nonlinear part, we use the cancellation property to get
\begin{align*}
 &\left<  R,\tilde{\theta}(U_1) B(U_1, U_1)-\tilde{\theta}(U_2) B(U_2, U_2) \right>\\
 &\quad = \left<  R,(\tilde{\theta}(U_1)-\tilde{\theta}(U_2) )B(U_1, U_1) \right>+ \left<  R,\tilde{\theta}(U_2) (B(U_1, U_1)-B(U_2, U_2)) \right>\\
 &\quad = \left<  R,(\tilde{\theta}(U_1)-\tilde{\theta}(U_2) )B(U_1, U_1)  \right>+ \left<  R,\tilde{\theta}(U_2) B(U_1, R) \right> + \left<  R,\tilde{\theta}(U_2) B(R, U_2)\right>\\
 &\quad =(\tilde{\theta}(U_1)-\tilde{\theta}(U_2)) \left<  R, B(U_1, U_1) \right>+ \tilde{\theta}(U_2) \left<  R,B(R, U_2)\right>
\end{align*}
and similarly
\begin{align*}
	&\left< \partial_z R,\partial_z(\tilde{\theta}(U_1) B(U_1, U_1)-\tilde{\theta}(U_2) B(U_2, U_2)) \right>\\
	&\hspace{-0.5em}\quad = \left<  \partial_z R,(\tilde{\theta}(U_1)-\tilde{\theta}(U_2) )\partial_z B(U_1, U_1) \right>+ \left<  \partial_z R,\tilde{\theta}(U_2) \partial_z (B(U_1, U_1)-B(U_2, U_2)) \right>\\
	&\hspace{-0.5em}\quad = (\tilde{\theta}(U_1)-\tilde{\theta}(U_2) ) \left<  \partial_z R, \partial_zB(U_1, U_1)  \right>+ \tilde{\theta}(U_2) \left<  \partial_z R, \partial_z B(U_1, R) \right>\\
 	&\hspace{-0.5em}\quad \quad + \tilde{\theta}(U_2) \left< \partial_z R, \partial_z B(R, U_2)\right>\\
 	&\hspace{-0.5em}\quad = (\tilde{\theta}(U_1)-\tilde{\theta}(U_2) ) \left<  \partial_z R, \partial_zB(U_1, U_1)  \right> + \tilde{\theta}(U_2) \left< \partial_z R, B(\partial_z U_1, R) \right>\\
 	&\hspace{-0.5em}\quad \quad + \tilde{\theta}(U_2) \left< \partial_z R, \partial_z B(R, U_2)\right>.
\end{align*}
We concentrate on the estimates for the part of nonlinearity with vertical derivatives since the ones without vertical derivatives are easier to obtain. As above, we provide details only for the estimates for the temperature for simplicity of presentation. Using the anisotropic H\"{o}lder inequality \eqref{eq:anisotropichoelder} and the Ladyzhenskaya inequality, we get
\begin{align*}
&\left\| \partial_z B(v_R, T_2) \right\|_{L^2_zL^{4/3}_{xy}}\\
&\quad =\left\|\partial_z v_R\cdot \nablah T_2-\divh v_R \partial_z T_2+v_R \cdot \nablah \partial_z T_2+w(R)\partial_{zz} T_2 \right\|_{L^2_zL^{4/3}_{xy}}\\
&\quad \leq c\big( 
 \|\partial_z R \|_{L^2_z L^2_{xy}} \| \nablah U_2\|_{L^2_zL^4_{xy}}
+ \| \divh R\|_{L^\infty_z L^2_{xy}} \|\partial_z U_2 \|_{L^2_z L^4_{xy}}\\
&\quad\qquad +\|R \|_{L^2_z L^4_{xy}} \| \nablah \partial_z U_2|_{L^\infty_zL^2_{xy}}
+\|w(R)\|_{L^\infty_z L^2_{xy}} \| \partial_{zz} U_2 \|_{L^2_zL^4_{xy}}
\big)\\
&\quad  \leq c\big( 
\|R \|_{H^1_zL^2_{xy}}\| U_2\|_{L^2_zH^2_{xy}}
+ \|R\|_{L^\infty_z H^1_{xy}} \| \partial_z U_2 \|_{L^2_z L^4_{xy}}\\
&\quad\qquad +\|R \|_{L^2_z L^4_{xy}} \| U_2 \|_{H^2_zH^1_{xy}}
+\|\divh R\|_{L^2_z L^2_{xy}} \|U_2 \|_{H^2_zL^4_{xy}}
\big)\\
&\quad\leq c\big( 
\|R \|_{H^1_zL^2_{xy}}\|U_2\|_{L^2_zH^2_{xy}}
+ \|R\|_{H^1_z H^1_{xy}} \| U_2 \|_{H^1_zL^2_{xy}}^{1/2}\|U_2 \|_{H^1_zH^1_{xy}}^{1/2}\\
&\quad\qquad+ \|R\|_{L^2}^{1/2}\| R \|_{L^2_z H^1_{xy}}^{1/2} \|U_2 \|_{H^2_zH^1_{xy}}
+ \|R\|_{L^2_z H^1_{xy}} \| v \|_{H^2_zL^2_{xy}}^{1/2}\|U_2 \|_{H^2_zH^1_{xy}}^{1/2}
\big)\\
&\quad\leq c\big( 
\|R \|_{H^1_zL^2_{xy}}\|U_2\|_{L^2_zH^2_{xy}}
+ \|R\|_{H^1_z H^1_{xy}} \|v \|_{H^2_zL^2_{xy}}^{1/2}\|U_2 \|_{H^2_zH^1_{xy}}^{1/2}\\
&\quad\qquad+ \|R\|_{L^2}^{1/2}\| R \|_{L^2_z H^1_{xy}}^{1/2} \|U_2 \|_{H^2_zH^1_{xy}}
\big)
\end{align*}
which yields
\begin{align*}
&\left| \tilde{\theta}(U_2) \left< \partial_z  R,\partial_z B(R, U_2)\right> \right| \leq \|\partial_z R \|_{L^2_zL^4_{xy}}\left\| \partial_z B(R, U_2) \right\|_{L^2_zL^{4/3}_{xy}}\\
&\quad\leq \|R\|_{H^1_zL^2_{xy}}^{1/2} \|R\|_{H^1_z H^1_{xy}}^{1/2}\left\| \partial_z B(R, U_2) \right\|_{L^2_zL^{4/3}_{xy}}
\\
&\quad\leq 
c \|R\|_{H^1_zL^2_{xy}}^{3/2} \|R\|_{H^1_z H^1_{xy}}^{1/2} \|U_2\|_{L^2_zH^2_{xy}}
+c \|R\|_{H^1_zL^2_{xy}} \| R \|_{H^1_z H^1_{xy}} \|U_2 \|_{H^2_zH^1_{xy}}\\
&\quad\qquad+c \|R\|_{H^1_zL^2_{xy}}^{1/2} \|R\|_{H^1_z H^1_{xy}}^{3/2} \| U_2 \|_{H^2_zL^2_{xy}}^{1/2}  \| U_2 \|_{H^2_zH^1_{xy}}^{1/2}\\
&\quad\leq \frac{\varepsilon}{6}  \|(-\Deltah)^{1/2}R\|_{H^1_zL^2_{xy}}^{2}\\
&\quad\qquad+c_\varepsilon\big(\|U_2\|_{L^2_zH^2_{xy}}^{4/3}+\|U_2 \|_{H^2_zH^1_{xy}}^2 + \|  U_2 \|_{H^2_zL^2_{xy}}^{2}  \| U_2 \|_{H^2_zH^1_{xy}}^{2}\big) \| R \|_{H^1_zL^2_{xy}}^{2}.
\end{align*}
Similarly as above, we deduce
\begin{align*}
	&\left\| (\partial_z v_1, T_R) \right\|_{L^2_zL^{4/3}_{xy}} = \left\| \partial_z v_1 \cdot \nablah T_R + \divh v_1 \partial_z T_R \right\|_{L^2_z L^{4/3}_{xy}}\\
	&\quad \leq \| \partial_z v_1 \|_{L^2_z L^4_{xy}} \| \nablah T_R \|_{L^\infty_z L^2_{xy}} + \| \nablah v_1 \|_{L^\infty_z L^2_x} \| \partial_z T_R \|_{L^2_z L^{4/3}_{xy}}\\
	&\quad \leq c \big( \| U_1 \|_{H^1_z L^2_{xy}}^{1/2} \| U_1 \|_{H^1_z H^1_{xy}}^{1/2} \| R \|_{H^1_z H^1_{xy}}\\
	&\quad\quad + \| U_1 \|_{L^2_z H^1_{xy}}^{1/2} \| U_1 \|_{H^1_z H^1_{xy}}^{1/2} \| R \|_{H^1_z L^2_{xy}}^{1/2} \| R \|_{H^1_z H^1_{xy}}^{1/2} \big)
\end{align*}
which leads to
\begin{align*}
	&\left| \tilde{\theta}(U_2) \left< \partial_z  R, B(\partial_z U_1, T)\right> \right| \leq \|\partial_z R \|_{L^2_zL^4_{xy}}\left\| \partial_z B(\partial_z U_1, R) \right\|_{L^2_zL^{4/3}_{xy}}\\
	&\quad \leq \frac{\varepsilon}{6} \| R \|_{H^1_z H^1_{xy}}^2 + c_\varepsilon \| R \|_{H^1_z L^2_{xy}}^2 \| U_1 \|_{H^1_z L^2_{xy}}^2 \| U_1 \|_{H^1_z H^1_{xy}}^2.
\end{align*}
Furthermore,  
\begin{align*}
&\left\| \partial_z B(v_1, T_1) \right\|_{L^2}
 =\left\|\partial_z v_1\cdot \nablah T_1-\divh v_1 \partial_z T_1+v_1 \cdot \nablah \partial_z T_1+w(U_1)\partial_{zz} T_1 \right\|_{L^2}\\
&\hspace{-0.55em}\quad \leq c\big( \|\partial_z U_1 \|_{L^\infty_z L^4_{xy}} \| \nablah U_1\|_{L^2_zL^4_{xy}}
+ \| \divh U_1 \|_{L^2_zL^4_{xy}}\|\partial_z U_1 \|_{L^\infty_z L^4_{xy}} \\
&\hspace{-0.55em}\quad\qquad  +\|U_1 \|_{L^2_z L^\infty_{xy}} \| \nablah \partial_z U_1 \|_{L^\infty_zL^2_{xy}}
+\|w(U_1)\|_{L^\infty_z L^4_{xy}} \| \partial_{zz} U_1 \|_{L^2_zL^4_{xy}}\big)\\
&\hspace{-0.55em}\quad\leq c\big( 
\|\partial_z U_1 \|_{L^\infty_z L^2_{xy}}^{1/2}\|\partial_z U_1 \|_{L^\infty_z H^1_{xy}}^{1/2} \| \nablah U_1\|_{L^2_zL^2_{xy}}^{1/2} \| \nablah U_1\|_{L^2_zH^1_{xy}}^{1/2}\\
&\hspace{-0.55em}\quad\qquad+ \|U_1 \|_{L^2_zH^2_{xy}}^{1/2}\|U_1\|_{L^2_zH^1_{xy}}^{1/2} \|\partial_z U_1 \|_{L^\infty_z L^2_{xy}}^{1/2}\|\partial_z U_1 \|_{L^\infty_z H^1_{xy}}^{1/2}\\
&\hspace{-0.55em}\quad\qquad+\|U_1 \|_{L^2_z L^2_{xy}}^{1/2}\|U_1 \|_{L^2_z H^2_{xy}}^{1/2} \| \partial_z U_1 \|_{L^\infty_zH^1_{xy}}\\
&\quad\qquad+\|w(U_1)\|_{L^\infty_z L^2_{xy}}^{1/2}\|w(U_1)\|_{L^\infty_z H^1_{xy}}^{1/2}  \| \partial_{zz} U_1 \|_{L^2_zL^2_{xy}}^{1/2}\| \partial_{zz} U_1 \|_{L^2_zH^1_{xy}}^{1/2}\big)\\
%& \leq c\big( 
%\| U_1 \|_{H^2_z L^2_{xy}}^{1/2}\| U_1 \|_{H^2_z H^1_{xy}}^{1/2} \|  U_1\|_{L^2_zH^1_{xy}}^{1/2} \| U_1\|_{L^2_zH^2_{xy}}^{1/2}\\
%&\qquad\quad+ \|U_1 \|_{L^2_zH^2_{xy}}^{1/2}\|U_1\|_{L^2_zH^1_{xy}}^{1/2} \|U_1 \|_{H^2_z L^2_{xy}}^{1/2}\|U_1 \|_{H^2_z H^1_{xy}}^{1/2}\\
%&\qquad\quad \ +\|U_1 \|_{L^2_z L^2_{xy}}^{1/2}\|U_1 \|_{L^2_z H^2_{xy}}^{1/2} \| U_1 \|_{H^2_zH^1_{xy}}\\
%&\qquad\quad+\|U_1\|_{L^2_z H^1_{xy}}^{1/2}\|U_1\|_{L^2_z H^2_{xy}}^{1/2}  \|  U_1 \|_{H^2_zL^2_{xy}}^{1/2}\|  U_1 \|_{H^2_zH^1_{xy}}^{1/2}\big)\\.
&\hspace{-0.55em}\quad\leq c\big( 
\| U_1 \|_{H^2_z L^2_{xy}}^{1/2}  \|  U_1\|_{L^2_zH^1_{xy}}^{1/2}  \| U_1 \|_{H^2_z H^1_{xy}}^{1/2}\| U_1\|_{L^2_zH^2_{xy}}^{1/2}
+\|U_1 \|_{L^2}^{1/2}\|U_1 \|_{L^2_z H^2_{xy}}^{1/2} \| U_1 \|_{H^2_zH^1_{xy}}\big).
\end{align*}
Hence, by the Lipschitz continuity of the cut-off $\tilde{\theta}$, we have
\begin{align*}
&|(\tilde{\theta}(U_1)-\tilde{\theta}(U_2)) \left< \partial_z R, \partial_z B(U_1, U_1) \right>| \leq c \|R\|_{H^1_z L^4_{xy}} \|\partial_z  R \|_{L^2} \| \partial_z B(U_1, U_1) \|_{L^2}\\
&\quad\leq c \|R\|_{H^1_z L^2_{xy}}^{3/2}\|R\|_{H^1_z H^1_{xy}}^{1/2}\big( 
\| U_1 \|_{H^2_z L^2_{xy}}^{1/2}\| U_1 \|_{H^2_z H^1_{xy}}^{1/2} \|  U_1\|_{L^2_zH^1_{xy}}^{1/2} \| U_1\|_{L^2_zH^2_{xy}}^{1/2}\\
&\quad\qquad+\|U_1 \|_{L^2}^{1/2}\|U_1 \|_{L^2_z H^2_{xy}}^{1/2} \| U_1 \|_{H^2_zH^1_{xy}}\big)\\
&\quad\leq \frac{\varepsilon}{6} \|(-\Deltah R)^{1/2}\|_{H^1_zL^2_{xy}}^2 +c_\varepsilon\big( 
\| U_1 \|_{H^2_z L^2_{xy}}^{2/3} \|  U_1\|_{L^2_zH^1_{xy}}^{2/3} \| U_1 \|_{H^2_z H^1_{xy}}^{2/3} \| U_1\|_{L^2_zH^2_{xy}}^{2/3}\\
&\quad\qquad +\|U_1 \|_{L^2}^{2/3}\|U_1 \|_{L^2_z H^2_{xy}}^{2/3} \| U_1 \|_{H^2_zH^1_{xy}}^{4/3}\big) \| R \|_{H^1_zL^2_{xy}}^2.
\end{align*}
We conclude
\begin{align*}
	&\E\left[ \frac12 \sup_{s\in[\tau_a,\tau_b]}\| \hat{R} \|_{H^1_zL^2_{xy}}^{2} 
	+ c(\nu, \varepsilon, \gamma) \int_{\tau_a}^{\tau_b}\|(-\Deltah)^{1/2}  \hat{R}\|_{H^1_zL^2_{xy}}^{2} \ds\right]
	\leq c \E \|  \hat{R}(\tau_a)\|_{H^1_zL^2_{xy}}^{2}\\
	&\quad+c_\varepsilon \E\bigg[\int_{\tau_a}^{\tau_b}\big( 
	1+\| U_1 \|_{H^2_z L^2_{xy}}^{2/3} \|  U_1\|_{L^2_zH^1_{xy}}^{2/3} \| U_1 \|_{H^2_z H^1_{xy}}^{2/3} \| U_1\|_{L^2_zH^2_{xy}}^{2/3}\\
 &\quad\qquad\qquad+\|U_1 \|_{L^2}^{2/3}\|U_1 \|_{L^2_z H^2_{xy}}^{2/3} \| U_1 \|_{H^2_zH^1_{xy}}^{4/3} +\|U_2\|_{L^2_zH^2_{xy}}^{4/3}+\|U_2 \|_{H^2_zH^1_{xy}}^2 \\
&\quad\qquad\qquad+ \| U_1 \|_{H^1_z L^2_{xy}}^2 \| U_1 \|_{H^1_z H^1_{xy}}^2 + \|  U_2 \|_{H^2_zL^2_{xy}}^{2}  \| U_2 \|_{H^2_zH^1_{xy}}^{2}\big) \| \hat{R} \|_{H^1_zL^2_{xy}}^2 \dt\bigg],
\end{align*}
where $c(\nu, \varepsilon, \gamma) = 2[\nu - \varepsilon -\gamma^2(2 c_{BDG}^2 + \frac{1}{2})]$.
From this, the claim follows by applying the stochastic Gronwall lemma, see Lemma \ref{lemma:stoch_Gronwall}, and recalling that $\Upsilon$ was arbitrary.
\end{proof}

\begin{remark}
\label{remark:uniqueness}
The last estimate in the above proof also identifies $U_1 = U_2$ in $L^2(0, \Upsilon; H^1_z H^1_{xy})$ for all $\Upsilon > 0$ $\PP$-almost surely. The identification therefore holds also in the space $\cX_U = L^2(0, \Upsilon; L^2_z H^1_{xy} \cap L^\infty_z L^4_{xy})$ from Proposition \ref{prop:martingale_existence} which justifies the use of the Gy\"{o}ngy-Krylov theorem in the proof of Theorem \ref{thm:maximalExistence} below.
\end{remark}

\subsection{Proof of Theorem \ref{thm:maximalExistence}}
\label{sect:maximalSolutions_proof}

The existence of global strong solutions of the modified problem \eqref{eq:primeqcutoffmoreregular} for initial data $U_0 \in L^q(\Omega; H^1 \cap H^2_z L^2_{xy})$ for some $q \geq 4$ follows similarly as in \cite[Section 5.2]{Debussche2011} from the Gy\"{o}ngy-Krylov theorem and the pathwise uniqueness established established in Proposition \ref{prop:uniqueness}, see also Remark \ref{remark:uniqueness}. Thus, we omit the proof.

In contrast to the approximations in \cite{Debussche2011}, the cut-off function in \eqref{eq:primeqcutoffmoreregular} does not start at $0$ and the cut-off is active from the initial time $t=0$. A localization procedure that also establishes existence of local solutions of the original equation \eqref{eq:primeqFunct} for initial data in $L^2(\Omega; H^1 \cap H^2_z L^2_{xy})$ is therefore required. For $n \in \N$, let
\[
	\Omega_n := \lbrace n-1 \leq \| U_0 \|_{H^1} + \| U_0 \|_{H^2_z L^2_{xy}} < n \rbrace, \qquad U_{0, n} = U_0 \mathds{1}_{\Omega_n},
\]
and let $U_n$ be the global pathwise solution of the equation \eqref{eq:primeqcutoffmoreregular} with cut-off $\tilde{\theta}$ with $\mu = 2c_{e}n$, where $c_e>0$ a fixed number such that $\| f \|_{H^1_z L^4_{xy}} \leq c_e( \| f \|_{H^2_z L^2_{xy}} + \| f \|_{L^2_z H^1_{xy}})$ for all $f \in H^2_z L^2_{xy} \cap L^2_z H^1_{xy}$, and initial data $U_{0, n} \in L^\infty(\Omega; H^1 \cap H^2_z L^2_{xy})$. Let $U = \sum_{n=1}^\infty \mathds{1}_{\Omega_n} U_n$.
Let us fix $M > 0$ and define the stopping times
\begin{gather*}
	\tau_n^{\mu} = \inf \left\lbrace t \geq 0 \mid \| U_n \|_{H^1} + \| U_n \|_{H^2_z L^2_{xy}} \geq 2n \right\rbrace,\\
	\tau_n^M = \inf \left\lbrace t \geq 0 \mid \sup_{s \in [0, t]}\| U_n \|_{H^1}^2 + \int_0^t \| (-\Deltah) U_n \|_{L^2}^2 \ds \geq M + \| U_{0, n} \|_{H^1}^2 \right\rbrace.
\end{gather*}
Let $\tau = \sum_{i=1}^\infty \mathds{1}_{\Omega_n} \left( \tau_n^\kappa \wedge \tau_n^M \right)$. Clearly, $\tau > 0$ $\PP$-almost surely. Since we assume that the filtration $\bF$ is right-continuous, $\tau$ is a stopping time, see \cite[pp.\ 6 and 7]{KaratzasShreve}. It is now straightforward to check that $(U, \tau)$ has the desired integrability and is indeed a local strong solution.

Existence of maximal solutions follows similarly as in \cite[Section 3.4]{Brzezniak2020}. This concludes the proof of Theorem \ref{thm:maximalExistence}.

\section{Global existence}\label{sec:globalexistence}

In the whole section, let $(U, \xi)$ be the maximal solution established in Theorem \ref{thm:maximalExistence} and let $U_0$, $\sigma$ and $F_v$ satisfy the assumptions of Theorem \ref{thm:globalExistence}. 

The proof of global existence combines the technique from \cite{DebusscheGlattholtzTemamZiane2012} with the use of the logarithmic Sobolev inequality from \cite{CaoLiTiti2017}. In particular, our goal is to find a sequence of stopping times $\rho_K$ satisfying $\rho_K \to \infty$ and $\rho_K \leq \xi$ for all $K \in \N$. In the process, we obtain an inequality of the form
\[
	f(t_1) \leq f(t_0) + \int_{t_0}^{t_1} \Vert v\Vert _{L^\infty} ^2 f(s) \ds,
\]
where $f$ contains certain Sobolev norms of the solution, see Lemma \ref{lemma:linfty}. To control $\Vert v\Vert_\infty$, we need the logarithmic Sobolev inequality. After that, we employ an argument similar to the logarithmic Gronwall inequality from \cite[Lemma 2.5]{CaoLiTiti2017}, see also the references therein.
%, see Proposition~\ref{prop_log_sob} below. The desired bound is then obtained by the classical Gr\"onwall lemma. \comm{We have to unify the spelling of Gronwall.} Applying the contradiction argument from \cite{DebusscheGlattholtzTemamZiane2012}, this implies global existence.

\begin{proposition}[Logarithmic Sobolev inequality]
\label{prop:logarithmic_Sobolev}
Let $p_1, p_2 \in (1, \infty)$ satisfy $\frac{2}{p_1}+\frac{1}{p_2}<1$. Then there exist $r_1, r_2 \in [2, \infty)$ such that for all $F \in L^{p_1}_z H^{1, p_1}_{0, xy}$ with $\partial_z F \in L^{p_2}$ we have
\begin{multline}
	\label{eq:log_Sobolev}
	\Vert F\Vert _{L^\infty }\le C_{p,\lambda} \max \left \{ 1,\max_{i = 1, 2}\frac{\Vert F\Vert _{L^{r_i}}}{{r_i}^\lambda}\right \}\\
	\cdot \log^\lambda \left ( e+\Vert F\Vert _{L^{p_1}}+\Vert \nablah F\Vert _{L^{p_1}}+\Vert F\Vert _{L^{p_2}}+\Vert \partial_z F\Vert _{L^{p_2}}\right ),
	\end{multline}
	for any $\lambda >0$ provided all the norms are finite.
\end{proposition}

Proposition \ref{prop:logarithmic_Sobolev} can be proved in exactly the same manner as \cite[Lemma A.1]{CaoLiTiti2017}, the only exception being using the maximum of several $L^p$ norms rather than $\sup_{p \in [2, \infty)} \| f \|_{L^p}$. The numbers $r_i$ can be computed explicitly. From the proof of \cite[Lemma A.1]{CaoLiTiti2017}, we observe $r_i = (q-1) \kappa_i$ for some $q \geq 3$, where
\[
	\alpha_i = \frac{1}{p_i}, \quad \kappa_i = \frac{p_i\left( 1 + \sum_{j=1}^2 \alpha_j \right)}{1 - \sum_{j=1}^2 \alpha_j}.
\]
Below, we apply the inequality with $p_1 = 6$ and $p_2 = 2$, choosing $q = 3$ we get
\[
	r_1 = (3-1)\kappa_1 = 132, \qquad r_2 = (3-1) \frac{22}{5} = \frac{44}{5}.
\]
Moreover, since $r_1>r_2$, we observe
\[
	\max \left\{ 1, \max_{i = 1, 2}\frac{\Vert F\Vert _{L^{r_i}}}{{r_i}^\lambda} \right\} \leq 1+\max_{i = 1, 2} \frac{\Vert F\Vert_{L^{r_i}}}{{r_i}^{\lambda}}\leq c \left(1+\frac{\Vert F\Vert_{L^{r_1}}}{{r_1}^{\lambda}} \right),
\]
and therefore \eqref{eq:log_Sobolev} reads as
\begin{align}\label{equation:sobolevlogarithmic}
	\Vert F\Vert_{L^\infty }\le &\, c\left(1+\Vert F\Vert_{L^{r_1}}\right) \log^{\lambda} (e+\Vert \nabla_H F \Vert_6+\Vert F \Vert _6+\Vert \partial_z F \Vert _{L^2}+\Vert F \Vert_{L^2}).
\end{align}

Let us start with the standard energy estimate. 

\begin{lemma}[$L^2$ bounds]
\label{lemma:global_L2}
Let $q \geq 2$, $U_0\in L^q(\Omega;L^2(M))$ and assume that
 $F_U\in L^q(\Omega;L^2_{loc}(0,\infty;L^2(M))$.
Then the stopping times $\tau^{w, q}_K$ defined for $K \in \N$ by
	\begin{align*}
		\tau_K^{w,q} = \inf \Bigg\{ s \geq 0 \mid  &\sup_{s \in [0, t \wedge \xi)} \Vert U \Vert_{L^2}^q + \int_0^{t \wedge \xi} \norm{U}_{L^2}^{q-2} \norm{U}_{L^2_z H^1_{xy}}^2 \ds\\
		&+ \left( \int_0^{t \wedge \xi} \Vert U \Vert_{L^2_z H^1_{xy}}^2 \ds \right)^{q/2} \geq K \Bigg\} 	
	\end{align*}
	satisfy $\tau_K^{w,q} \to \infty$ $\PP$-a.s.\ as $K \to \infty$.
\end{lemma}

Note that the assumptions of Theorem \ref{thm:globalExistence} are such that the above Lemma can be applied for $q\leq 16/3$.

\begin{proof}
Similarly as in Step 1 of Lemma \ref{lemma:galerkinestimate}, we deduce
\begin{multline*}
	\E \left[ \sup_{s \in [0, t \wedge \tau_N]} \Vert U \Vert_{L^2}^q + \int_0^{t \wedge \tau_N} \norm{U}_{L^2}^{q-2} \norm{U}_{L^2_z H^1_{xy}}^2 \ds \right]\\
	\leq C_t \E\left[ \Vert U_0 \Vert_{L^2}^q +1+ \left( \int_0^{t \wedge \tau_N} \norm{F_U}_{L^2}^2 \ds\right)^{q/2} \right].
\end{multline*}
By the monotone convergence theorem, we get
\begin{multline*}
	\E \left[ \sup_{s \in [0, t \wedge \xi)} \Vert U \Vert_{L^2}^q + \int_0^{t \wedge \xi} \norm{U}_{L^2}^{q-2} \norm{U}_{L^2_z H^1_{xy}}^2 \ds \right]\\
	\leq C_t \E\left[ \Vert U_0 \Vert_{L^2}^q +1+ \left(  \int_0^{t \wedge \xi} \norm{F_U}_{L^2}^2 \ds\right)^{q/2} \right].
\end{multline*}
The a.s.\ convergence to infinity of the auxiliary stopping time $\tau_K^1$ defined by
\[
	\tau_K^1 = \inf \left\{ s \geq 0 \mid \sup_{r \in [0, s \wedge \xi)} \Vert U \Vert_{L^2}^q + \int_0^{s \wedge \xi} \norm{U}_{L^2}^{q-2} \norm{U}_{L^2_z H^1_{xy}}^2 \dr \geq K \right\}
\]
follows by the first claim in Lemma \ref{lemma:auxiliary}. Hence, the second claim in Lemma \ref{lemma:auxiliary} with $\Psi(s) = s^{q/2}$ leads to the convergence of the auxiliary stopping time $\tau_K^2$ defined by
\[
	\tau_K^2 = \inf \left\{ s \geq 0 \mid \left( \int_0^{s \wedge \xi} \Vert U \Vert_{L^2_z H^1_{xy}}^2 \dr \right)^{q/2} \geq K \right\}.
\]
The proof is concluded by setting $\tau_K^{w, q} = \tau_K^1 \wedge \tau_K^2$.
\end{proof}

\subsection{Estimates by splitting}
Here we use the decomposition of the momentum equation into the barotropic and baroclinic modes $\overline{v}$ and $\widetilde{v}$ in \eqref{eq:bar} and \eqref{eq:tilde}, respectively, to establish estimates later leading to control of the $L^p$ norm of the solution.
Note that the baroclinic mode is mean-value free in the vertical direction and $\partial_z \widetilde{v} =\partial_z v$. Thus, the vertical Poincar\'e inequality
\begin{equation*}
	%	\label{eq:poincare.vertical}
	\int_{-h}^0 |\widetilde{v}|^q \dz \leq c_q \| \partial_z \widetilde{v} \|_{L^q(-h, 0)}^q  = c_q \| \partial_z v \|_{L^q(-h, 0)}^q
\end{equation*}
holds. This implies
\begin{equation}
	\label{eq:v_bound_tildebar}
	\| v\|_{L^{6}}	\leq \|\overline{v}\|_{L^{6}}+\| \widetilde{v}\|_{L^{6}}
	\leq c\left( \|\nablah \overline{v}\|_{L^{2}}+\|\partial_z v\|_{L^{6}}\right).
\end{equation}
It is straightforward to check that $\| \Ac v \|_{L^q(G)} \leq h^{-1/q} \| v \|_{L^q(M)}$ and $\| \Rc v \|_{L^q} \leq (1 + h^{-1/q}) \| v \|_{L^q}$.

We begin with an open estimate for $\overline{v}$ similarly as in \cite{HieberKashiwabara2016}. In the next step, we will combine the estimate with an estimate for $\widetilde{v}$ to obtain a useful bound.

\begin{lemma}
For any $N \in \N$, $\varepsilon > 0$ and any two stopping times $0\leq\tau_a\leq \tau_b\leq t \wedge \tau_N$, it holds
\begin{align}\label{eq:restimatevbarH2}
	\begin{split}
	\E & \left[ \frac12 \sup_{s\in[\tau_a, \tau_b]} \| (-\Deltah)^{1/2}  \overline{v}\|_{L^2}^{2} + c(\nu, \varepsilon,\eta) \int_{\tau_a}^{\tau_b}   \|\Deltah \overline{v}\|_{L^2}^{2} \ds\right]\\
	&\quad \leq  c\left( \E\|(-\Deltah)^{1/2}  \overline{v}(\tau_a)\|_{L^2}^{2} +1 \right)\\
	&\quad \hphantom{\leq \ } + c\E	\int_{\tau_a}^{\tau_b} (1+\norm{U}_{L^2}^2\norm{U}_{L^2_zH^1_{xy}}^2)
	\| (-\Deltah)^{1/2} \overline v \|^2_{L^2} \ds\\
	&\quad \hphantom{\leq \ } +c_\varepsilon \E \int_{\tau_a}^{\tau_b} \norm{U}^2_{L^2_zH^1_{xy}}+\norm{\overline{f_v}}^2_{L^2}+ \norm{|\widetilde{v}|\nablah \widetilde{v} }^2_{L^2}\ds,
	\end{split}
\end{align}
where $c(\nu,\varepsilon, \eta) = 2[\nu -\varepsilon- \eta^2(\frac{1}{2} + 2c_{BDG}^2)]$ and the constants $c, c_\varepsilon$ are independent of $N$, $\tau_a$ and $\tau_b$.
\end{lemma}

\begin{proof}
	 We apply the It\^{o} formula to $\norm{(-\Deltah)^{-1/2} P_G \cdot}^2_{L^2}$ and obtain
	\begin{equation*}
		%\label{eq:galerkinAU}
		\begin{split}
			\d & \|(-\Deltah)^{1/2} \overline{v}\|_{L^2}^2 + 2 \nu  \|\Deltah \overline{v}\|_{L^2}^{2} \dt\\
			&\leq 2\|(-\Deltah)^{1/2} \left< (-\Deltah)^{1/2}\overline{v}, (-\Deltah)^{1/2} P_G \left[\Ac F_v(U) -(\overline{v} \cdot \nablah \overline{v}) - N(\widetilde{v}) \right] \right> \dt \\
			&\hphantom{\leq \ } +2 \|(-\Deltah)^{1/2} \Ac \sigma_1(U)\|^2_{L_2(\cU,L^2)} \dt\\
			&\hphantom{\leq \ } +2\left< (-\Deltah)^{1/2} \overline{v},(-\Deltah)^{1/2} \Ac \sigma_1(U) \dW \right>,
		\end{split}
	\end{equation*}
	where we already used $P_G \overline{v} = \overline{v}$ and $P_G \Ac \sigma_1(U) = \Ac \sigma_1(U)$ from \eqref{eq:v_bar_div} and \eqref{eq:sigma_leray}, respectively. In \cite[Lemma 5.3]{HusseinSaalWrona}, it was shown that, for $\varepsilon>0$, we have
	\begin{multline*}
		\left| \left< (-\Deltah)^{1/2} \overline{v},(-\Deltah)^{1/2}(\overline{v} \cdot \nablah \overline{v})
		+ (-\Deltah)^{1/2}N(\widetilde{v}) \right> \right|
		\\
		\leq
		C_\varepsilon \left( \norm{\overline{v}}^2_{L^2}\norm{v}^2_{L^2_zH^1_{xy}} \|(-\Deltah)^{1/2} \overline{v}\|_{L^2}^2 + \norm{|\widetilde{v}|\nablah \widetilde{v} }^2_{L^2} \right)
		+\frac{\varepsilon}{2} \norm{\Deltah\overline{v}}^2_{L^2}
	\end{multline*}
	and
	\begin{multline*}
		\left| \left< (-\Deltah)^{1/2}\overline{v}, (-\Deltah)^{1/2} \Ac F_v(U)\right> \right|\\
		\leq C_\varepsilon \left( \norm{\nablah T}^2_{L^2}+\norm{v}_{L^2}^2+ \norm{\overline{f_v}}^2_{L^2} \right) + \frac{\varepsilon}{2}  \norm{\Deltah\overline{v}}^2_{L^2}.
	\end{multline*}
	By the sub-linear growth of $\Ac \sigma_1$ \eqref{eq:sigma_bar_growth} and the Burkholder-Davis-Gundy inequality \eqref{eq:bdg}, we deduce
	\begin{align*}
		&2\E \sup_{s \in [\tau_a, \tau_b]} \left|  \int_0^s \left< (-\Deltah)^{1/2} \overline{v},(-\Deltah)^{1/2} \Ac \sigma_1(U) \right> \dW \right|\\
		&\quad\leq \frac{1}{2} \E \sup_{s\in[\tau_a, \tau_b]}\| (-\Deltah)^{1/2}\overline{v}\|_{L^2}^{2}	+2c_{BDG}^2\eta^2 \E \int_{\tau_a}^{\tau_b}\|\Deltah \overline{v}\|_{L^{2}}^{2}\ds \\
		&\qquad + C \E\int_0^t 1 + \|(-\Deltah)^{1/2}  v\|_{L^2}^{2} \ds.
	\end{align*}
%	Also by  we have
%	\begin{align*}
%		\|(-\Deltah)^{1/2} \Ac \sigma_1(U)\|^2_{L_2(\cU,L^2)}
%		\leq \eta^2 \| \Deltah \overline{v}\|_{L^2}^2 + C\left(1+\|(-\Deltah)^{1/2}  v\|_{L^2}^{2}\right)
%	\end{align*}
	Using \eqref{eq:sigma_bar_growth} to deal with the correction term and recalling \eqref{lemma:global_L2} and $\norm{\nablah \overline{v}}_{L^2}\leq c \norm{\nablah v}_{L^2} \leq c\norm{U}_{L^2_zH^1_{xy}}$, the claim follows by collecting the above estimates.
\end{proof}

\begin{lemma}[$H^1$ bound for $\overline{v}$ and $L^4$ bound for $\widetilde{v}$]
\label{lemma:vtildevbar}
The stopping times $\tau_K^{\widetilde{v}, 4}$ and $\tau_K^{\overline{v}, 2}$ defined for $K \in \N$  by
\begin{align*}
	\tau_K^{\widetilde{v},4} &= \inf \left\{ s \geq 0 \mid \sup_{r\in[0, s \wedge \xi)} \Vert \widetilde{v}\Vert _{L^4}^4+ \int_{0}^{s \wedge \xi}  \norm{|\widetilde{v}| \nablah \widetilde{v}}_{L^2}^2 \dr \geq K \right\},\\
	\tau_K^{\overline{v}, 2} &= \inf \left\lbrace s \geq 0 \mid \sup_{r\in[0, s \wedge \xi)} \Vert (-\Deltah)^{1/2} \overline{v}\Vert _{L^2}^2 + \int_{0}^{s \wedge \xi}  \norm{\Deltah \overline{v}}_{L^2}^2 \ds \geq K \right\rbrace,
\end{align*}
satisfy $\tau_K^{\widetilde{v}} \to \infty$ $\PP$-a.s.\ as $K \to \infty$.
\end{lemma}

\begin{proof}
	Recalling the cancellation property of the nonlinear term, we apply the It\^{o} formula to $\Vert\widetilde{v}\Vert _{L^4}^4 $ and obtain
	\begin{equation*}
		\begin{split}
			\d\! \|\widetilde{v}\|_{L^4}^4 & +4\nu \left( \| |\widetilde{v}|\nablah \widetilde{v}\|_{L^2}^{2}+2\| |\widetilde{v}|\nablah |\widetilde{v}|\|_{L^2}^{2} \right)\dt\\
			&\leq -4\left< \vert \widetilde{v}\vert^2 \widetilde{v}, \Rc F_v(U) 
			- \widetilde{v}\cdot\nabla_H\overline{v}+N(\widetilde{v}) \right> \dt
			 +6 \sum_{k=1}^{\infty}\left< |\widetilde{v}|^2, (\Rc \sigma_1(U)e_k)^2 \right> \dt\\
			 & +4 \left< \vert \widetilde{v}\vert^2 \widetilde{v}, \Rc \sigma_1(U) \dW \right>
		\end{split}
	\end{equation*}
	almost surely. Integrating by parts, we get
	\begin{align*}
		4\left|\left< \vert \widetilde{v}\vert^2 \widetilde{v},\beta_T g \int_{\cdot}^{0}\nablah T(x,y,z')\dz'\right> \right|
		& \leq 
		c \norm{\int_{\cdot}^{0} T(x,y,z')\dz'}_{L^4}
		\norm{|\widetilde{v}|\nablah \widetilde{v}}_{L^2}
		\norm{\widetilde{v}}_{L^4}  \\
		& \leq c_\varepsilon\norm{T}_{L^2_zL^4_{xy}}^2
		\norm{\widetilde{v}}_{L^4}^2 +\varepsilon \norm{|\widetilde{v}|\nablah \widetilde{v}}_{L^2}^2\\
		& \leq c_\varepsilon \norm{T}_{L^2_zH^1_{xy}}^2
		\norm{\widetilde{v}}_{L^4}^2 + \varepsilon \norm{|\widetilde{v}|\nablah \widetilde{v}}_{L^2}^2
	\end{align*}
	for some $\varepsilon>0$. A similar estimate for the term with double vertical integral can be established in a similar manner. Hence, by the Young inequality, we have
	\begin{align*}
		4 &\int_{\tau_a}^{\tau_b} \left|\left< |\widetilde{v}|^2 \widetilde{v},\Rc F_v(U) \right>\right| \ds \\
		&\quad \leq \int_{\tau_a}^{\tau_b} c_\varepsilon \left( \norm{T}_{L^2_z H^1_{xy}}^2 \|\widetilde{v}\|_{L^4}^{2} + \norm{\widetilde{f_v}}_{L^4}\|\widetilde{v}\|_{L^{4}}^{3}+\|\widetilde{v}\|_{L^{4}}^{4} \right) +\varepsilon \norm{|\widetilde{v}|\nablah \widetilde{v}}_{L^2}^2\ds \\
		&\quad \leq \varepsilon \int_{\tau_a}^{\tau_b} \norm{|\widetilde{v}|\nablah \widetilde{v}}_{L^2}^2\ds + c_\varepsilon \int_{\tau_a}^{\tau_b} \left(1+\norm{T}_{L^2_zH^1_{xy}}^2\right)(1+\|\widetilde{v}\|_{L^4}^{4})\ds \\
		&\qquad + \frac14 \sup_{s\in[\tau_a,\tau_b]}\|\widetilde{v}\|_{L^4}^{4}  + c \left(\int_{\tau_a}^{\tau_b} \norm{\widetilde{f_v}}_{L^4}^2\ds\right)^2
	\end{align*}
	for any stopping times $0\leq \tau_a\leq \tau_b$. Following \cite[Proof of Theorem 1.1 in Section 6, Step 3, integrals $I_7$ and $I_8$]{HieberKashiwabara2016}, we obtain
	\begin{align*}
		4\left| \left< \vert \tilde{v}\vert^2 \tilde{v}, 
		- \tilde v\cdot\nabla_H\bar v+N(\tilde v) \right> \right| 
		\leq 
		c_\varepsilon\|\nablah\bar v\|_{L^2}^2 \|\tilde v\|_{L^4}^4 
		+\varepsilon\||\tilde{v}|\nablah \tilde{v}\|_{L^2}^{2}.
	\end{align*}
	Using \eqref{eq:sigma1_r}, \eqref{eq:h_k_growth} and \eqref{eq:noise_eta}, the boundedness of the operator $\Rc$ and the H\"older and Young inequalities, we get
	\begin{equation}
	\label{eq:l4_est_sum1}
	\begin{split}
		&6\sum_{k=1}^{\infty}\left< |\widetilde{v}|^2, (\Rc \sigma_1(U)e_k)^2 \right>\\
		&\quad = 6\sum_{k=1}^{\infty}\left< |\widetilde{v}|^2, (\Psi_k\cdot \nablah \widetilde{v}+ (\Rc \Phi_k) \cdot \nablah \overline{v}+\Rc h_k(v) )^2 \right>\\
		&\quad \leq 6\sum_{k=1}^{\infty}\left< |\widetilde{v}|^2, (1+\frac{\varepsilon}{3}) |\Psi_k\cdot \nablah \widetilde{v}|^2+c_\varepsilon |(\Rc \Phi_k) \cdot \nablah \overline{v}|^2+c_\varepsilon|\Rc h_k(v)|^2 \right>\\
		&\quad \leq (6\eta^2+2\varepsilon) \| |\widetilde{v}| \nablah \widetilde{v}\|_{L^2}^2
		+c_\varepsilon \left( | \widetilde{v}\|_{L^4}^2  \| \nablah \overline{v}\|_{L^4}^2
		+ \| \widetilde{v}\|_{L^4}^2  \| v \|_{L^4}^2+ \| \widetilde{v}\|_{L^4}^2 \right)\\
		&\quad \leq (6\eta^2+2\varepsilon) \| |\widetilde{v}| \nablah \widetilde{v}\|_{L^2}^2
		+c_\varepsilon \| \widetilde{v}\|_{L^4}^2  \left(  \| \nablah \overline{v}\|_{L^4}^2
		+ (\| \widetilde{v}\|_{L^4}+ \|  \overline{v}\|_{L^4})^2
		+ 1 \right) \\
		&\quad \leq (6\eta^2+2\varepsilon) \| |\widetilde{v}| \nablah \widetilde{v}\|_{L^2}^2
		+c_\varepsilon\| \widetilde{v}\|_{L^4}^2  \| \nablah \overline{v}\|_{L^2}\| \Deltah \overline{v}\|_{L^2}\\
		&\quad \quad +c_\varepsilon \| \widetilde{v}\|_{L^4}^2  \left(\| \widetilde{v}\|_{L^4}^2 + \|\nablah  \overline{v}\|_{L^2}^2 + 1\right)\\
		&\quad \leq (6\eta^2+2\varepsilon) \| |\widetilde{v}| \nablah \widetilde{v}\|_{L^2}^2
		+\frac{\overline{\varepsilon}}{2} \| \Deltah \overline{v}\|_{L^2}^{2}
		+c_{\varepsilon, \overline{\varepsilon}}(1+\| \nablah \overline{v}\|_{L^2}^2)(1+\| \widetilde{v}\|_{L^4}^4).
		\end{split}
	\end{equation}
	Similarly, we deduce
	\begin{equation}
	\label{eq:l4_est_sum2}
	\begin{split}
		&\sum_{k=1}^{\infty}|\left< |\widetilde{v}|^2\widetilde{v}, \Rc \sigma_1(U)e_k\right>|\\
		&\quad =\sum_{k=1}^{\infty}|\left< |\widetilde{v}|^2\widetilde{v}, \Psi_k\cdot \nablah \widetilde{v}+ (\Rc \Phi_k) \cdot \nablah \overline{v}+\Rc h_k(v) \right>| \\
		&\quad \leq \eta \| |\widetilde{v}| \nablah \widetilde{v}\|_{L^2} \| \widetilde {v}\|_{L^4}^2
		+c (\| \widetilde {v}\|_{L^4}^3\|\nablah \overline{v}\|_{L^4}+\|\widetilde{v}\|_{L^4}^3 \|v\|_{L^4} + \| \widetilde{v}\|_{L^4}^2)\\
		&\quad \leq \eta \| |\widetilde{v}| \nablah \widetilde{v}\|_{L^2} \| \widetilde {v}\|_{L^4}^2
		+c \| \widetilde {v}\|_{L^4}^3\|\nablah \overline{v}\|_{L^2}^{1/2}\|\Deltah \overline{v}\|_{L^2}^{1/2}\\
		&\quad \hphantom{\leq \ } +c \left( \|\widetilde{v}\|_{L^4}^4+\|\widetilde{v}\|_{L^4}^3 \|\nablah \overline v\|_{L^2}+1 \right).
	\end{split}
	\end{equation}
	Let $K, N \in \N$ be fixed and let $0 \leq \tau_a \leq \tau_b \leq t \wedge \tau_N \wedge \tau^{w, 4}_K$. By the Burkholder-Davis-Gundy inequality \eqref{eq:bdg} and \eqref{eq:l4_est_sum2}, we obtain
	\begin{align*}
		4&\E \sup_{s \in [\tau_a, \tau_b]}
		\left| \int_{\tau_a}^s \left<  \vert \widetilde{v}\vert^2 \widetilde{v}, \Rc \sigma_1(U) \right> \dW \right|\\
		&\leq 4 c_{BDG} \E \left( \int_{\tau_a}^{\tau_b} \left( \sum_{k=1}^{\infty}|\left< |\widetilde{v}|^2\widetilde{v}, \Rc \sigma_1(U)e_k\right>| \right)^2 \ds \right)^{1/2}\\
		&\leq 4 c_{BDG}\E \bigg(  \int_{\tau_a}^{\tau_b} \sqrt 2\eta^2 \| |\widetilde{v}| \nablah \widetilde{v}\|_{L^2}^2 \| \widetilde {v}\|_{L^4}^4
		+c \| \widetilde {v}\|_{L^4}^6\|\nablah \overline{v}\|_{L^2} \|\Deltah \overline{v}\|_{L^2}\\
		&\hphantom{\leq 4 c_{BDG} \E \bigg( \int_{\tau_a}^{\tau_b} \sqrt 2} + \|\widetilde{v}\|_{L^4}^8 + \|\widetilde{v}\|_{L^4}^6 \|\nablah \overline v\|_{L^2}^2 + 1 \ds \bigg)^{1/2}\\
		&\leq \frac12 \E \sup_{s\in[\tau_a, \tau_b]} \Vert \widetilde{v}\Vert_{L^4}^4 + 16c_{BDG}^2\eta^2 \E \int_{\tau_a}^{\tau_b} \| |\widetilde{v}| \nablah \widetilde{v}\|_{L^2}^2\ds\\
		&\quad +c \E \int_{\tau_a}^{\tau_b}  \| \widetilde {v}\|_{L^4}^2\|\nablah \overline{v}\|_{L^2}\|\Deltah \overline{v}\|_{L^2}+\| \widetilde{v}\|_{L^4}^4+\|\widetilde{v}\|_{L^4}^2 \|\nablah \overline v\|_{L^2}^2+1 \ds\\
		&\leq \frac12 \E \sup_{s\in[\tau_a, \tau_b]} \Vert \widetilde{v}\Vert_{L^4}^4 + 16c_{BDG}^2\eta^2\E  \int_{\tau_a}^{\tau_b} \| |\widetilde{v}| \nablah \widetilde{v}\|_{L^2}^2\ds+\frac{\overline{\varepsilon}}{2}\E  \int_{\tau_a}^{\tau_b} \|\Deltah \overline{v}\|_{L^2}^2\ds\\
		& \quad
		+c_{\overline{\varepsilon}} \E \int_{\tau_a}^{\tau_b}  (1+\|\nablah \overline{v}\|_{L^2}^2)(1+\|\widetilde{v}\|_{L^4}^4)\ds.
	\end{align*}
	Collecting the above, we deduce
	\begin{align*}
		%	\label{eq:vtilde_estimate}
		\E &\left[ \sup_{s\in[\tau_a, \tau_b]} \Vert \widetilde{v}\Vert_{L^4}^4 + 4\left( \nu - \varepsilon-\eta^2\left(4c_{BDG}^2+\frac32 \right)\right) \int_{\tau_a}^{\tau_b} \norm{|\widetilde{v}| \nablah \widetilde{v}}_{L^2}^2\ds \right]\\
		&\leq \E \left[ c \|\widetilde{v}(\tau_a)\|_{L^4}^{4} +\overline{\varepsilon} \int_{\tau_a}^{\tau_b} \|\Deltah \overline{v}\|_{L^2}^2\ds \right]\\
		&\hphantom{\leq \ } \quad +c_{\varepsilon,\overline{\varepsilon}} \E \int_{\tau_a}^{\tau_b}  (1+\|U\|_{L^2_zH^1_{xy}}^2)(1+\|\widetilde{v}\|_{L^4}^4)\ds +c\left(\int_{\tau_a}^{\tau_b}\norm{\widetilde{f_v}}_{L^4}^2\ds\right)^2.
	\end{align*}
	Next, multiply \eqref{eq:restimatevbarH2} by $b>0$ precisely determinetd later and add it to the estimate above to obtain 
	\begin{align*}
		&\E  \left[ \frac{b}{2} \sup_{s\in[\tau_a, \tau_b]} \| (-\Deltah)^{1/2}  \overline{v}\|_{L^2}^{2} + 2b\left(\nu - \widetilde{\varepsilon} - \eta^2\left( \frac12 + 2 c_{BDG}^{2} \right) \right) \int_{\tau_a}^{\tau_b}   \|\Deltah \overline{v}\|_{L^2}^{2} \ds\right]\\
		&\quad +\E \left[ \sup_{s\in[\tau_a, \tau_b]} \Vert \widetilde{v}\Vert_{L^4}^4 + 4\left(\nu - \varepsilon-\eta^2\left(4c_{BDG}^2+\frac32\right)\right) \int_{\tau_a}^{\tau_b} \norm{|\widetilde{v}| \nablah \widetilde{v}}_{L^2}^2\ds \right]\\
		&\leq  c b\E\left[\|(-\Deltah)^{1/2}  \overline{v}(\tau_a)\|_{L^2}^{2}
		+ \int_{\tau_a}^{\tau_b} (1+\norm{U}_{L^2}^2\norm{U}_{L^2_zH^1_{xy}}^2)
		\| (-\Deltah)^{1/2} \overline v \|^2_{L^2} \d\! s  \right]\\
		&\quad +c_{\widetilde{\varepsilon}} b \E \left[\int_{\tau_a}^{\tau_b} 1 +\norm{U}^2_{L^2_zH^1_{xy}}+\norm{\overline{f_v}}^2_{L^2}+ \norm{|\widetilde{v}|\nablah \widetilde{v} }^2_{L^2}\ds\right] 
		+\overline{\varepsilon}\E  \int_{\tau_a}^{\tau_b} \|\Deltah \overline{v}\|_{L^2}^2\ds \\
		&\quad + c \E \|\widetilde{v}(\tau_a)\|_{L^4}^{4}
		+c \E \int_{\tau_a}^{\tau_b}  (1+\|U\|_{L^2_zH^1_{xy}}^2)(1+\|\widetilde{v}\|_{L^4}^4)\ds +c\left(\int_{\tau_a}^{\tau_b}\norm{\widetilde{f_v}}_{L^4}^2\ds\right)^2.
	\end{align*}
	With $b=\frac{4\nu\varepsilon}{c_{\widetilde{\varepsilon}}}$ and $\overline{\varepsilon}=2\nu b \widetilde{\varepsilon}$, the above reads as
	\begin{align*}
		&\E  \left[ \sup_{s\in[\tau_a, \tau_b]} \| (-\Deltah)^{1/2}  \overline{v}\|_{L^2}^{2} + 2b\left(\nu - 2\widetilde{\varepsilon} - \eta^2\left( \frac12 + 2 c_{BDG}^{2} \right) \right) \int_{\tau_a}^{\tau_b}   \|\Deltah \overline{v}\|_{L^2}^{2} \ds\right]\\
		&\quad +\E \left[ \sup_{s\in[\tau_a, \tau_b]} \Vert \widetilde{v}\Vert_{L^4}^4 + 4\left(\nu - 2\varepsilon-\eta^2\left(4c_{BDG}^2-\frac32 \right)\right) \int_{\tau_a}^{\tau_b} \norm{|\widetilde{v}| \nablah \widetilde{v}}_{L^2}^2\ds \right]\\
		&  
		\leq c \E\left[\|(-\Deltah)^{1/2} \overline{v}(\tau_a)\|_{L^2}^{2}
		+\|\widetilde{v}(\tau_a)\|_{L^4}^{4} + \left(\int_{\tau_a}^{\tau_b}\norm{\widetilde{f_v}}_{L^4}^2\ds\right)^2+\int_{\tau_a}^{\tau_b} \| \overline{f_v} \|_{L^2}^2 \ds \right]\\
		&\quad + c \E
		\int_{\tau_a}^{\tau_b} 1+(1+\norm{U}_{L^2}^2\norm{U}_{L^2_zH^1_{xy}}^2)
		\| (-\Deltah)^{1/2} \overline v \|^2_{L^2} \ds\\
		&\quad +c \E \int_{\tau_a}^{\tau_b}\|U\|_{L^2_zH^1_{xy}}^2 + (1+\|U\|_{L^2_zH^1_{xy}}^2) \|\widetilde{v}\|_{L^4}^4 \ds,
	\end{align*}
	Choosing $\varepsilon$ and $\tilde{\varepsilon}$ sufficiently small, we may apply the stochastic Gronwall lemma \ref{lemma:stoch_Gronwall} to get
	\begin{align*}
		\E &\sup_{s\in[0, t \wedge \tau_K^{w, 4} \wedge \tau_N]} \Vert \widetilde{v}\Vert_{L^4}^4 + \E \sup_{s\in[0, t \wedge \tau_K^{w, 4} \wedge \tau_N]} \| (-\Deltah)^{1/2}  \overline{v}\|_{L^2}^{2}\\
		 &\quad + \E \int_{0}^{t \wedge \tau_K^{w, 4} \wedge \tau_N}
		 \norm{|\widetilde{v}| \nablah \widetilde{v}}_{L^2}^2+\|\Deltah \overline{v}\|_{L^2}^{2} \ds\\
		&\leq C_{t, K}\E \left[ \|\widetilde v(0)\|_{L^4}^{4}+\|(-\Deltah)^{1/2} \overline{v}(0)\|_{L^2}^{2} +1 + \int_{0}^{t \wedge \tau_K^{w, 4} \wedge \tau_N} \norm{U}_{L^2_zH^1_{xy}}^2\ds\right]\\
		&\quad + C_{t, K}\E \left[ \left(\int_{0}^{t \wedge \tau_K^{w, 4} \wedge \tau_N}\norm{\widetilde{f_v}}_{L^4}^2\ds\right)^2 + \int_0^{t \wedge \tau_K^{w, 4} \wedge \tau_N} \| \overline{f_v} \|_{L^2}^2 \ds \right].
	\end{align*}
	Clearly, the right-hand side of the above estimate is uniformly bounded w.r.t.\ $N$. Passing to the limit w.r.t.\ $N\to\infty$, we observe
	\begin{multline}
		\label{eq:vtilde_finiteness}
		 \sup_{s\in[0, t \wedge \tau_K^{w, 4} \wedge \xi)} \Vert \widetilde{v}\Vert_{L^4}^4 + \sup_{s\in[0, t \wedge \tau_K^{w, 4} \wedge \xi)} \| (-\Deltah)^{1/2}  \overline{v}\|_{L^2}^{2}\\
		+ \int_{0}^{t \wedge \tau_K^{w, 4} \wedge \xi}
		\norm{|\widetilde{v}| \nablah \widetilde{v}}_{L^2}^2+\|\Deltah \overline{v}\|_{L^2}^{2} \ds < \infty
	\end{multline}
	almost surely. Recalling $\tau_K^{w, 4} \to \infty$ a.s.\ as $K \to \infty$, from \eqref{eq:vtilde_finiteness} we deduce
	\begin{multline*}
	\sup_{s\in[0, t \wedge \xi)} \Vert \widetilde{v}\Vert _{L^4}^4 + \sup_{s\in[0, t \wedge \xi)} \| (-\Deltah)^{1/2}  \overline{v}\|_{L^2}^{2}\\
	+ \int_{0}^{t \wedge \xi}   \norm{|\widetilde{v}| \nablah \widetilde{v}}_{L^2}^2+\| (-\Deltah)^{1/2}  \overline{v}\|_{L^2}^{2}\ds < \infty
	\end{multline*}
	almost surely. The proof is concluded by Lemma \ref{lemma:auxiliary} similarly as in the proof of Lemma \ref{lemma:global_L2}.
\end{proof}

\begin{lemma}[improved $L^4$ bound for $\widetilde{v}$]
\label{lemma:vtilde_l4}
Let $q\geq 4$, $\widetilde{v_0}\in L^q(\Omega;L^4(M))$ and $\widetilde{f_v}\in L^q(\Omega;L^2_{loc}(0,\infty;L^4(M))$.
Then the stopping time $\tau_K^{\widetilde{v}, q}$ defined for $K \in \N$ and $q \in [4, \infty)$ by
\begin{align*}
	\tau_K^{\widetilde{v},q} = \inf \Bigg\{ s \geq 0 \mid &\sup_{r\in[0, s \wedge \xi)} \Vert \widetilde{v}\Vert _{L^4}^q + \int_{0}^{s \wedge \xi} \Vert \widetilde{v}\Vert _{L^4}^{q-4}  \norm{|\widetilde{v}| \nablah \widetilde{v}}_{L^2}^2 \dr\\
	&+ \left( \int_{0}^{s \wedge \xi} \norm{|\widetilde{v}| \nablah \widetilde{v}}_{L^2}^2\ds\right)^{q/2} \geq K \Bigg\}
\end{align*}
satisfies $\tau_K^{\widetilde{v},q} \to \infty$ $\PP$-a.s.\ as $K \to \infty$.
\end{lemma}

\begin{proof}
Recalling the cancellation property of the nonlinear term, the It\^{o} formula applied to $\Vert\widetilde{v}\Vert _{L^4}^q $ yields
\begin{equation*}
	\begin{split}
		\d\! \|\widetilde{v}\|_{L^4}^q & +q\nu \|\widetilde{v}\|_{L^4}^{q-4} \left( \| |\widetilde{v}|\nablah \widetilde{v}\|_{L^2}^{2}+2\| |\widetilde{v}|\nablah |\widetilde{v}|\|_{L^2}^{2} \right)\dt\\
		&\leq -q \|\widetilde{v}\|_{L^4}^{q-4}  \left< \vert \widetilde{v}\vert^2 \widetilde{v}, \Rc F_v(U) 
		- \widetilde{v}\cdot\nabla_H\overline{v}+N(\widetilde{v}) \right> \dt,\\
		&\quad +\frac{q(q-4)}{2} \|\widetilde{v}\|_{L^4}^{q-8} \left(\sum_{k=1}^{\infty}\left< |\widetilde{v}|^2\widetilde{v}, \Rc \sigma_1(U)e_k\right>\right)^2 \dt\\
		&\quad +\frac{3q}{2} \|\widetilde{v}\|_{L^4}^{q-4} \sum_{k=1}^{\infty}\left< |\widetilde{v}|^2, (\Rc \sigma_1(U)e_k)^2 \right> \dt +q \|\widetilde{v}\|_{L^4}^{q-4} \left< \vert \widetilde{v}\vert^2 \widetilde{v}, \Rc \sigma_1(U) \dW \right>
	\end{split}
\end{equation*}
almost surely. Nearly all the terms can be handled as in the previous proof, only the correction term requires further treatment. Due to Lemma \ref{lemma:vtildevbar}, we can now handle the mixed term $\|\Deltah \overline{v}\|_{L^2}^2 \|\widetilde{v}\|_{L^4}^{q}$.
%As above we have
%\begin{align*}
%	q\|\widetilde{v}\|_{L^4}^{q-4} &\left|\left< |\widetilde{v}|^2 \widetilde{v},\Rc F_v(U) \right>\right|\\
%	\leq& c  \left(1+\norm{T}_{L^2_zH^1_{xy}}^2+\norm{f_v}_{L^4}\right)(1+q\|\widetilde{v}\|_{L^4}^{q}) + \frac{q\varepsilon}{2}\|\widetilde{v}\|_{L^4}^{q-4} \norm{|\widetilde{v}|\nablah \widetilde{v}}_{L^2}^2.
%\end{align*}
%and
%\begin{align*}
%	\left| \left< \vert \tilde{v}\vert^2 \tilde{v}, 
%	- \tilde v\cdot\nabla_H\bar v+N(\tilde v) \right> \right| 
%	\leq 
%	C\|\nablah\bar v\|_{L^2}^2 \|\tilde v\|_{L^4}^4 
%	+\frac{\varepsilon}{2}\||\tilde{v}|\nablah \tilde{v}\|_{L^2}^{2}.
%\end{align*}
Recalling \eqref{eq:l4_est_sum1}, we have
\begin{equation*}
\sum_{k=1}^{\infty}\left< |\widetilde{v}|^2, (\Rc \sigma_1(U)e_k)^2 \right> 
% & \leq (\eta^2+\varepsilon) \| |\widetilde{v}| \nablah \widetilde{v}\|_{L^2}^2 +\| \Deltah \overline{v}\|_{L^2}^{2} +c(1+\| \nablah \overline{v}\|_{L^2}^2)(1+\| \widetilde{v}\|_{L^4}^4)\\
\leq (\eta^2+\varepsilon) \| |\widetilde{v}| \nablah \widetilde{v}\|_{L^2}^2
+c(1+\| \Deltah \overline{v}\|_{L^2}^2)(1+\| \widetilde{v}\|_{L^4}^4),
\end{equation*}
and, from \eqref{eq:sigma1_r}, \eqref{eq:h_k_growth} and \eqref{eq:noise_eta}, we deduce
\begin{align*}
	&\sum_{k=1}^{\infty}|\left< |\widetilde{v}|^2\widetilde{v}, \Rc \sigma_1(U)e_k\right>|\\
	&\quad \leq \eta \| |\widetilde{v}| \nablah \widetilde{v}\|_{L^2} \| \widetilde {v}\|_{L^4}^2\\
	&\hphantom{\quad \leq \ } +c (\| \widetilde {v}\|_{L^4}^3\|\nablah \overline{v}\|_{L^2}^{1/2}\|\Deltah \overline{v}\|_{L^2}^{1/2}+\|\widetilde{v}\|_{L^4}^4+\|\widetilde{v}\|_{L^4}^3 \|\nablah \overline v\|_{L^2}+1)\\
	&\quad \leq \eta \| |\widetilde{v}| \nablah \widetilde{v}\|_{L^2} \| \widetilde {v}\|_{L^4}^2
	+c (1+\|\Deltah \overline{v}\|_{L^2})(1+\|\widetilde{v}\|_{L^4}^4).
\end{align*}
%\comm{This is the point where we proceed different than before and need an estimate for $\|\Deltah \overline{v}\|_{L^2}^2$, if we do the same as above, we get the term  $\|\widetilde{v}\|_{L^4}^{q-4}\|\Deltah \overline{v}\|_{L^2}^2$ which cannot be absorbed on the left hand side.}
Hence,
\begin{align*}
	&\frac{q(q-4)}{2} \|\widetilde{v}\|_{L^4}^{q-8} \left(\sum_{k=1}^{\infty}\left< |\widetilde{v}|^2\widetilde{v}, \Rc \sigma_1(U)e_k\right>\right)^2
	+\frac{3q}{2} \|\widetilde{v}\|_{L^4}^{q-4} \sum_{k=1}^{\infty}\left< |\widetilde{v}|^2, (\Rc \sigma_1(U)e_k)^2 \right>\\
	& \leq \frac{q(q-1)}{2}\eta^2\|\widetilde{v}\|_{L^4}^{q-4} \| |\widetilde{v}| \nablah \widetilde{v}\|_{L^2}^2+c (1+\|\Deltah \overline{v}\|_{L^2}^2)(1+\|\widetilde{v}\|_{L^4}^q).
\end{align*}
%Let $K, N \in \N$ be fixed and let $0 \leq \tau_a \leq \tau_b \leq t \wedge \tau_N \wedge \tau^{w, 4}_K\wedge \tau^{\overline{v},2}$. By the Burkholder-Davis-Gundy inequality \eqref{eq:bdg}, we obtain
%\begin{align*}
%	q\E \sup_{s \in [\tau_a, \tau_b]}
%	&\left| \int_{\tau_a}^s\|\widetilde{v}\|_{L^4}^{q-4} \left<  \vert \widetilde{v}\vert^2 \widetilde{v}, \Rc \sigma_1(U) \right> \dW \right|\\
%\end{align*}
%Collecting the above, we deduce
%\begin{multline*}
%	\label{eq:vtilde_estimate}
%	\E \left[ \sup_{s\in[\tau_a, \tau_b]} \Vert \widetilde{v}\Vert_{L^4}^q + q[\nu - \varepsilon] \int_{\tau_a}^{\tau_b} \Vert \widetilde{v}\Vert _{L^4}^{q-4}  \norm{|\widetilde{v}| \nablah \widetilde{v}}_{L^2}^2\ds \right]\\
%	\leq .
%\end{multline*}
%By the stochastic Gronwall lemma, we get
%\begin{multline*}
%	\E \left[ \sup_{s\in[0, t \wedge \tau_K^{w, 2} \wedge \tau_N]} \Vert \%widetilde{v}\Vert_{L^4}^q + \int_{0}^{t \wedge \tau_K^{w, 2} \wedge \tau_N} \Vert \w%idetilde{v}\Vert _{L^4}^{q-4}  \norm{|\widetilde{v}| \nablah \widetilde{v}}_{L^2}^2\ds \right]\\
%	\leq
%\end{multline*}
Dealing with the remaining terms as in Lemma \ref{lemma:vtildevbar} and using the stochastic Gronwall lemma \ref{lemma:stoch_Gronwall}, we obtain
\[
	\sup_{s\in[0, t \wedge \xi)} \Vert \widetilde{v}\Vert _{L^4}^q + \int_{0}^{t \wedge \xi} \Vert \widetilde{v}\Vert_{L^4}^{q-4}  \norm{|\widetilde{v}| \nablah \widetilde{v}}_{L^2}^2\ds < \infty \quad \text{a.s.}
\]
The proof is concluded by Lemma \ref{lemma:auxiliary} similarly as in the proof of Lemma \ref{lemma:global_L2}.
\end{proof}

Analogously, adapting the estimates of the stochastic terms in Lemma \ref{lemma:vtildevbar}, we may also deduce higher integrability in the probability space for $\overline{v}$. The proof is omitted.

\begin{lemma}[improved $H^1$ bound for $\overline{v}$]
	Let $q\geq 2$, $\overline{v_0}\in L^q(\Omega;H^1_0(G))$ and $\overline{f_v}\in L^q(\Omega;L^2_{loc}(0,\infty;L^2(G))$.
Then the stopping time $\tau_K^{\nablah \overline{v}, q}$ defined for $K \in \N$ by
\begin{align*}
	\tau_K^{\nablah \overline{v}, q} = \inf \Bigg\lbrace s \geq 0 \mid &\sup_{r \in [0, s \wedge \xi)} \norm{\overline{v}}_{H^1}^q + \int_0^{s \wedge \xi} \norm{\nablah \overline{v}}_{L^2}^{q-2} \norm{\Deltah \overline{v}}^2_{L^2} \dr\\
	&+\left( \int_{\tau_a}^{\tau_b}\|\Deltah \overline{v}\|_{L^2}^{2}\ds\right)^{q/2} \geq K \Bigg\rbrace
\end{align*}
satisfies $\tau_K^{\nablah \overline{v}, q} \to \infty$ $\PP$-a.s.\ as $K \to \infty$.
\end{lemma}

\begin{remark}
In the above proofs, the choice $\varepsilon$ is not uniform w.r.t.\ $q \in [2, \infty)$. However, since we will use the above result only for a single value of $q=132$ (see the discussion below Proposition \ref{prop:logarithmic_Sobolev}), the argument goes through for $\eta^2$ sufficiently small.
\end{remark}

\subsection{Estimate for $\|U\|_{L^q}^p$}

First, we establish the following estimate on the pressure. We emphasize that the estimate is made possible by \eqref{eq:sigma_leray} as then there is no noise present in the equation for the pressure and the bound can be obtained using a standard (and essentially deterministic) argument.

\begin{lemma}[Pressure bounds]\label{lemma:pressure}
Let $\tau_K^{\nablah p}$ be the stopping time defined for $K \in \N$ by
\[
	\tau_K^{\nablah p} = \inf \left\{ s \geq 0 \mid \int_0^{s \wedge \xi} \norm{\nablah p_s}_{L^2}^2 \dr \geq K \right\}
\]
satisfies $\tau_K^{\nablah p} \to \infty$ $\PP$-a.s.\ as $K \to \infty$.
%\begin{align*}
%	\|\nablah p_s\|_{L^2}^2 \leq c \left( \|\Deltah \overline{v}\|_{L^2}^2
%	+\| \overline{v}\|_{H^2}^2\|\nablah \overline{v}\|_{L^2}^2 
%	+ \| |\widetilde{v}| \nablah \widetilde{v}\|_{L^2}^2 + \|U\|_{L^2_zH^1_{xy}}^2 + \| f \|_{L^2}^2 \right).
%\end{align*}
\end{lemma}

\begin{proof}
Let $K, N \in \N$. Applying $1-P_G$ to \eqref{eq:bar} and recalling \eqref{eq:sigma_leray}, we obtain
\begin{align*}
	 \frac{1}{\rho_0}\nablah p_s
	=(1-P_G) (\nu_v \Deltah \overline{v}-\overline{v} \cdot \nablah \overline{v} -N(\widetilde{v})+ \Ac F_v(U)).
\end{align*}
Using the bounds
\begin{gather*}
	\| \overline{v} \cdot \nablah \overline{v} \|_{L^2}^2 \leq c \| \Deltah \overline{v} \|_{L^2}^2 \| \nablah \overline{v} \|_{L^2}^2, \quad \|N(\widetilde{v})\|_{L^2}^2 \leq c \| |\widetilde{v}| \nablah \widetilde{v}\|_{L^2}^2,\\
	\| \Ac F_v(U) \|^2 \leq c\left( \| f \|_{L^2}^2 + \| U \|_{L^2_z H^1_{xy}}^2 \right),
\end{gather*}
integrating from $0$ to $s \wedge \tau_N \wedge \tau_K^{w, 2} \wedge \tau_K^{\widetilde{v}, 2} \wedge \tau_K^{\nablah \overline{v}, 2} \wedge \tau_K^{\nablah \overline{v}, 4}$, passing to the limit w.r.t.\ $N \to \infty$ and recalling the convergences of the stopping times established above, we get $	\int_0^{s \wedge \xi} \norm{\nablah p_s}_{L^2}^2 \dr < \infty$ for a.a.\ $s \geq 0$. The claim follows by Lemma \ref{lemma:auxiliary}.
\end{proof}

We are now in the position to prove the main estimate of this section.

\begin{proposition}[$L^q$ bound for $v$]
\label{prop:Lq_estimates}
Let $2\leq p \leq q < \infty$, $U_0 \in L^{2p}\left( \Omega; L^2_z H^1_{xy} \right) \cap L^p\left( \Omega; L^q \right)$ and $f \in L^{p}\left( \Omega; L^2(0, t; L^q) \right) $.
Then the stopping time $\tau_K^{v, q,p}$ defined for $K \in \N$ by
\[
	\tau_K^{v, q,p} = \inf \left\lbrace s \geq 0 \mid \sup_{r \in [0, s \wedge \xi)} \norm{v}_{L^q}^p \geq K \right\rbrace
\]
satisfies $\tau_K^{v, q,p} \to \infty$ $\PP$-a.s.\ as $K \to \infty$.
\end{proposition}
%(Due to the revised form of the logarithmic Gronwall inequality in Proposition \ref{prop:logarithmic_Sobolev}, it will be sufficient to consider stopping times of the form $\rho_K = \tau_K^{v, r_1} \wedge \tau_K^{v, r_2}$ which will deal with the maximum in \eqref{eq:log_Sobolev}. The constraint $\eta^2 q < 1$ (or some other constant) has therefore to be checked only twice and therefore should not be too limiting. 
\begin{proof}
Assume for the moment that
\begin{equation}
	\label{eq:q_boundedness}
	\E \sup_{s \in [0, t \wedge \tau_N]} \| U \|_{L^q}^p < \infty
\end{equation} 
holds for all $N \in \N$ and let $A(v) = 1+\| v \| _{L^q}^q$. The It\^{o} formula applied to $A(\cdot)^{p/q}$ and the cancellation property of the nonlinear term yields
\begin{align*}
	\begin{split}
		& \d\! A(v)^{p/q} +p\nu A(v)^{(p-q)/q} \left[\| |v|^{(q-2)/2} \nablah v\|_{L^2}^{2}+(q-2)\| |v|^{(q-2)/2} \nablah |v|\|_{L^2}^{2} \right] \dt\\
		&\leq -p A(v)^{(p-q)/q} \left< |v|^{q-2} v, F_v(U)- \nablah p_s \right> \dt\\
		&\quad + \frac{p}{2} (q-1)\sum_{k=1}^{\infty}A(v)^{(p-q)/q} \left< | v|^{q-2}, (\sigma_1 (U)e_k)^2 \right> \dt\\
		&\quad + \frac{p}{2}(p-q) \sum_{k=1}^\infty A(v)^{(p-2q)/q} \left< | v|^{q-2}v, \sigma_1 (U)e_k \right>^2\dt\\
		&\quad +p  A(v)^{(p-q)/q}\left<|v|^{q-2} v, \sigma_1 (U) \dW \right>.
	\end{split}
\end{align*}
Integrating by parts, we obtain
\begin{align*}
	 & \left|\left< |v|^{q-2} v,\beta_T g \int_{\cdot}^{0}\nablah T(x,y,z')\dz'\right> \right| \\
	&\qquad \leq c\norm{|v|^{(q-2)/2} \int_{\cdot}^{0} T(x,y,z')\dz}_{L^2}^2+ \varepsilon \| |v|^{(q-2)/2}\nablah v \|_{L^2}^2\\
	 &\qquad \leq c\norm{\int_{\cdot}^{0} T(x,y,z')\dz}_{L^{2q/(q+2)}}^2\norm{v}_{L^q}^{q-2} + \varepsilon \| |v|^{(q-2)/2}\nablah v \|_{L^2}^2\\ 
	 &\qquad \leq c\norm{U}_{L^2_zH^1_{xy}}^2(1+\norm{v}_{L^q}^{q}) +\nu \varepsilon \| |v|^{(q-2)/2}\nablah v \|_{L^2}^2.
\end{align*}
 Using $\norm{v}_{L^q}^{q'}\leq A(v)^{q'/q}$ for $q'\geq 0$, we get
 \begin{align*}
 A(v)^{(p-q)/q}|\left< |v|^{q-2} v,f_v\right>|
 & \leq A(v)^{(p-q)/q}\norm{v}_{L^q}^{q-1}\norm{f_v}_{L^q}\\
 & \leq A(v)^{(p-1)/q} \norm{f_v}_{L^q}.
 \end{align*}
Hence, for $0\leq \tau_a\leq \tau_b$ stopping times specified later, we obtain
\begin{align*}
	&\int_{\tau_a}^{\tau_b} A(v)^{(p-q)/q} |\left< |v|^{q-2} v,F_v(U) \right>| \ds\\
	& \leq \nu\varepsilon \int_{\tau_a}^{\tau_b} A(v)^{(p-q)/q}\| |v|^{(q-2)/2}\nablah v \|_{L^2}^2\ds+\frac14\sup_{s\in[\tau_a,\tau_b]}A(v)^{p/q}\\
	& \quad +c_\varepsilon \int_{\tau_a}^{\tau_b}(1+\norm{U}_{L^2_zH^1_{xy}}^2)	(1+A(v)^{p/q})\ds +c \left(\int_{\tau_a}^{\tau_b}\norm{f_v}_{L^q}^2\ds\right)^{p/2}.
\end{align*}
Before we estimate the pressure term, we recall that the surface pressure $p_s$ is independent of $z$ and that we may shift it by a constant so that $\int_G p_s(t,x,y) \d(x,y)=0$. Therefore, we get
\begin{align*}
	|\left< |v|^{q-2} v, \nablah p_s \right>| 
	& \leq (q-1) |\left< |v|^{q-2}  |\nablah v|, p_s \right>|\\
	& \leq (q-1) \| |v|^{(q-2)/2}\nablah v \|_{L^2 }\|v^{(q-2)/2}\|_{L^2_z L^{2q/(q-2)}_{xy} }\|p_s \|_{L^\infty_z L^{q}_{xy}}\\
	& \leq \nu(q-1) \| |v|^{(q-2)/2}\nablah v \|_{L^2 }^2+c_q\|\nablah p_s \|_{L^2}^2 \|v\|_{L^q }^{q-2},
\end{align*}
with $c_q \approx q^{3/2}$.
The second part of the correction term is non-positive and thus it can be dropped in the estimates. For the first part, we use \eqref{eq:noiseglobal}, \eqref{eq:h_k_growth} and \eqref{eq:noise_eta} to get
\begin{align*}
	&\sum_{k=1}^{\infty}\left< | v|^{q-2}, (\sigma_1 (U)e_k)^2 \right>\\
	&\quad=\sum_{k=1}^{\infty}\left< | v|^{q-2}, (\Psi_k\cdot \nablah v+(\Phi_k-\Psi_k)\cdot \nablah \overline{v}+h_k(v))^2 \right>\\
	&\quad \leq \sum_{k=1}^{\infty}
	\left< | v|^{q-2},(1+\frac{\varepsilon}{q-2}) |\Psi_k\cdot \nablah v|^2+c|(\Phi_k-\Psi_k)\cdot \nablah \overline{v}|^2+c|h_k(v)|^2 \right>\\
	&\quad \leq (1+\frac{\varepsilon}{q-1}) \eta^2 
	\| | v|^{(q-2)/2} \nablah v \|_{L^2}^2
	+c  \| | v|^{q-2} \|_{L^{q/(q-2)}} \| |\nablah  \overline{v}|^2| \|_{L^{q/2}}\\
	&\hphantom{\quad \leq \ } +c\| | v|^{q-2} \|_{L^{q/(q-2)}}(1+\||v|^2\|_{L^{q/2}})\\
	&\quad \leq (1+\frac{\varepsilon}{q-1}) \eta^2 
	\| | v|^{(q-2)/2} \nablah v \|_{L^2}^2
	+c (1+\| v \|_{L^{q}}^q)+c\| \Deltah  \overline{v}\|_{L^{2}}^2\| v \|_{L^{q}}^{q-2}.
\end{align*}
Hence, we obtain
\begin{multline*}
	 \frac{p}{2}(q-1) A(v)^{(p-q)/q} \sum_{k=1}^{\infty}\left< | v|^{q-2}, (\sigma_1 (U)e_k)^2 \right>\\
	\hspace{-0.3em}\leq  \frac{p}{2}(q-1+\varepsilon) \eta^2  A(v)^{(p-q)/q} 
	\| | v|^{(q-2)/2} \nablah  v \|_{L^2}^2
	+c  (A(v)^{p/q}+\| \Deltah  \overline{v}\|_{L^{2}}^2A(v)^{(p-2)/q}). 
\end{multline*}
For $K \in \N$, let $\Upsilon_K = \tau_K^{w, 2} \wedge \tau_K^{\nablah \overline{v}, 2} \wedge \tau_K^{\nablah p,2}$. Let $N \in \N$ and let $0 \leq \tau_a \leq \tau_b \leq t \wedge \tau_N \wedge \Upsilon_K$ be stopping times. With \eqref{eq:noiseglobal}, \eqref{eq:h_k_growth} and \eqref{eq:noise_eta} and the Young inequality, we deduce
\begin{align*}
	\sum_{k=1}^{\infty}& \left<  |v|^{q-1} |\sigma_1(U)e_k | \right>^2\\
	&\leq \left(  \eta \| | v|^{(q-2)/2} \nablah v \|_{L^2}  \|v \|_{L^{q}}^{q/2} + c \| \nablah \overline{v}\|_{L^{q}} \|v \|_{L^{q}}^{q-1} +c(1+\| v \|_{L^{q}}^{q})\right)^2\\
	&\leq (1+\frac{\varepsilon}{q-1}) \eta^2 \| | v|^{(q-2)/2} \nablah  v \|_{L^2}^2A(v) + c_q (A(v)^2+\| \Deltah  \overline{v}\|_{L^{2}}^2A(v)^{2-2/q}).
\end{align*}
Hence, employing the Burkholder-Davis-Gundy inequality \eqref{eq:bdg}, we get
\begin{align*}
	\E &\sup_{s \in [\tau_a,\tau_b]}
	\left| \int_{\tau_a}^s A(v)^{(p-q)/q} \left< |v|^{q-2} v, \sigma_1(U) \right> \dW \right|\\
	&\leq c_{BDG}\E \left( \int_{\tau_a}^{\tau_b} A(v)^{2(p-q)/q}  \sum_{k=1}^{\infty} \left<  |v|^{q-1} |\sigma_1(U)e_k | \right>^2 \ds\right)^{1/2}\\
%	&\leq c_{BDG}\E \left[\left(  \int_{\tau_a}^{\tau_b} 
%	A(v)^{(2p-2q)/q}
%	\left(\eta
%	\| | v|^{(q-2)/2} \nablah  \overline{v} \|_{L^2}  \|v \|_{L^{q}}^{q/2}
%	+c  \| \nablah \overline{v}\|_{L^{2}} \|v \|_{L^{q}}^{q-1}
%	+c(1+\| v \|_{L^{q}}^{q})\right)^2
%	\ds\right)^{1/2} \right]\\
%	&\leq c_{BDG}\E \left[\left(  \int_{\tau_a}^{\tau_b} 
%	A(v)^{(2p-2q)/q}\left(
%	(1+\frac{\varepsilon}{q-1}) \eta^2 
%	\| | v|^{(q-2)/2} \nablah  \overline{v} \|_{L^2}^2A(v)
%	+c  (1+\| \Deltah  \overline{v}\|_{L^{2}}^2)
%	A(v)^2\right)\ds\right)^{1/2} \right]\\
	&\leq c_{BDG} \E \bigg( \left(1+\frac{\varepsilon}{q-1}\right) \eta^2 \int_{\tau_a}^{\tau_b} A(v)^{p/q + (p-q)/q}\| | v|^{(q-2)/2} \nablah  v \|_{L^2}^2 \ds\\
	&\hphantom{\leq c_{BDG} \E \bigg( \ } + c_q \int_{\tau_a}^{\tau_b} 	A(v)^{2p/q}+\| \Deltah  \overline{v}\|_{L^{2}}^2	A(v)^{2(p-1)/q}\ds \bigg)^{1/2}\\
	&\leq \frac14 \E \sup_{s\in[\tau_a,\tau_b]}A(v)^{p/q}
	+c_q\E  \int_{\tau_a}^{\tau_b} 
	(1+\| \Deltah  \overline{v}\|_{L^{2}}^2)
	A(v)^{p/q}\ds\\
	& \quad +\left(1+\frac{\varepsilon}{q-1}\right) c_{BDG}^2\eta^2 \E \int_{\tau_a}^{\tau_b} 
	A(v)^{(p-q)/q}\| | v|^{(q-2)/2} \nablah v \|_{L^2}^2 \ds.
\end{align*}
The claim follows by collecting the above and the stochastic Gronwall lemma \ref{lemma:stoch_Gronwall} similarly as in the previous lemmata provided we establish \eqref{eq:q_boundedness}.

To that end, we consider the Galerkin approximations $U^n = (v^n, T^n)$ from Section 3. Since $\Phi_{m, k} \in L^q(M)$ for all $k \in \N_0$, $m \in \N$ and $q \in [2, \infty]$, one has $v^n \in L^\infty(0, t; L^q)$. Using similar estimates as above, boundedness of the operator $\Ac$ and the deterministic Gronwall lemma, we obtain
\begin{align*}
	\E \sup_{s \in [0, t]} A(v^n)^{p/q} &\leq  \E A(v^n(0))^{p/q} + c_{q}\\
	&\quad +c_q \E \left[ \sup_{s \in [0, t]} \| U^n \|_{L^2_z H^1_{xy}}^p + \left( \int_0^t \| \Deltah v^n \|_{L^2}^2 \ds \right)^{p/2}\right]\\
	&\quad + c_q \E \left[ \left( \int_0^t \| f_v^n \|_{L^q}^2 \ds \right)^{p/2}+  \left( \int_0^t  \| \nablah p_s^n \|_{L^2}^2 \ds \right)^{p/2}\right].
\end{align*}
Since $U^n$ in $L^p(\Omega; L^2(0, t; L^2_z H^2_{xy}))$ and $L^p(\Omega; L^\infty(0, t; H^1))$ can be controlled by Lemma \ref{lemma:galerkinestimate}, it remains to deal with the pressure term. First, we observe that $H^1_z H^1_{xy} \cap L^2_z H^2_{xy} \subseteq H^{3/4}_z H^{5/4}_{xy} \subseteq L^\infty$ by the mixed derivative theorem \cite[Proposition 3.2]{Meyries2012}. Thus, since the operator $\Rc$ is bounded, we have
\begin{align*}
	\int_0^t \| &|\widetilde{v}^n| \nablah \widetilde{v}^n \|_{L^2}^2 \ds \leq c \sup_{s \in [0, t]} \| v^n \|_{L^2_z H^1_{xy}}^2 \int_0^t \| v^n \|_{L^\infty}^2 \ds\\
	&\leq c \sup_{s \in [0, t]} \| v^n \|_{L^2_z H^1_{xy}}^4 + c \left( \int_0^t \| v^n \|_{L^2_z H^2_{xy}}^2 \ds \right)^2 + c \left( \int_0^t \| v^n \|_{H^2_z L^2_{xy}}^2 \ds \right)^2.
\end{align*}
The estimates in the proof of Lemma \ref{lemma:pressure} and boundedness of $\Ac$ now imply
\begin{align*}
	\int_0^t \| \nablah p_s^n \|_{L^2}^2 \ds &\leq c \left( 1 + \sup_{s \in [0, t]} \| U^n \|_{L^2_z H^1_{xy}}^4 \right) + c\left( \int_0^t \| v^n \|_{L^2_z H^2_{xy}}^2 \ds \right)^2\\
	&\quad +c\left( \int_0^t \| v^n \|_{H^1_z H^1_{xy}}^2 \ds \right)^2 + c \int_0^t \| f_v \|_{L^2}^2 \ds.
\end{align*}
Hence, by the above and the Young inequality, we get
\begin{align*}
	\E \sup_{s \in [0, t]} &A(v^n)^{p/q} \leq  \E A(v^n(0))^{p/q} + c_{p, q}\\
	&+c_{p_q} \E \left[ \sup_{s \in [0, t]} \| U^n \|_{L^2_z H^1_{xy}}^{2p} 
	+ \left( \int_0^t \| v^n \|_{L^2_zH^2_{xy}}^2 \ds \right)^{p} \right]\\
	&+ c_{p,q} \E \left[ \left( \int_0^t \| v^n \|_{H^1_zH^1_{xy}}^2 \ds \right)^{p} + \left( \int_0^t \| f_v^n \|_{L^q}^2 \ds \right)^{p/2} \right].
\end{align*}
Now, \eqref{eq:q_boundedness} follows from the construction of the solution provided
\[
	U_0 \in L^{2p} \left( \Omega; L^2_z H^1_{xy} \right) \cap L^p\left( \Omega; L^q \right), \quad f \in L^{p}\left( \Omega; L^2(0, t; L^q) \right).
\]
\end{proof}

We conclude this section with an $L^q(M)$-regularity result for the temperature.
\begin{lemma}
	Let $2\leq p \leq q < \infty$, $U_0 \in L^{3p/2}\left( \Omega; L^2_z H^1_{xy} \right) \cap L^p\left( \Omega; L^q \right)$ and $f \in  L^{3p/2}\left( \Omega; L^2(0, t; L^2) \right) \cap L^{p}\left( \Omega; L^2(0, t; L^q) \right) $.
	Then the stopping times $\tau_K^{T, q,p}$ defined for $K \in \N$ by
	\[
	\tau_K^{T,q,p} = \inf \left\lbrace s \geq 0 \mid \sup_{r \in [0, s \wedge \xi)}\norm{T}_{L^q}^p  %+ \int_0^{s \wedge \xi}  \norm{|T|^{(q-2)/2}\nablah T}^2_{L^2} \dr 
	\geq K \right\rbrace
	\]
	satisfies $\tau_K^{T,q,p} \to \infty$ $\PP$-a.s.\ as $K \to \infty$.
\end{lemma}

\begin{proof}
%Applying the It\^{o} formula to $\|T\|_{L^q}^q$ and the cancellation property of the nonlinear term, we get
%\begin{align*}
%	\d\!  \|T\|_{L^q}^q +q&\nu_T \| |T|^{(q-2)/2}\nablah T\|_{L^2}^{2}\dt
%	\leq -q\left<|T|^{q-2} T, f_T \right> \dt\\
%	&+\frac{q(q-1)}{2}\sum_{k=1}^{\infty}\left< |T|^{q-2}, (\sigma_2(U)e_k)^2 \right> \dt
%	 +q \left< \vert T\vert^{q-2} T, \sigma_2(U) \dW \right>
%\end{align*}
For $K \in \N$, let $\Upsilon_K = \tau_K^{w, 2} \wedge \tau_K^{\nablah \overline{v}, 2} \wedge \tau_K^{v, q, 2}$. Let $N \in \N$ and let $0 \leq \tau_a \leq \tau_b \leq t \wedge \tau_N \wedge \Upsilon_K$ be stopping times. Similarly as above, we use \eqref{eq:noiseglobal} and \eqref{eq:g_k_growth} to deduce
\begin{multline*}
\sum_{k=1}^{\infty} \left< |T|^{q-2}, (\sigma_2(U)e_k)^2 \right>\\
\leq (\eta^2+\varepsilon)\| |T|^{(q-2)/2}\nablah T\|_{L^2}^{2}+ (1+\|\Deltah \overline v\|_{L^2}^2+\|v\|_{L^q}^2)(1+\|T\|_{L^q}^q) \ds.
\end{multline*}
Thus, by the Burkholder-Davis-Gundy ineuquality \eqref{eq:bdg}, we get
\begin{align*}
q \E &\sup_{s \in [\tau_a,\tau_b]}
\left| \int_0^s\left< \vert T\vert^{q-1} T, \sigma_2(U) \right> \dW \right|\\
& \leq \E \left[\frac12 \sup_{s \in [\tau_a,\tau_b]} \|T\|_{L^q}^q +
 (q^2 c_{BDG}^2 \eta^2+\varepsilon) \int_{\tau_a}^{\tau_b}  \| |T|^{(q-2)/2}\nablah T\|_{L^2}^{2} \ds \right]\\
&\hphantom{\leq \ } +c\E \int_{\tau_a}^{\tau_b}  (1+\|\Deltah \overline v\|_{L^2}^2+\|v\|_{L^q}^2)(1+\|T\|_{L^q}^q) \ds.
\end{align*}
With $ \int_{\tau_a}^{\tau_b}|\left<|T|^{q-2} T, f_T \right>|\ds  \leq c  \left(\int_{\tau_a}^{\tau_b}\|f_T\|_{L^q}^2\right)^{q/2} + c\int_{\tau_a}^{\tau_b}\|T\|_{L^q}^q\ds$, the claim follows as in the previous proof by applying the It\^{o} formula to $(1+\|T\|_{L^q}^q)^{p/q}$, the stochastic Gronwall lemma \ref{lemma:stoch_Gronwall} and Lemma \ref{lemma:auxiliary}.
\end{proof}

\subsection{Higher order estimates}

\begin{lemma}[$L^\infty$ bound for $v$]
\label{lemma:linfty}
Let $\tau_K^{v, \infty}$ be the stopping time defined for $K \in \N$ by
\[
\tau_K^{v, \infty} = \inf \left\lbrace s \geq 0 \mid \int_0^{s \wedge \xi} \norm{v}_{L^\infty}^2 \ds \geq K \right\rbrace
\]
satisfies $\tau_K^{v, \infty} \to \infty$ $\PP$-a.s.\ as $K \to \infty$.
\end{lemma}

\begin{proof}
	We apply the It\^{o} formula to $\log \left(e+\|(-\Deltah)^{1/2} v\|_{L^2}^{2}+\| \partial_z v\|_{L^2}^{2}+\|\partial_z v \|_{L^{6}}^{6} \right)$. Denoting $A(t) = e+\|(-\Deltah)^{1/2} v\|_{L^2}^{2} +\| \partial_z v\|_{L^2}^{2}  +\|\partial_z v \|_{L^{6}}^{6}$, we obtain
	\begin{align*}
		&\d \log (A(t)) +\frac{2\nu}{A(t)} \|\Deltah v\|_{L^2}^{2}+\frac{2\nu}{A(t)} \|\nablah \partial_z v\|_{L^2}^{2} \dt\\
		&\quad +\frac{6\nu}{A(t)} \left[\| |\partial_z v|^{2} \nablah \partial_z v\|_{L^2}^{2}+4\| |\partial_z v|^{2} \nablah |\partial_z v|\|_{L^2}^{2} \right] \dt\\
		&\leq  \frac{2}{A(t)} \left< (-\Deltah)^{1/2} v, (-\Deltah)^{1/2} F_v(U) - (-\Deltah)^{1/2} b(v, v) \right> \dt\\
		&\quad + \frac{1}{A(t)}\|(-\Deltah)^{1/2} \sigma_1 (U)\|^2_{L_2(\cU,L^2)} \dt
		+ \frac{2}{A(t)} \left<(-\Deltah)^{1/2} v,(-\Deltah)^{1/2}\sigma_1(U) \dW \right>\\
		&\quad + \frac{2}{A(t)} \left< \partial_z v, \partial_zF_v(U) - \partial_zb(v, v) \right> \dt\\
		&\quad + \frac{1}{A(t)}\|\partial_z \sigma_1 (U)\|^2_{L_2(\cU,L^2)} \dt
		+ \frac{2}{A(t)} \left<\partial_z v,\partial_z \sigma_1(U) \dW \right>\\
		&\quad -\frac{6}{A(t)} \left< |\partial_z v|^{4}\partial_z v, \partial_z F_v(U) + b(\partial_z v,v)\right> \dt\\
		&\quad 	+\frac{15}{A(t)} \sum_{k=1}^{\infty}\left< |\partial_z v|^{4},  (\partial_z \sigma_1 (U)e_k)^2 \right> \dt
		 + \frac{6}{A(t)}\left< |\partial_z v|^{4}\partial_z v,\partial_z \sigma_1(U) \dW \right>
	\end{align*}
	Integrating by parts, we have
	\begin{align*}
		\frac{6}{A(t)} \left| \left< |\partial_z v|^{4}\partial_z v,b(\partial_z v,v)\right> \right| &\leq \frac{c}{A(t)} \int_{M} |\partial_z v|^5 | (-\Deltah)^{1/2}\partial_z v| |v| \dxyz\\
		&\leq \frac{\varepsilon}{A(t)} \| | \partial _zv|^2 (-\Deltah)^{1/2}\partial_zv\|_{L^2}^2 + \frac{c_{\varepsilon}}{A(t)} \|v\|_{L^\infty}^2 \| \partial _zv\|_{L^6}^6\\
		&\leq \frac{\varepsilon}{A(t)} \| | \partial _zv|^2 (-\Deltah)^{1/2}\partial_zv\|_{L^2}^2 + c_\varepsilon \|v\|_{L^\infty}^2
	\end{align*}
	and 
		\begin{align*}
		\frac{2}{A(t)} \left| \left< \partial_z v,b(\partial_z v,v)\right> \right| 
		&\leq \frac{c}{A(t)} \| \partial _zv\|_{L^4}^2 \| \nablah v\|_{L^2} \leq c.
	\end{align*}
	Similarly, since $v = 0 = \partial_z v$ on $\Gamma_\ell$, the H\"{o}lder, Young and two-dimensional Gagliardo-Nirenberg inequalities yield
	\begin{equation}
	\label{eq:v_infty_nonlin_z}
	\begin{split}
		\frac{2}{A(t)} &\left| \left< (-\Deltah)^{1/2} v, (-\Deltah)^{1/2} b(v,v) \right>  \right| = \frac{2}{A(t)} \left| \left< \Deltah v, v \cdot \nablah v + w(v) \partial_z v \right>  \right| \\
		&\leq \frac{c}{A(t)} \| \Deltah v \|_{L^2} \left( \| v \|_{L^\infty} \| \nablah v \|_{L^2} + \| \divh v \|_{L^2_z L^3_{xy}} \| \partial_z v \|_{L^2_z L^6_{xy}}\right)\\
		&\leq \frac{c}{A(t)} \| \Deltah v \|_{L^2 }\left( \| v \|_{L^\infty} \| \nablah v \|_{L^2} + \| \Deltah v \|_{L^2 }^{1/3}
		 \| \nablah v \|_{L^2}^{2/3} \| \partial_z v \|_{L^6}\right)\\
		&\leq \frac{\varepsilon}{A(t)} \Vert\Delta_Hv\Vert _{L^2 }^2+ \frac{c_{\varepsilon}}{A(t)}(\Vert v\Vert _{L^\infty } ^2\Vert\nabla_Hv\Vert_{L^2 }^2+\Vert \partial _zv\Vert^{6}_{L^6} + \| \nablah v \|_{L^2}^4)\\
		&\leq \frac{\varepsilon}{A(t)} \Vert\Delta_Hv\Vert _{L^2 }^2+ c_{\varepsilon}(\Vert v\Vert_{L^\infty}^2 + 1 + \| \nablah v \|_{L^2}^2).
	\end{split}
	\end{equation}
	The linear terms are straightforward. We estimate
	\begin{align*}
		\frac{2}{A(t)}& \left| \left< (-\Deltah)^{1/2}v, (-\Deltah)^{1/2} F_v(U)\right> \right|\\
		&\leq \frac{\varepsilon}{A(t)} \norm{\Deltah v}^2_{L^2} + \frac{c_\varepsilon}{A(t)} \left( \norm{T}^2_{L^2_zH^1_{xy}}+ \norm{f_v}^2_{L^2}+ \norm{(-\Deltah)^{1/2} v}^2_{L^2} \right)\\
		&\leq \frac{\varepsilon}{A(t)} \norm{\Deltah v}^2_{L^2} + c_\varepsilon \left( \norm{T}^2_{L^2_zH^1_{xy}}+ \norm{f_v}^2_{L^2} + 1 \right),
	\end{align*}
		\begin{align*}
 		\frac{2}{A(t)} \left< \partial_z v, \partial_zF_v(U)\right>
 		& \leq  \frac{c}{A(t)} \left(\|\partial_z  v\|_{L^2}^2+\|\partial_z  v\|_{L^2}\norm{\partial_z  f_v}_{L^2}+\|\partial_z  v\|_{L^2} \norm{T}_{L^2_zH^1_{xy}} \right)\\
 		& \leq c\left(1+\norm{\partial_z  f_v}_{L^2}+\norm{T}_{L^2_zH^1_{xy}} \right)
	\end{align*}
	and
	\begin{align*}
		\frac{6}{A(t)}& |\left< |\partial_z v|^{4}\partial_z v,\partial_z  F_v(U) \right>|\\
		&\leq \frac{\varepsilon}{A(t)} \| |\partial_z v|^{2}\nablah \partial_z v \|_{L^2}^2 + \frac{c_\varepsilon}{A(t)} \left(1+\norm{T}_{L^6}^2+\norm{\partial_z  f_v}_{L^6}\right)(1+\|\partial_z  v\|_{L^6}^{6})\\
		&\leq \frac{\varepsilon}{A(t)} \| |\partial_z v|^{2}\nablah \partial_z v \|_{L^2}^2 + c_\varepsilon \left(1+\norm{T}_{L^6}^2+\norm{\partial_z  f_v}_{L^6}\right),
	\end{align*}
where we used 
	\begin{align*}
	\left|\left< |\partial_z v|^{4}\partial_z v, \int_{\cdot}^{0}\nablah T(x,y,z')\dz'\right> \right| \leq  \| |\partial_z v|^{2}\nablah \partial_z v \|_{L^2}\|\partial_z  v\|_{L^6}^2\norm{T}_{L^6}.
	\end{align*}	
	For $K \in \N$, let $\Upsilon_K = \tau_K^{v, 132, 8/3}$. Let $N \in \N$ and let $0 \leq \tau_a \leq \tau_b \leq t \wedge \tau_N \wedge \Upsilon_K$ be stopping times. By the sub-linear growth of $\sigma_1$ \eqref{eq:sigmaGrowthL2H1} and the Burkholder-Davis-Gundy inequality \eqref{eq:bdg}, we deduce
	\begin{align*}
		2\E \sup_{s\in[\tau_a,\tau_b]} &\left| \int_{\tau_a}^{\tau_b} \frac{1}{A(s)}\left<  (-\Deltah)^{1/2}v, (-\Deltah)^{1/2}\sigma_1(U) \d\! W\right> \right| \\
		&\leq c \E \left( \int_{\tau_a}^{\tau_b}  \frac{\| (-\Deltah)^{1/2}v\|_{L^2}^{2}}{A(s)}  \frac{\| (-\Deltah)^{1/2} \sigma_1(U)\|^2_{L_2(U,L^2)}}{A(t)} \ds \right)^{1/2}\\
		&\leq c \E \left( \int_{\tau_a}^{\tau_b}  \frac{1+\|\Deltah v\|_{L^{2}}^{2}}{A(s)} \ds \right)^{1/2}\\
		&\leq \varepsilon \E \int_{\tau_a}^{\tau_b}  \frac{\|\Deltah v\|_{L^{2}}^{2}}{A(s)} \ds +c_\varepsilon \E \int_{\tau_a}^{\tau_b}  \frac{1}{A(s)} \ds + c_\varepsilon\\
		&\leq \varepsilon \E \int_{\tau_a}^{\tau_b}  \frac{\|\Deltah v\|_{L^{2}}^{2}}{A(s)} \ds + c_\varepsilon		
%		&\leq \varepsilon \E \left( \int_{\tau_a}^{\tau_b}  \frac{1}{A(s)} \|\Deltah v\|_{L^{2}}^{2} \ds\right) +c\E \left( \int_0^t  \frac{1}{A(s)}(1+\|\nablah v\|_{L^2}^{2}) \ds\right)\\
%		&\leq \varepsilon \E \left( \int_{\tau_a}^{\tau_b}  \frac{1}{A(s)} \|\Deltah v\|_{L^{2}}^{2} \ds\right) +c
	\end{align*}
	and, similarly, we estimate the first correction term by
	\begin{align*}
		\frac{1}{A(t)}\|(-\Deltah)^{1/2} \sigma_1 (U)\|^2_{L_2(\cU,L^2)} \leq \eta^2 \frac{1}{A(t)}\|\Deltah v\|^2_{L^2} + c.
	\end{align*}
	In the same way, we obtain
		\begin{align*}
		2\E \sup_{s\in[\tau_a,\tau_b]} &\left| \int_{\tau_a}^{\tau_b} \frac{1}{A(s)}\left<  \partial_z v, \partial_z\sigma_1(U) \d\! W\right> \right| \leq \varepsilon \E \int_{\tau_a}^{\tau_b}  \frac{\|\nablah  \partial_z v\|_{L^{2}}^{2}}{A(s)} \ds + c_\varepsilon	
	\end{align*}
	and
	\begin{align*}
		\frac{1}{A(t)}\|(-\Deltah)^{1/2} \sigma_1 (U)\|^2_{L_2(\cU,L^2)} \leq \eta^2 \frac{1}{A(t)}\|\nablah  \partial_z v\|^2_{L^2} + c.
	\end{align*}
	Recalling \eqref{eq:sigma1_dz} and \eqref{eq:h_k_growth}, we have
	\begin{align*}
		\sum_{k=1}^{\infty} &\left| \left< |\partial_z v|^{4} \partial_z  v,\partial_z  \sigma_1(U)e_k \right> \right|^2\\
%		&= \sum_{k=1}^{\infty} \left| \left< |\partial_z v|^{4} \partial_z  v, \Psi_k\cdot\nablah \partial_z v+ \partial_z \phi_k \cdot \nablah \overline{v}
%		+ \partial_z h_k(v) \right> \right|^2\\
		&\leq c\big( \| |\partial_z v|^{2} \nablah \partial_z v\|_{L^2}\| |\partial_z v|^{3}\|_{L^2} + \| |\partial_z v|^{5}\|_{L^{6/5}}\|\nablah \overline{v}\|_{L^{6}} + \| \partial_z v\|_{L^{6}}^6\\
		&\hphantom{\leq c \big( \ } + \| |\partial_z v|^{5}\|_{L^{6/5}}
		\| v\|_{L^{6}}+ \| |\partial_z v|^{5}\|_{L^{1}}\big)^2\\
		&\leq c \left( \| |\partial_z v|^{2} \nablah \partial_z v\|_{L^2}^2\|\partial_z v\|_{L^6}^6 + (1+\|\nablah \overline{v}\|_{L^{2}}^2+\|\nablah \overline{v}\|_{L^{6}}^2 )(1+\| \partial_z v\|_{L^{6}}^{12}) \right),
	\end{align*}
%	\[
%	\partial_z\sigma_1(U) e_k = \Psi_k\cdot\nablah \partial_z v+ \partial_z \phi_k \cdot \nablah\overline{v} + \partial_z h_k(v),
%	\]
	and hence, by the Burkholder-Davis-Gundy inequality \eqref{eq:bdg} and \eqref{eq:v_bound_tildebar},
	\begin{align*}
		6\E &\sup_{s \in [\tau_a,\tau_b]}
		\left| \int_0^s\frac{1}{A(s)} \left< |\partial_z v|^{4} \partial_z  v,\partial_z  \sigma_1(U) \right> \dW \right|\\
		&\leq c \E \left( \int_{\tau_a}^{\tau_b} \frac{1}{A(s)^2}   \sum_{k=1}^{\infty} |\left< |\partial_z v|^{4} \partial_z  v,\partial_z  \sigma_1(U)e_k \right>|^2 \ds\right)^{1/2}\\
%		& \leq c\E \left[\left( \int_{\tau_a}^{\tau_b} \frac{1}{A(s)^2}   \sum_{k=1}^{\infty} |\left< |\partial_z v|^{4} \partial_z  v,|\Psi_k\cdot\nablah \partial_z v+ \partial_z \phi_k \cdot \nablah \overline{v}
%		+ \partial_z h_k(v)| \right>^2 \ds\right)^{1/2} \right]\\
%		& \leq c\E \Big[\Big( \int_{\tau_a}^{\tau_b} \frac{1}{A(s)^2}  
%		(\| |\partial_z v|^{2} \nablah \partial_z v\|_{L^2}\| |\partial_z v|^{3}\|_{L^2}
%		+ \| |\partial_z v|^{5}\|_{L^{6/5}}\|\nablah \overline{v}\|_{L^{6}}
%		+ \| \partial_z v\|_{L^{6}}^6\\
%		& \qquad\qquad\qquad\qquad +\| |\partial_z v|^{5}\|_{L^{6/5}}
%		\| v\|_{L^{6}}+ \| |\partial_z v|^{5}\|_{L^{1}})^2 \ds\Big)^{1/2} \Big]\\
		& \leq c\E \bigg(
		\int_{\tau_a}^{\tau_b} \frac{\| |\partial_z v|^{2} \nablah \partial_z v\|_{L^2}^2\|\partial_z v\|_{L^6}^6}{A(s)^2}\\
		&\hphantom{\leq c\E \bigg(
		\int_{\tau_a}^{\tau_b} \ } +\frac{(1+\| \partial_z v\|_{L^{6}}^{12})}{A(s)^2} (1+\|\nablah \overline{v}\|_{L^{2}}^2+\|\nablah \overline{v}\|_{L^{6}}^2) \ds\bigg)^{1/2}\\
		& \leq c\E \left( \int_{\tau_a}^{\tau_b} \frac{\| |\partial_z v|^{2} \nablah \partial_z v\|_{L^2}^2}{A(s)}  
		+1+\|\Deltah \overline{v}\|_{L^{2}}^2 \ds\right)^{1/2}\\
		&\leq \varepsilon \E \int_{\tau_a}^{\tau_b} \frac{\| |\partial_z v|^{2} \nablah \partial_z v\|_{L^2}^2}{A(s)}  \ds + c_{\varepsilon} \E \int_{\tau_a}^{\tau_b}1+\|\Deltah \overline{v}\|_{L^{2}}^2\ds + c_\varepsilon.
	\end{align*}
%	We obtain for $\varepsilon>0$
%	\begin{align*}
%		\E &\sup_{s \in [\tau_a,\tau_b]}
%		\left| \int_0^s\frac{1}{A(s)} \left< |\partial_z v|^{4} \partial_z  v,\partial_z  \sigma_1(U) \right> \dW \right|\\
%		& \leq \varepsilon \E \left[\int_{\tau_a}^{\tau_b} \frac{1}{A(s)}  
%		\| |\partial_z v|^{2} \nablah \partial_z v\|_{L^2}^2 \ds \right] +c \E \left[\int_{\tau_a}^{\tau_b}1+\|\Deltah \overline{v}\|_{L^{2}}^2\ds \right].
%	\end{align*}
	For the remaining correction term, we proceed similarly. By \eqref{eq:v_bound_tildebar}, \eqref{eq:sigma1_dz}, \eqref{eq:h_k_growth} and \eqref{eq:noise_eta}, we obtain
	\begin{align*}
		\frac{15}{A(t)} &\sum_{k=1}^{\infty} \left< |\partial_z v|^{4}, (\partial_z\sigma_1 (U)e_k)^2 \right>\\
%		&= \frac{15}{A(t)} \sum_{k=1}^{\infty}\left< |\partial_z v|^{4}, ( \Psi_k\cdot\nablah \partial_z v+ \partial_z \phi_k \cdot \nablah \overline{v} + \partial_z h_k(v))^2 \right>\\
		&\leq
		\left(15 \eta^2+\varepsilon\right) \frac{\| |\partial_z v|^{2} \nablah \partial_z v\|_{L^2}^2}{A(t)} +
		\frac{c_\varepsilon}{A(t)} \big( \| |\partial_z v|^{4}\|_{L^{3/2}}\| |\nablah \overline{v}|^{2}\|_{L^{3}}  +
		\|\partial_z v \|_{L^6}^6\\
		&\quad +
		\| |\partial_z v|^{4}\|_{L^{3/2}}\| |v|^{2}\|_{L^{3}}+
		\| \partial_z v\|_{L^4}^4 \big)\\
		&\leq
		\left(15 \eta^2+\varepsilon\right) \frac{\| |\partial_z v|^{2} \nablah \partial_z v\|_{L^2}^2}{A(t)}\\
		&\hphantom{\leq \ } + \frac{c_\varepsilon}{A(t)}\left(\| \partial_z v\|_{L^6}^4 \left(1+\|\Deltah \overline{v}\|_{L^{2}}^2 + \|  v\|_{L^6}^2 \right) + \| \partial_z v\|_{L^6}^6+1\right)\\
		&\leq
		\left(15 \eta^2+\varepsilon\right) \frac{\| |\partial_z v|^{2} \nablah \partial_z v\|_{L^2}^2}{A(t)}	+c_\varepsilon \left(1+\|\Deltah \overline{v}\|_{L^{2}}^2\right)	
	\end{align*}
%	It follows that
%	\begin{align*}
%	\E &\sup_{s \in [\tau_a,\tau_b]}
%	\left| \int_0^s\frac{1}{A(s)} \left< |\partial_z v|^{4} \partial_z  v,\partial_z  \sigma_1(U) \right> \dW \right|+ \E \left|  \int_{\tau_a}^{\tau_b} \frac{1}{A(s)}\sum_{k=1}^{\infty}\left< |\partial_z v|^{4}, (\partial_z\sigma_1 (U)e_k)^2 \right>\right| \ds\\
%	&\leq (\eta^2+\varepsilon) \E \left[\int_{\tau_a}^{\tau_b} \frac{1}{A(s)}  
%	\| |\partial_z v|^{2} \nablah \partial_z v\|_{L^2}^2 \ds \right] +c\E \left[\int_{\tau_a}^{\tau_b} 1 +\|\Deltah \overline{v}\|_{L^{2}}^2\ds \right].
%	\end{align*}
Collecting the estimates above, we have
\begin{equation}
	\label{eq:v_inf_est1}
	\begin{split}
	&\E \sup_{s\in[\tau_a,\tau_b]} \log (A(s)) +\E \int_{\tau_a}^{\tau_b}\frac{2\nu-\eta^2-3\varepsilon}{A(s)} \|\Deltah v\|_{L^2}^{2}\\
	&\quad +\E \int_{\tau_a}^{\tau_b}\frac{2\nu-\eta^2-3\varepsilon}{A(s)} \|\nablah \partial_z v\|_{L^2}^{2} +\frac{(6\nu-15\eta^2-4\varepsilon)}{A(s)} \| |\partial_z v|^{2} \nablah \partial_z v\|_{L^2}^{2}\ds\\
	&\leq \E \log (A(\tau_a))+c_\varepsilon\\
	&\quad + c_\varepsilon \E  \int_{\tau_a}^{\tau_b} \|v\|_{L^{\infty}}^2+\| \Deltah \overline{v}\|_{L^{2}}^2+\norm{T}_{L^6}^2+\norm{\partial_z  f_v}_{L^6} \ds.
	\end{split}
\end{equation}
Let $B(t) = e + \| \Deltah v\|_{L^2}^2+ \| \nablah \partial_z v\|_{L^2}^2 + \| \partial_z v\|_{L^6}^6$. Now we employ an argument similar to the logarithmic Sobolev inequality. By the Poincar\'e inequality, \eqref{eq:v_bound_tildebar}, the logarithmic Sobolev inequality \eqref{equation:sobolevlogarithmic} with $\lambda=1/2$ and $r_1 = 132$ and the inequality $\log z \leq c z^{1/4}$ holding for $z \geq 1$, we obtain
\begin{align*}
	\Vert v\Vert_{L^\infty}^2 &\leq c \left(1+\Vert v\Vert_{L^{132}}^2 \right) \log\left(e+\Vert \nablah v\Vert_{L^6}+\Vert v\Vert_{L^6}+\Vert \partial_z v\Vert_{L^2}+\Vert v\Vert_{L^2}\right)\\
	&\leq c \left(1+\Vert v\Vert_{L^{132}}^2 \right) \log B(t)\\
	&\leq c \left(1+\Vert v\Vert_{L^{132}}^2 \right) \left( \log A(t) + \log \frac{B(t)}{A(t)} \right)\\
	&\leq c \left(1+\Vert v\Vert_{L^{132}}^2 \right) \left( \log A(t) + \left(\frac{B(t)}{A(t)}\right)^{1/4} \right)\\
	&\leq c_\varepsilon \left(1+\Vert v\Vert_{L^{132}}^{8/3} \right) \left( \log A(t) + 1 \right) + \varepsilon \frac{\| \Deltah v \|_{L^2}^2}{A(t)} + \varepsilon \frac{1 + \| \partial_z v \|_{L^6}^6}{A(t)}\\
	&\leq c_\varepsilon \left(1+\Vert v\Vert_{L^{132}}^{8/3} \right) \left( \log A(t) + 1 \right) + \varepsilon \frac{\| \Deltah v \|_{L^2}^2}{A(t)}.
\end{align*}
%With
%\begin{align*}
%\log & (e+\Vert \nablah v\Vert_6+\Vert v\Vert_{L^6}+\Vert \partial_zv\Vert _{L^2}+\Vert v\Vert_{L^2 })\\
%	& \leq c + \log(e+\Vert \Deltah v\Vert_{L^2 }+\Vert \nablah\overline{v} \Vert _{L^2}+\Vert \partial_zv\Vert _{L^6}+\Vert \nablah v\Vert_{L^2 })\\	
%	& \leq c + \log(e+\Vert \Deltah v\Vert_{L^2 }+\Vert \partial_zv\Vert _{L^6}+\Vert \nablah v\Vert_{L^2 })
%\end{align*}
Thus, from \eqref{eq:v_inf_est1}, it follows
\begin{align*}
	&\E \sup_{s\in[\tau_a,\tau_b]} \log (A(s))+ \E \int_{\tau_a}^{\tau_b}\frac{2\nu-\eta^2-4\varepsilon}{A(s)} \|\Deltah v\|_{L^2}^{2}\\
	&\quad +\E \int_{\tau_a}^{\tau_b}\frac{2\nu-\eta^2-3\varepsilon}{A(s)} \|\nablah \partial_z v\|_{L^2}^{2}
	 +\frac{(6\nu-15\eta^2-4\varepsilon)}{A(s)} \| |\partial_z v|^{2} \nablah \partial_z v\|_{L^2}^{2}\ds\\
%	& \leq \E\left[  \log (A(\tau_a))\right]+c\E \big[\int_{\tau_a}^{\tau_b} 1+\norm{T}_{L^6}^2+\norm{\partial_z  f_v}_{L^6}+\norm{\partial_z  f_v}_{L^2}^2+\|\Deltah \overline{v}\|_{L^{2}}^2\\
%	&+\left(1+\Vert v\Vert_{L^{r_1}}^2 \right) \log(e+\Vert \Deltah v\Vert_2+\Vert \partial_zv\Vert _{L^6}+\Vert \nablah v\Vert_{L^2 })\ds \big]\\
%	& = \E\left[  \log (A(\tau_a))\right]+c\E \big[ \int_{\tau_a}^{\tau_b} 1+\norm{T}_{L^6}^2+\norm{\partial_z  f_v}_{L^6}+\norm{\partial_z  f_v}_{L^2}^2+\|\Deltah \overline{v}\|_{L^{2}}^2\\
%	&+\left(1+\Vert v\Vert_{L^{r_1}}^2 \right)\left(\log(A(s))+ \log\left(\frac{e+\Vert \Deltah v\Vert_2+\Vert \partial_zv\Vert _{L^6}+\Vert \nablah v\Vert_{L^2 }}{A(s)}\right)\right)\ds \big]\\
	&\leq\E \left[ \log (A(\tau_a)) + c_\varepsilon \int_{\tau_a}^{\tau_b} 1+\norm{T}_{L^6}^2+\norm{\partial_z  f_v}_{L^6}+\|\Deltah \overline{v}\|_{L^{2}}^2 \ds \right]\\
	&\quad +c_\varepsilon \E \int_{\tau_a}^{\tau_b} \left(1+\Vert v\Vert_{L^{132}}^{8/3} \right)\left(1 + \log A(s) \right) \ds
\end{align*}
By a similar argument as in the previous proofs relying on the stochastic Gronwall lemma \ref{lemma:stoch_Gronwall} and Lemma \ref{lemma:auxiliary}, we deduce that the stopping time $\tau^{\log A}_K$ defined for $K \in \N$ by
\[
	\tau^{\log A}_K = \left\{ s \geq 0 \mid \sup_{r \in [0, s \wedge \xi)} \log A(r) + \int_0^{s \wedge \xi} \frac{B(r)}{A(r)} \dr \geq K \right\}
\]
satisfies $\tau^{\log A}_K \to \infty$ $\PP$-almost surely as $K \to \infty$.
%for $0<\alpha$.
%Finally, it follows for $\alpha<1$ (for constants depending on $c,\alpha,r_1,\varepsilon$)
%\begin{align*}
%	&\E\left[ \sup_{s\in[\tau_a,\tau_b]} \log (A(s)) \right]+\E \left[\int_{\tau_a}^{\tau_b}\frac{2\nu-2\eta^2-\varepsilon}{A(s)} \|\Deltah v\|_{L^2}^{2} +\frac{(6\nu-15\eta^2-\varepsilon)}{A(s)} \| |\partial_z v|^{2} \nablah \partial_z v\|_{L^2}^{2}\ds \right]\\
%	& \E\left[  \log (A(\tau_a))\right]+c\E \big[ \int_{\tau_a}^{\tau_b} 1+\norm{T}_{L^6}^2+\norm{f_v}_{H^1}^2+\|\Deltah \overline{v}\|_{L^{2}}^2\\
%	&+\left(1+\Vert v\Vert_{L^{r_1}}^2 \right)\log(A(s)) 
%	+\left(1+\Vert v\Vert_{L^{r_1}}^2 \right)^{1/(1-\alpha)}\big]
%\end{align*}

Finally, similarly as above, for $\hat{\Upsilon}_K = \tau^{\log A}_K \wedge \tau^{v, 132, 8/3}_K$ and $s \geq 0$, we get
\begin{align*}
	\E \int_{0}^{s \wedge \tau_N \wedge \hat{\Upsilon}_K} \Vert v\Vert_{L^\infty}^2 \dr  &\leq  c\E \int_0^{s \wedge \tau_N \wedge \hat{\Upsilon}_K} \left( 1 + \Vert v\Vert_{L^{132}}^{8/3}\right)\left( 1 +  \log A(r) \right) + \frac{B(r)}{A(r)} \dr
%	& \leq  c\E\left[ \sup_{s\in[\tau_a,\tau_b]} \left(1+\Vert v\Vert_{L^{r_1}}^2 \right) \log^{\alpha}(A(s))\right]\\
%	& \leq c\E\left[ \sup_{s\in[\tau_a,\tau_b]} \left(1+\Vert v\Vert_{L^{r_1}}^2 \right)^{1/(1-\alpha)}+ \log(A(s))\right],
\end{align*}
and the claim follows by Lemma \ref{lemma:auxiliary}.
\end{proof}

\begin{remark}
The above Lemma can be shown by considering $A(t) = e+\|(-\Deltah)^{1/2} v\|_{L^2}^{2} +\| \partial_z v\|_{L^2}^{2}  +\|\partial_z v \|_{L^{q}}^{4q/(q-2)}$ for $q\in (2,6)$ instead of the case $q=6$.
Taking $q$ large means stronger integrability in the spacial variable, but less in the probability space. The only estimate that has to be substantially adopted is  \ref{eq:v_infty_nonlin_z}. One has to replace it by
	\begin{align*}
		\frac{2}{A(t)} &\left| \left< (-\Deltah)^{1/2} v, (-\Deltah)^{1/2} b(v,v) \right>  \right| = \frac{2}{A(t)} \left| \left< \Deltah v, v \cdot \nablah v + w(v) \partial_z v \right>  \right| \\
		&\leq \frac{c}{A(t)} \| \Deltah v \|_{L^2} \left( \| v \|_{L^\infty} \| \nablah v \|_{L^2} + \| \divh v \|_{L^2_z L^{2q/(q-2)}_{xy}} \| \partial_z v \|_{L^2_z L^q_{xy}}\right)\\
		&\leq \frac{c}{A(t)} \| \Deltah v \|_{L^2 }\left( \| v \|_{L^\infty} \| \nablah v \|_{L^2} + \| \Deltah v \|_{L^2 }^{2/q}
		\| \nablah v \|_{L^2}^{(q-2)/q} \| \partial_z v \|_{L^q}\right)\\
		&\leq \frac{\varepsilon}{A(t)} \Vert\Delta_Hv\Vert _{L^2 }^2+ \frac{c_{\varepsilon}}{A(t)}(\Vert v\Vert _{L^\infty } ^2\Vert\nabla_Hv\Vert_{L^2 }^2+\Vert \partial _zv\Vert^{4q/(q-2)}_{L^q} + \| \nablah v \|_{L^2}^4)\\
		&\leq \frac{\varepsilon}{A(t)} \Vert\Delta_Hv\Vert _{L^2 }^2+ c_{\varepsilon}(\Vert v\Vert_{L^\infty}^2 + 1 + \| \nablah v \|_{L^2}^2).
	\end{align*}
We decided to consider the case $q=6$ for simplicity.
\end{remark}

Having established the $L^\infty$ bound, we can prove the boundedness of $\partial_{z} U$ and deduce the boundedness of $\Deltah U$ with estimates close to the ones used in Section \ref{sect:maximal_solutions_main} for the local existence. In most cases, only the estimates on the nonlinear terms have to be changed.

\begin{lemma}[$L^2$ bound for $\partial_z v$]
\label{lemma:dz_v_L2}
	Let $q\geq 2$, $\partial_z v_0\in L^q(\Omega;L^2)$ and assume that $f_v\in L^q(\Omega;L^2_{loc}(0,\infty;H^1_zL^2_{xy})$.
Then stopping time $\tau_K^{\partial_z v, q}$ defined for $K \in \N$ by
\begin{align*}
	\tau_K^{\partial_z v, q} = \inf \Bigg\lbrace s \geq 0 \mid &\sup_{r \in [0, s \wedge \xi)} \norm{\partial_z v}_{L^2}^q + \int_0^{s \wedge \xi} \norm{\partial_z v}_{L^2}^{q-2} \norm{\nablah \partial_z v}^2_{L^2} \dr\\
	&+\left( \int_0^{s \wedge \xi}\|\nablah \partial_z v\|_{L^2}^{2}\dr\right)^{q/2} \geq K \Bigg\rbrace
\end{align*}
satisfies $\tau_K^{\partial_z v, q} \to \infty$ $\PP$-a.s.\ as $K \to \infty$.
\end{lemma}

\begin{proof}
Following Step 3 in Lemma \ref{lemma:galerkinestimate}, we employ the It\^{o} formula and get
\begin{equation*}
	%\label{eq:galerkinuz}
	\begin{split}
		\d\! &\|\partial_z v\|_{L^2}^q +q \nu \|\partial_z v\|_{L^2}^{q-2} \|(-\Deltah)^{1/2} \partial_z v\|_{L^2}^{2} \dt\\
		&\leq -q \|\partial_z v\|_{L^2}^{q-2} \left< \partial_z v, \partial_z f(U) +   b(\partial_z v,v)\right> \dt\\
		&\quad +\frac{q(q-1)}{2}  \|\partial_z v\|_{L^2}^{q-2} \|\partial_z \sigma_1 (U)\|^2_{L_2(U,L^2)} \dt\\
		&\quad +q \|\partial_z v\|_{L^2}^{q-2} \left< \partial_z v,\partial_z \sigma(U) \dW \right>.
	\end{split}
\end{equation*}
For the nonlinear term, similarly as in \eqref{eq:v_infty_nonlin_z}, we get
\begin{align*}
|\left< \partial_z v,  b(\partial_z v,v)\right>| &\leq \varepsilon \|\divh \partial_z v\|_{L^2}^2+c\|v\|_{L^\infty}^2  \| \partial_z v\|_{L^2}^2.
% &=|\left< \partial_z v, \partial_zv\nablah v-\divh v \partial_z v\right>|\\
% &\leq \|v\|_{L^\infty}  \| \partial_z v\|_{L^2}  \|\divh \partial_z v\|_{L^2}
\end{align*}
For $K \in \N$ and $N \in \N$, let $0 \leq \tau_a \leq \tau_b \leq t \wedge \tau_N \wedge \tau_K^{w, q} \wedge \tau_K^{v,\infty}$ be stopping times. By the sub-linear growth of $\sigma_1$ \eqref{eq:sigmaGrowthH1L2} and the Burkholder-Davis-Gundy inequality \eqref{eq:bdg}, we follow the estimates in Step 3 in Lemma \ref{lemma:galerkinestimate} and deduce 
\begin{multline*}
	q\E \sup_{s \in [\tau_a, \tau_b]} \left| \int_{\tau_a}^{\tau_b} \|\partial_z v\|_{L^2}^{q-2} \left< \partial_z v,\partial_z \sigma(U)\right> \dW \right| \leq \frac14 \E \sup_{s\in[\tau_a, \tau_b]}\| \partial_z v\|_{L^2}^{q}\\
	+ c\E\int_{\tau_a}^{\tau_b} 1 + \|  v\|_{H^1_zL^2_{xy}}^{q} \ds
	+q^2c_{BDG}^2\eta^2 \E \int_{\tau_a}^{\tau_b} \|v\|_{H^1_zL^2_{xy}}^{q-2}\|(-\Deltah)^{1/2} \partial_z v\|_{L^{2}}^{2}\ds
\end{multline*}
and
\begin{multline*}
	\int_{\tau_a}^{\tau_b} \|\partial_z v\|_{L^2}^{q-2}\left|  \left< \partial_z v, \partial_z f(U)\right> \right| \ds \leq \frac14 \sup_{s \in [0, t]} \| \partial_z v\|_{L^2}^q\\
	+ c \left( \int_{\tau_a}^{\tau_b} \| \nablah T \|_{L^2}^2 \ds \right)^{q/2}
	+ c \int_{\tau_a}^{\tau_b} \| \partial_z v\|_{L^2}^q \ds + c\left(\int_{\tau_a}^{\tau_b} \| \partial_z f_v\|_{L^2}^2 \ds\right)^{q/2}.
\end{multline*}
Collecting the above and using \eqref{eq:sigmaGrowthH1L2} for the correction term, we get
\begin{align*}
	\frac12 \E &\sup_{s\in[\tau_a, \tau_b]}\|\partial_z v\|_{L^2}^{q} + c(q, \nu, \varepsilon, \eta) \E \int_{\tau_a}^{\tau_b} \|\partial_z v\|_{L^2}^{q-2} \|(-\Deltah)^{1/2} \partial_z v\|_{L^2}^{2} \ds\\
	&\leq c \E\left[\|\partial_z v(\tau_a)\|_{L^2}^{q} +1+ \int_{\tau_a}^{\tau_b}\|v\|_{L^2}^{q} + (1+\|v\|_{L\infty}^2) \|\partial_z v\|_{L^2}^{q} \ds \right]\\
	&\quad + c \E \left( \int_{\tau_a}^{\tau_b} \| U \|_{L^2_zH^1_{xy}}^2 \ds \right)^{q/2}+ c\left(\int_{\tau_a}^{\tau_b} \| \partial_z f_v\|_{L^2}^2 \ds\right)^{q/2}
\end{align*}
where $c(q, \nu,\varepsilon, \eta) = q[\nu - \varepsilon -\eta^2(q c_{BDG}^2 + \frac{q-1}{2})]$. The claim follows as in the above proofs.
\end{proof}

\begin{lemma}[$L^2$ bound for $\partial_z T$]
\label{lemma:dz_T_L2}
	Let $q\geq 2$, $\partial_z T_0\in L^q(\Omega;L^2)$ and $f_T\in L^q(\Omega;L^2_{loc}(0,\infty;H^1_zL^2_{xy})$.
	Then  the stopping time $\tau_K^{\partial_z T, q}$ defined for $K \in \N$ by
	\begin{align*}
	\tau_K^{\partial_z T, q} = \inf \Bigg\lbrace s \geq 0 \mid & \sup_{r \in [0, s \wedge \xi)}\norm{\partial_z T}_{L^2}^q  + \int_0^{s \wedge \xi} \norm{\partial_z T}_{L^2}^{q-2} \norm{\nablah \partial_z T}^2_{L^2} \dr\\
	&+\left( \int_{\tau_a}^{\tau_b}\|\nablah \partial_z T\|_{L^2}^{2}\ds\right)^{q/2} \geq K \Bigg\rbrace
	\end{align*}
	satisfies $\tau_K^{\partial_z T, q} \to \infty$ $\PP$-a.s.\ as $K \to \infty$.
\end{lemma}

\begin{proof}
%	Applying the Ito-formula to $\|\partial_z T\|_{L^2}^2$ we get
%	\begin{align*}
%		\d\!  \|\partial_z T\|_{L^2}^q +q&\nu_T  \norm{\partial_z T}_{L^2}^{q-2}\| \nablah \partial_z T \|_{L^2}^{2}\dt
%		\leq -q \norm{\partial_z T}_{L^2}^{q-2}\left<\partial_z T, \partial_z f_T \right>
%		-q \norm{\partial_z T}_{L^2}^{q-2}\left<\partial_z T, b(\partial_z v,T) \right> \dt\\
%		&+\frac{q(q-1)}{2} \norm{\partial_z T}_{L^2}^{q-2}\| \partial_z T \sigma_2(U)\|_{L^2(\cU,L^2 )}^2 \dt
%		+q \norm{\partial_z T}_{L^2}^{q-2} \left< \partial_z T, \partial_z \sigma_2(U) \dW \right>
%	\end{align*}
%	As before we deduce
%	\begin{multline*}
%		\| \partial_z T \sigma_2(U)\|_{L^2(\cU,L^2 )}^2
%		\leq (\eta^2+\varepsilon)\| | \nablah \partial_z T \|_{L^2}^{2}+c (1+\| \partial_z \nablah v\|_{L^2}^2+\|U\|_{L^2}^2+\| \partial_z v\|_{L^2}^2+\| \partial_z T\|_{L^2}^2)
%	\end{multline*}
%	and
%	\begin{align*}
%		q	\E &\sup_{s \in [\tau_a,\tau_b]}
%		\left| \int_{\tau_a}^{\tau_b} \norm{\partial_z T}_{L^2}^{q-2}\left< \partial_z T, \partial_z \sigma_2(U) \right> \dW \right|\\
%		& \leq \frac34  \E \left[\sup_{s \in [\tau_a,\tau_b]} \|\partial_z T\|_{L^2}^2 \right]+
%		(4 c_{BDG}^2 \eta^2+\varepsilon) 
%		\E \left[\int_{\tau_a}^{\tau_b}  \| \nablah \partial_z T\|_{L^2}^{2} \ds \right]\\
%		& +c\E \left[\int_{\tau_a}^{\tau_b}   1+\|U\|_{L^2}^q+\| \partial_z v\|_{L^2}^q+\| \partial_z T\|_{L^2}^q\ds +\left(\int_{\tau_a}^{\tau_b}   \| \partial_z \nablah v\|_{L^2}^2\ds\right)^{q/2} \right].
%	\end{align*}
%	With $|\left<|\partial_z T,\partial_z  f_T \right>| \leq c \|f_T\|_{L^2} (1+\|T\|_{L^2}^2)$ and
As in the previous Lemma, the estimates in Step 3 in Lemma \ref{lemma:galerkinestimate} can be carried over with the exception of the estimate of the nonlinear term. We replace it by
	\begin{align*}
	 &|\left<\partial_z T, b(\partial_z v,T) \right> |
	 = |\left< \partial_z v \nablah T, \partial_z T \right> - \left< \div_H v \partial_z T, \partial_z T \right>|\\
	 &\quad \leq c( 
	 \| \partial_z v\|_{L^2_zL^4_{xy}} \| \nablah T\|_{L^\infty_zL^2_{xy}} \|  \partial_z T\|_{L^2_zL^4_{xy}}
	 +\| \div_H v\|_{L^\infty_zL^2_{xy}} \| \partial_z T\|_{L^2_zL^4_{xy}}^2)\\
	 &\quad \leq c\|\partial_z v\|_{L^2}^{1/2}\|\partial_z \nablah v\|_{L^2}^{1/2}  (\| \nablah T\|_{L^2}^{1/2} \| \nablah \partial_z T\|_{L^2}^{1/2} 
	 +\| \nablah T\|_{L^2})\\ 
	 &\quad \hphantom{\leq \ } \cdot (\| \partial_z T\|_{L^2}^{1/2} \| \nablah \partial_z T\|_{L^2}^{1/2} 
	 +\| \partial_z T\|_{L^2})\\	 
	 &\quad \hphantom{\leq \ }  +c\|\nablah v\|_{L^2}^{1/2}\|\partial_z\nablah v\|_{L^2}^{1/2}(\| \partial_z T\|_{L^2} \| \nablah \partial_z T\|_{L^2} 
	 +\| \partial_z T\|_{L^2}^2)\\
	 &\quad \leq \varepsilon \| \nablah \partial_z T\|_{L^2}^2 +c (1+\| \nablah T\|_{L^2}^2 +\|\partial_z v\|_{L^2}^2\|\partial_z \nablah v\|_{L^2}^2\\
	 &\quad \hphantom{\leq \ } +\|\nablah v\|_{L^2}\|\partial_z \nablah v\|_{L^2})(1+\| \partial_z T\|_{L^2}^2)\\
	  &\quad \leq \varepsilon \| \nablah \partial_z T\|_{L^2}^2 +c (1 +(1+\|\partial_z v\|_{L^2}^2)\|\partial_z \nablah v\|_{L^2}^2+\|U\|_{L^2_zH^1_{xy}}^2)
	 (1+\| \partial_z T\|_{L^2}^2).
	\end{align*}
	For $K \in \N$, let $\Upsilon_K = \tau_K^{w, 2} \wedge \tau_K^{\partial_z v, 4}$. Let $N \in \N$ and let $0 \leq \tau_a \leq \tau_b \leq t \wedge \tau_N \wedge \Upsilon_K$ be stopping times. We get
	 \begin{multline*}
	 	\E\left[ \frac12 \sup_{s\in[\tau_a, \tau_b]}\|\partial_z T\|_{L^2}^{q} + c(q, \nu, \varepsilon, \eta) \int_{\tau_a}^{\tau_b} \|\partial_z T\|_{L^2}^{q-2} \|(-\Deltah)^{1/2} \partial_z T\|_{L^2}^{2} \ds\right]\\
	 	\leq c \E\Bigg[\|\partial_z T(\tau_a)\|_{L^2}^{q} +1+ \int_{\tau_a}^{\tau_b} \| U \|_{L^2}^2+\| \partial_z v \|_{L^2}^2\ds \\
	 	+ \left( \int_{\tau_a}^{\tau_b} \| U \|_{L^2_zH^1_{xy}}^2 \ds \right)^{q/2}+\left( \int_{\tau_a}^{\tau_b} \| \nablah \partial_z v \|_{L^2}^2 \ds \right)^{q/2}+\left( \int_{\tau_a}^{\tau_b} \| \nablah \partial_z f_T \|_{L^2}^2 \ds \right)^{q/2}  \Bigg]\\
	 	+ c \E \int_{\tau_a}^{\tau_b} (1 +(1+\|\partial_z v\|_{L^2}^2)\|\partial_z \nablah v\|_{L^2}^2+\|U\|_{L^2_zH^1_{xy}}^2)
	 	(1+\| \partial_z T\|_{L^2}^q)\ds
	 \end{multline*}
	 where again $c(q, \nu,\varepsilon, \eta) = q[\nu - \varepsilon -\eta^2(q c_{BDG}^2 + \frac{q-1}{2})]$. The claim follows as above.
\end{proof}

\begin{lemma}[$L^2$ bound for $\nablah U$]
\label{lemma:nablah_U_L2}
	Let $q\geq 2$, $U_0\in L^q(\Omega;L^2_zH^1)$ and $F\in L^q(\Omega;L^2_{loc}(0,\infty;L^2)$.
	Then the stopping $\tau_K^{\nablah U, q}$ time defined for $K \in \N$ by
	\begin{align*}
		\tau_K^{\nablah U, q} = \inf \Bigg\lbrace s \geq 0 \mid & \sup_{r \in [0, s \wedge \xi)}\norm{\nablah U}_{L^2}^q  + \int_0^{s \wedge \xi} \norm{\nablah U}_{L^2}^{q-2} \norm{\Deltah U}^2_{L^2} \dr\\
		&+\left( \int_{\tau_a}^{\tau_b}\|\Deltah U\|_{L^2}^{2}\ds\right)^{q/2} \geq K \Bigg\rbrace
	\end{align*}
	satisfies $\tau_K^{\nablah U, q} \to \infty$ $\PP$-a.s.\ as $K \to \infty$.
\end{lemma}

\begin{proof}
%For $K \in \N$, let $\Upsilon_K = \tau_K^{w, 4} \wedge \tau_K^{\partial_z v, 4}$. Let $N \in \N$ and let $0 \leq \tau_a \leq \tau_b \leq t \wedge \tau_N \wedge \Upsilon_K$.\\
We only show the necessary changes in the estimates in Step 5 in Lemma \ref{lemma:galerkinestimate}, in particular changes in the estimates of the nonlinear term. From \eqref{eq:nonlinearDelta} and \eqref{eq:cltinterpolation}, we observe
\begin{align*}
	&\| \nablah U \|_{L^2}^{q-2} |\left< B(U, U), \Deltah U \right>| \\
	&\hspace{-0.3em}\quad \leq c \| \nablah U \|_{L^2}^{q-2} \| U \|_{L^\infty_z L^4_{xy}} \left( \| U \|_{L^2_z H^1_{xy}}^{1/2} \| U \|_{L^2_z H^2_{xy}}^{3/2} + \| U \|_{L^2_z H^1_{xy}}^{1/2} \| U \|_{H^1_z H^1_{xy}} \| U \|_{L^2_z H^2_{xy}}^{1/2} \right)\\
	&\hspace{-0.3em}\quad \leq \varepsilon \| U \|_{L^2_z H^2_{xy}}^2 \| \nablah U \|_{L^2}^{q-2} + c_\varepsilon \| U \|_{L^\infty_z L^4_{xy}}^4 \| U \|_{L^2_z H^1_{xy}}^2 \| \nablah U \|_{L^2}^{q-2}\\
	&\hspace{-0.3em}\quad \hphantom{\leq \ } + c_\varepsilon \| U \|_{L^\infty_z L^4_{xy}}^{4/3} \| U \|_{L^2_z H^1_{xy}}^{2/3} \| U \|_{H^1_z H^1_{xy}}^{4/3} \| \nablah U \|_{L^2}^{q-2}\\
	&\hspace{-0.3em}\quad \leq \varepsilon \| U \|_{L^2_z H^2_{xy}}^2 \| \nablah U \|_{L^2}^{q-2} + c_\varepsilon \| U \|_{H^1_z L^2_{xy}}^2 \| U \|_{L^2_z H^1_{xy}}^2 \| U \|_{L^2_z H^1_{xy}}^2 \| \nablah U \|_{L^2}^{q-2}\\
	&\hspace{-0.3em}\quad \hphantom{\leq \ } + c_\varepsilon \| U \|_{H^1_z L^2_{xy}}^{2/3} \| U \|_{L^2_z H^1_{xy}}^{4/3} \| U \|_{H^1_z H^1_{xy}}^{4/3} \| \nablah U \|_{L^2}^{q-2}\\
	&\hspace{-0.3em}\quad \leq \varepsilon \| \Deltah U \|_{L^2}^2 \| \nablah U \|_{L^2}^{q-2}\\
	&\hspace{-0.3em}\quad \hphantom{\leq \ } + c_\varepsilon \left( 1 + \| U \|_{L^2}^2 \right)\left( 1 + \| U \|_{H^1_z L^2_{xy}}^2 \right) \left( 1 + \| U \|_{L^2_z H^1_{xy}}^2 + \| U \|_{H^1_z H^1_{xy}}^2 \right)\\
	&\hspace{-0.3em}\quad  \hphantom{\leq + c_\varepsilon \ } \cdot \left( 1 + \| \nablah U \|_{L^2}^{q}\right).
\end{align*}
Accordingly, we deduce a suitable inequality for the stochastic Gronwall lemma \ref{lemma:stoch_Gronwall} for stopping times $0 \leq \tau_a \leq \tau_b \leq t \wedge \tau_N \wedge \Upsilon_K$ for $K, N \in \N$ and $\Upsilon_K = \tau_K^{w, 2} \wedge \tau_K^{\partial_z v, 2}$. The claim then follows similarly as in the above proofs.
\end{proof}

We remark that above Lemma can be established in an easier manner if one considers $\Upsilon_K = \tau_K^{w, 6} \wedge \tau_K^{\partial_z v, 6}$ instead of the stopping time $\Upsilon_K$ above.

The last auxiliary step is to bound $\|\partial_{zz}U\|_{L^2}$, which also follows similarly as the corresponding bound for the finite-dimensional approximations.

\begin{lemma}[$L^2$ bound for $\partial_{zz} U$]
\label{lemma:dzz_U_L2}
	Let $q\geq 2$, $U_0\in L^q(\Omega;H^2_zL^2)$ and $F\in L^q(\Omega;L^2_{loc}(0,\infty;H^2_zL^2)$.
	Then the stopping time $\tau_K^{\partial_{zz}U, q}$ defined for $K \in \N$ by
	\begin{align*}
		\tau_K^{\partial_{zz}U, q} = \inf \Bigg\lbrace s \geq 0 \mid & \sup_{r \in [0, s \wedge \xi)}\norm{\partial_{zz}U}_{L^2}^q  + \int_0^{s \wedge \xi} \norm{\partial_{zz}U}_{L^2}^{q-2} \norm{\nablah \partial_{zz}U}^2_{L^2} \dr\\
		&+\left( \int_{\tau_a}^{\tau_b}\|\nablah \partial_{zz}U\|_{L^2}^{2}\ds\right)^{q/2} \geq K \Bigg\rbrace
	\end{align*}
	satisfies $\tau_K^{\partial_{zz}U, q} \to \infty$ $\PP$-a.s.\ as $K \to \infty$.
\end{lemma}

\begin{proof}
	We can proceed as in Lemma \ref{lemma:galerkinestimatemoreregular} by replacing in the estimate for the $I^n_2$ the factor $\mu^4$ by $\| U \|_{H^1_zL^4_{xy}}^4$. Using
	\begin{align*}
		\| U \|_{H^1_zL^4_{xy}}^4\leq c\| U \|_{H^1_zL^2_{xy}}^2\| U \|_{H^1_zH^1_{xy}}^2
	\end{align*}
	we can argue as in the previous proof by taking stopping times $0 \leq \tau_a \leq \tau_b \leq t \wedge \tau_N \wedge \tau_K^{\partial_z v, q}$.
\end{proof}

\subsection{Proof of Theorem \ref{thm:globalExistence}}
\label{sect:globalSolutions_proof}

The proof closely follows the proof of \cite[Theorem 3.2]{DebusscheGlattholtzTemamZiane2012} and we include it mainly for completeness. Let $\rho_K$ be the stopping time (collected from Lemmata \ref{lemma:dz_v_L2}, \ref{lemma:dz_T_L2} and \ref{lemma:nablah_U_L2}) defined for $K \in \N$ by
\[
	\rho_K = \tau_K^{w, 2} \wedge \tau_K^{v, \infty} \wedge \tau_K^{\partial_z v, 2} \wedge \tau_K^{\partial_z v, 4}.
\]
Our aim is to show that $\rho_K \leq \xi$ for all $K \in \N$. Since $\rho_K \to \infty$ as $K \to \infty$ a.s.\ by the results above, this will yield global existence.

From the proofs of Lemmata \ref{lemma:dz_v_L2}, \ref{lemma:dz_T_L2} and \ref{lemma:nablah_U_L2} (see the final part of the proof of Lemma \ref{lemma:vtildevbar} for a detailed explanation), we get
\begin{equation}
	\label{eq:glob_ex_contradiction}
	\sup_{s \in [0, t \wedge \rho_K \wedge \xi)} \| U \|_{H^1}^2 + \int_{0}^{t \wedge \rho_K \wedge \xi} \| U \|_{L^2_z H^2_{xy}}^2 + \| U \|_{H^1_z H^1_{xy}}^2 \ds < \infty \quad \PP\text{-a.s.}
\end{equation}
For contradiction, let $\PP(\lbrace \rho_K > \xi \rbrace) > 0$ for some $K \in \N$. Then, since $\lbrace \rho_K > \xi \rbrace = \cup_{t \geq 0} \lbrace \rho_K \wedge t > \xi \rbrace$, there exists $t \geq 0$ such that $\PP(\lbrace \rho_K \wedge t > \xi \rbrace) > 0$. By Definition \ref{def:pathwise_solution} and \eqref{eq:blowup} in particular, this leads to
\begin{multline*}
	\sup_{s \in [0, t \wedge \rho_K \wedge \xi)} \| U \|_{H^1}^2 + \int_{0}^{t \wedge \rho_K \wedge \xi} \| U \|_{L^2_z H^2_{xy}}^2 + \| U \|_{H^1_z H^1_{xy}}^2 \ds\\
	\geq \sup_{s \in [0, t \wedge \xi)} \| U \|_{H^1}^2 + \int_{0}^{t \wedge \xi} \| U \|_{L^2_z H^2_{xy}}^2 + \| U \|_{H^1_z H^1_{xy}}^2 \ds = \infty
\end{multline*}
on a set of positive measure, which contradicts \eqref{eq:glob_ex_contradiction}. This concludes the proof of Theorem \ref{thm:globalExistence}.

\appendix

\section{Auxiliary results}

\begin{lemma}
\label{lemma:auxiliary}
Let $f: \Omega \times [0, \infty) \to [0, \infty]$ be such that $f(\cdot, \omega)$ is non-decreasing and continuous for a.a.\ $\omega \in \Omega$ and $f(t, \omega)$ is measurable and a.s.\ finite for all $t \geq 0$. For $K \in \N$, let $\rho_K = \inf \lbrace s \geq 0 \mid f(s, \omega) \geq K \rbrace$. Then $\rho_K \to \infty$ as $K \to \infty$ almost surely.

Moreover, if $\Phi: [0, \infty) \to [0, \infty)$ is a non-decreasing function such that $\Phi(s) \to \infty$ as $s \to \infty$, then $\rho_K^\Phi \to \infty$ as $K \to \infty$ almost surely, where $\rho_K^\Phi = \lbrace s \geq 0 \mid \Phi(f(s, \omega)) \geq K \rbrace$.
\end{lemma}

\begin{proof}
The first claim is established in e.g.\ \cite[Lemma 4.1]{Brzezniak2020}, see also \cite[Proposition A.1]{GlattholtzTemam2011}. To prove the second claim, we argue by contradiction. Assuming $\rho := \lim_K \rho^\Phi_K < \infty$ on a measurable set $\Omega_0 \subseteq \Omega$ with $\PP(\Omega_0) > 0$, we observe $f(\rho_K, \omega) \geq \Phi^{-1}(K)$. In particular, $f(\cdot, \omega)$ is unbounded on the (random) interval $[0, \rho]$ a.s.\ on $\Omega_0$. On the other hand, from $\rho_K \to \infty$, we deduce that $f(\cdot, \omega)$ is a.s.\ bounded on $[0, \rho]$, a contradiction.
\end{proof}

\begin{lemma}[stochastic Gronwall lemma, {\cite[Lemma 5.3]{GlattHoltz2009}}]
\label{lemma:stoch_Gronwall}
Let $t>0$ and let $X, Y, Z, R: [0, \infty) \times \Omega \to [0, \infty)$ be stochastic processes on a probability space $(\Omega, \cF, \PP)$. Let $\tau: \Omega \to [0, t)$ be a stopping time such that $\E \int_0^\tau RX + Z \ds < \infty$ and $\int_0^\tau R \ds < \kappa$ $\PP$-a.s.\ for some $\kappa > 0$. Assume that there exists a constant $c_0 > 0$ such that
\[
	\E \left[ \sup_{s \in [\tau_a, \tau_b]} X + \int_{\tau_a}^{\tau_b} Y \ds \right] \leq c_0 \E \left[ X(\tau_a) + \int_{\tau_a}^{\tau_b} RX + Z \ds \right]
\]
for all stopping times $\tau_a, \tau_b$ satisfying $0 \leq \tau_a \leq \tau_b \leq \tau$. Then
\[
	\E \left[ \sup_{s \in [0, \tau]} X + \int_0^\tau Y \ds \right] \leq c_{c_0, t, \kappa} \E \left[ X(0) + \int_0^\tau Z \ds \right].
\]
\end{lemma}

\end{document}